\documentclass[11pt,a4paper]{article}

\usepackage[margin=3cm]{geometry}
\usepackage{subcaption}
\usepackage{enumitem}

\usepackage[
  colorlinks,
  linkcolor = blue,
  citecolor = blue,
  urlcolor = blue]{hyperref}

\usepackage{amsmath,amssymb,amsthm}
\usepackage{mathtools}
\mathtoolsset{centercolon}

\usepackage[affil-it]{authblk}
\usepackage{tabu}

\newcommand{\C}{\mathbb{C}}
\newcommand{\Z}{\mathbb{Z}}

\newcommand{\ket}[1]{|#1\rangle}
\newcommand{\bra}[1]{\langle#1|}

\newcommand{\x}{\otimes}

\newcommand{\vphi}{\varphi}
\renewcommand{\phi}{\varphi}

\newcommand{\be}{\begin{equation}}
\newcommand{\ee}{\end{equation}}

\renewcommand{\vec}[1]{\mathbf{#1}}

\DeclareMathOperator{\Tr}{Tr}
\DeclareMathOperator{\tr}{tr}
\DeclareMathOperator{\spn}{span}

\DeclareMathOperator{\Cay}{Cay}

\newcommand{\Thm}[1]{\hyperref[thm:#1]{Theorem~\ref*{thm:#1}}}
\newcommand{\Lem}[1]{\hyperref[lem:#1]{Lemma~\ref*{lem:#1}}}
\newcommand{\Cor}[1]{\hyperref[cor:#1]{Corollary~\ref*{cor:#1}}}
\newcommand{\Def}[1]{\hyperref[def:#1]{Definition~\ref*{def:#1}}}
\newcommand{\Obs}[1]{\hyperref[obs:#1]{Observation~\ref*{obs:#1}}}
\newcommand{\Prop}[1]{\hyperref[prop:#1]{Proposition~\ref*{prop:#1}}}
\newcommand{\Rem}[1]{\hyperref[rem:#1]{Remark~\ref*{rem:#1}}}
\newcommand{\Ex}[1]{\hyperref[ex:#1]{Example~\ref*{ex:#1}}}

\newcommand{\Sec}[1]{\hyperref[sec:#1]{Section~\ref*{sec:#1}}}

\newcommand{\Fig}[1]{\hyperref[fig:#1]{Figure~\ref*{fig:#1}}}
\newcommand{\Tab}[1]{\hyperref[tab:#1]{Table~\ref*{tab:#1}}}
\newcommand{\EqRef}[1]{\hyperref[eq:#1]{(\ref*{eq:#1})}}
\newcommand{\Eq}[1]{Equation~\hyperref[eq:#1]{(\ref*{eq:#1})}}

\makeatletter
\newtheorem*{rep@theorem}{\rep@title}
\newcommand{\newreptheorem}[2]{%
\newenvironment{rep#1}[1]{%
 \def\rep@title{#2 \ref{##1}}%
 \begin{rep@theorem}}%
 {\end{rep@theorem}}}
\makeatother

\newreptheorem{theorem}{Theorem}
\newreptheorem{lemma}{Lemma}

\newtheorem{theorem}{Theorem}[section]
\newtheorem*{theorem*}{Theorem}
\newtheorem{lemma}[theorem]{Lemma}
\newtheorem{cor}[theorem]{Corollary}

\newtheorem{prob}[theorem]{Open problem}
\newtheorem*{prob*}{Open problem}

\theoremstyle{definition}
\newtheorem{definition}[theorem]{Definition}
\newtheorem{remark}[theorem]{Remark}
\newtheorem{example}[theorem]{Example}

\newcommand{\one}{\ensuremath{\vec{1}}}

\renewcommand\bra[1]{{\langle{#1}|}}
\renewcommand\ket[1]{{|{#1}\rangle}}

\newcommand{\inner}[2]{\left\langle #1 | #2 \right\rangle}
\newcommand{\outerp}[2]{\left| #1 \left\rangle \! \right\langle #2 \right|}

\newcommand{\A}{\ensuremath{\mathcal{A}}}
\newcommand{\B}{\ensuremath{\mathcal{B}}}

\DeclareMathOperator{\rk}{rk}

\DeclareMathOperator{\sym}{Sym}
\DeclareMathOperator{\bij}{Bij}
\DeclareMathOperator{\conv}{conv}

\newcommand{\sm}{\backslash}

\title{Group Invariant Quantum Latin Squares}

\author[1]{Arnbj\"{o}rg Soff\'{i}a  \'{A}rnad\'{o}ttir\footnote{arnbjorg.soffia@dcc.ufmg.br}}
\author[2]{David E.~Roberson\footnote{dero@dtu.dk}}

\affil[1]{Departamento de Ciência da Computação, Universidade Federal de Minas Gerais, Belo Horizonte -- MG, 31270-901, Brazil}
\affil[2]{Department of Applied Mathematics and Computer Science, Technical University of Denmark, DK-2800 Lyngby, Denmark}
\affil[2]{QMATH}

\date{February 28, 2025}

\begin{document}

\maketitle


\begin{abstract}
A quantum Latin square is an $n \times n$ array of unit vectors where each row and column forms an orthonormal basis of a fixed complex vector space. We introduce the notion of $(G,G')$-invariant quantum Latin squares for finite groups $G$ and $G'$. These are quantum Latin squares with rows and columns indexed by $G$ and $G'$ respectively such that the inner product of the $a,b$-entry with the $c,d$-entry depends only on $a^{-1}c \in G$ and $b^{-1}d \in G'$. This definition is motivated by the notion of group invariant bijective correlations introduced in [Roberson \& Schmidt (2020)], and every group invariant quantum Latin square produces a group invariant bijective correlation, though the converse does not hold. In this work we investigate these group invariant quantum Latin squares and their corresponding correlations. Our main result is that, up to applying a global isometry to every vector in a $(G,G')$-invariant quantum Latin square, there is a natural bijection between these objects and trace and conjugate transpose preserving isomorphisms between the group algebras of $G$ and $G'$. This in particular proves that a $(G,G')$-invariant quantum Latin square exists if and only if the multisets of degrees of irreducible representations are equal for $G$ and $G'$. Another motivation for this line of work is that whenever Cayley graphs for groups $G$ and $G'$ are quantum isomorphic, then there is a $(G,G')$-invariant quantum correlation witnessing this, and thus it suffices to consider such correlations when searching for quantum isomorphic Cayley graphs. Given a group invariant quantum correlation, we show how to construct all pairs of graphs for which it gives a quantum isomorphism.
\end{abstract}

\setcounter{tocdepth}{1}
\setlength{\parskip}{-0.15cm}
\tableofcontents
\setlength{\parskip}{0cm}

\section{Introduction}

A Latin square, of order $n$, is an $n \times n$ array of the numbers $1, 2, \ldots, n$ such that each row and column contains each number precisely once. These objects have been heavily studied in combinatorics and have connections to statistical designs~\cite{RABailey}, cryptography~\cite{pal2010design}, and algebraic graph theory~\cite{AGT}, and they are still an active area of research. The focus of this paper is on a quantum analog of Latin squares, known as \emph{quantum Latin squares}, introduced in~\cite{qlatin}. In a quantum Latin square, the numbers $1,2, \ldots, n$ are replaced by unit vectors, and the condition that every row/column contains each number exactly once is replaced by the condition that every row/column forms an orthonormal basis of some fixed vector space, usually (but not always) $\mathbb{C}^n$. It is straightforward to see that restricting to only using standard basis vectors, or any fixed orthonormal basis, results in a usual Latin square with numbers replaced by a fixed set of $n$ orthonormal vectors. Quantum Latin squares have connections to other primitives used in quantum information, such as unitary error bases~\cite{qlatin}, mutually unbiased bases~\cite{qlatin2}, and controlled families of Hadamards~\cite{biunitary}. Our main motivation in this work is their application as quantum strategies for the graph isomorphism game~\cite{qiso1}. Any quantum strategy for this game can be given as a \emph{quantum permutation matrix} whose entries are projections, and quantum Latin squares can be seen as a special case of quantum permutation matrices where every projection has rank one.

In~\cite{quantumvsnonlocal}, Roberson and Schmidt introduced a highly symmetric type of strategy for the graph isomorphism game. These so-called \emph{group invariant} strategies are invariant under the regular action of a given group $G$ on the players' inputs and outputs (this is made precise in Section~\ref{sec:prelims}). For a strategy arising from a quantum Latin square, this amounts to an invariance of the square of the modulus of the inner products of its entries under independent regular actions on its row and column indices. Currently, a complete characterization of such quantum Latin squares seems out of reach. Instead we strengthen this invariance condition to apply to the inner products themselves, rather than the squares of their moduli, and say that such quantum Latin squares are \emph{group invariant}. We also consider the case where the row and column indices are acted on by different groups, say $G$ and $G'$ respectively, and refer to such Latin squares as being $(G,G')$-invariant.


In general, there are $n^4$ inner products among the $n^2$ vectors in an $n \times n$ quantum Latin square. However, due to the symmetry of group invariant quantum Latin squares, they have at most $n^2$ distinct inner products, and these can be stored in an $n \times n$ matrix that we term the \emph{transformation matrix}. The transformation matrix completely determines a $(G,G')$-invariant quantum Latin square up to a global isometry, and we are able to completely characterize these matrices in terms of simple conditions based on the groups $G$ and $G'$ (see Theorem \ref{thm:qlstransf}). 
This allows us to prove our main result, that the maps consisting of conjugating by a transformation matrix are precisely the trace-preserving isomorphisms of the left regular representations of $G$ and $G'$ that commute with conjugate transpose. This implies that a $(G,G')$-invariant quantum Latin square exists if and only if the multiset of degrees of irreducible representations of $G$ is the same as that of $G'$, i.e., if and only if the group algebras of $G$ and $G'$ are isomorphic. In particular, if $G$ and $G'$ are abelian groups of the same order, such a quantum Latin square always exists and in this case there are finitely many. If, however, a $(G,G')$-invariant quantum Latin square exists for non-abelian groups $G$ and $G'$, then there are uncountably many of them.

We also investigate the strategies for the isomorphism game arising from group invariant quantum Latin squares. This amounts to considering the squared moduli of the inner products, which give the so-called \emph{correlation probabilities} of a quantum strategy. Here we are focused on whether or not the correlation arising from a group invariant quantum Latin square is \emph{non-classical}. Though it is not difficult to produce such examples, a full understanding of when this occurs eludes us, and so this remains an intriguing open problem. In the abelian case, we prove a surprising result: the convex hull of the correlations arising from $(G,G')$-invariant quantum Latin squares is a unitary transformation of the set of \emph{classical} $(G,G')$-invariant correlations. For $G = G' = \mathbb{Z}_2^d$ where $d\in\{1,2,3\}$, the two sets are in fact equal, but we provide examples for $d=4$ that show this pattern does not continue, answering a question from~\cite{quantumvsnonlocal}.

In~\cite{quantumvsnonlocal}, a construction of $(G,G)$-invariant correlations for abelian $G$ was introduced, and which has since been used to prove results in quantum group theory~\cite{mccarthy2023tracing} and mathematical physics~\cite{de2023magic}. We show that this construction in fact produces all $(G,G)$-invariant quantum Latin squares, and its natural generalization produces all $(G,G')$-invariant quantum Latin squares for abelian $G$ and $G'$ (see Remark~\ref{rem:all}).

The remainder of the paper is outlined as follows: Section~\ref{sec:prelims} introduces the graph isomorphism game and its classical and quantum strategies and their corresponding correlations. This serves as both a preliminaries and motivation section. In Section~\ref{sec:GInvCorr} we introduce the notion of group invariant correlations, generalizing the definition from~\cite{quantumvsnonlocal} to the case of two different groups, and we provide descriptions of the set of classical group invariant correlations. We also connect these correlations to the isomorphism game on Cayley graphs.  In Section~\ref{sec:QLS} we introduce quantum Latin squares and in particular define what it means for them to be group invariant. We also show that the inner products of vectors in a group invariant quantum Latin square can be used to construct its \emph{transformation matrix} which determines the quantum Latin square up to isometry. Moreover, we characterize precisely which matrices are the transformation matrices of some $(G,G')$-invariant quantum Latin square. In Section~\ref{sec:composition} we show that, unlike general quantum Latin squares, group invariant quantum Latin squares can be composed to produce new quantum Latin squares, which are moreover group invariant, and this composition corresponds to multiplication of the respective transformation matrices. 

In Section~\ref{sec:reptheory} we delve into some representation theory that we will need for our main result and also investigate the relationship between transformation matrices of $(G,G')$-invariant quantum Latin squares and what we refer to as \emph{quasi-regular representations} of $G$ and $G'$. We prove our main result in Section~\ref{sec:transexist}, namely that maps given by conjugating by a transformation matrix of a $(G,G')$-invariant quantum Latin square are precisely the trace and conjugate transpose preserving isomorphisms of the group algebras of $G$ and $G'$. Section~\ref{sec:constructing} shows how to actually construct $(G,G')$-transformation matrices, assuming one has unitaries that block-diagonalize the group algebras. In the abelian case this is achieved by the character tables of the groups, and thus we can construct all $(G,G')$-transformation matrices since there are finitely many in this case. We study the abelian case further in Section~\ref{sec:abelian} and show that in this case the convex hull of the correlations produced by $(G,G')$-invariant quantum Latin squares is the image of the set of classical $(G,G')$-invariant correlations under a particular unitary map. In Section~\ref{sec:isomorphisms} we show that the submatrices of $(G,G')$-transformation matrices corresponding to the rows/columns with a single nonzero entry give isomorphisms between subgroups of $G$ and $G'$, as well as the linear characters of these groups. Additionally, we prove a partial converse assuming a certain product structure of the groups $G$ and $G'$. In Section \ref{sec:supportgraphs} we define the \emph{support graph} of a correlation and consider in particular support graphs of group invariant correlations. Section \ref{sec:computations} presents some computational results and important examples, e.g., an example of a non-classical correlation produced by a $\mathbb{Z}_2^4$-invariant quantum Latin square, answering a question of~\cite{quantumvsnonlocal}. Finally, we end with a discussion of our results and some open problems.

\section{The Isomorphism Game}\label{sec:prelims}


Though developed independently, our perspective on quantum Latin squares is mainly influenced by their relation to the graph isomorphism game. In this section, we introduce this concept and talk about its connection to our work. We also give some preliminaries.

\begin{definition}
    Given graphs $X$ and $Y$ with $|V(X)| = |V(Y)|$\footnote{This condition can be dropped if one instead considers a slightly more complicated version of the isomorphism game, for which there exist winning strategies only if this condition holds. However, the version defined here allows for a much cleaner presentation of our results, and is equivalent to the more complicated game for the classes of strategies/correlations we consider.}, the \emph{$(X,Y)$-isomorphism game} consists of two players (Alice and Bob) attempting to convince a referee/verifier that they know an isomorphism from $Y$ to $X$\footnote{The reason for the seemingly backwards direction of this definition is that our $(G,G')$-invariant quantum Latin squares have rows indexed by $G$ and columns indexed by $G'$. Thus so do their associated transformation matrices, which we then think of as mapping $\mathbb{C}^{G'}$ to $\mathbb{C}^{G}$.}. Each player is given a (possibly different) vertex of $Y$, and must respond with a vertex of $X$. They win if their outputs satisfy the same relation (i.e., equal, adjacent, distinct non-adjacent) as their inputs. We say that a given strategy for the game is a \emph{winning strategy} if the players are guaranteed to win regardless of the particular inputs they are given. 
    Importantly, the players are not aware of each other's input. 
\end{definition}

\begin{remark}\label{rem:digraphs}
    Though the isomorphism game is usually only discussed in the setting of undirected graphs, as in the definition above, it can easily be extended to the setting of directed graphs. In this case, there must be an arc from the vertex Alice responded with to the vertex Bob responded with if and only if there was an arc from the vertex Alice received to the vertex Bob received. We can also allow loops, which will impose the condition that each player responds with a vertex with a loop if and only if their input vertex had a loop. We will mainly focus on undirected loopless graphs, but we will allow the more general case in Section~\ref{subsec:cayley}.
\end{remark}

A deterministic strategy for the isomorphism game consists of functions $f_A, f_B: V(Y) \to V(X)$ such that Alice (resp.~Bob) responds with $f_A(y)$ (resp.~$f_B(y'))$ on input $y$ (resp.~$y'$). It is straightforward to see that such a strategy is winning if and only if $f_A = f_B$ and this is an isomorphism from $Y$ to $X$. More generally, a \emph{classical strategy} allows for the players' answers to additionally depend on a shared source of randomness. This allows the players to probabilistically choose from some set of deterministic strategies, and thus a winning classical strategy exists if and only if a winning deterministic strategy exists, i.e., if the graphs $X$ and $Y$ are isomorphic.

In general, any strategy for the isomorphism game leads to a \emph{correlation} $p$ where $p(x,x'|y,y')$ is the probability that Alice and Bob answer $x,x'$ conditioned on their receiving $y,y'$ as respective inputs. Note that by definition, one always has that 
\[\sum_{x,x'\in V(X)} p(x,x'|y,y') = 1.\] The strategy is a winning strategy if and only if $p(x,x'|y,y') = 0$ when $x,x'$ do not satisfy the same relation as $y,y'$, i.e., if the probability of them answering incorrectly is 0. For a deterministic strategy given by $f_A,f_B$, the corresponding correlation is
\[p(x,x'|y,y') = \begin{cases}
    1 & \text{if } x = f_A(y) \ \& \ x' = f_B(y') \\
    0 & \text{otherwise.}
\end{cases}\]
Then the correlations of general classical (winning) strategies are precisely the convex combinations of correlations arising from deterministic (winning) strategies. These correlations are said to be \emph{classical} or \emph{local} correlations. 

Quantum strategies for the isomorphism game are more complicated. Shortly, Alice must select measurements $\{E_{xy} \in M_{d_A}(\mathbb{C}): x \in V(X)\}$ 
of positive semidefinite matrices that sum to identity for each $y \in V(Y)$, Bob must similarly select measurements $\{F_{xy} \in M_{d_B}(\mathbb{C}): x \in V(X)\}$ for each $y \in V(Y)$, and they must choose a shared quantum state (i.e., a unit vector) $\ket{\psi} \in \mathbb{C}^{d_A} \otimes \mathbb{C}^{d_B}$. The correlation arising from such a strategy is given by
\begin{equation}\label{eq:qcorrdef}
    p(x,x'|y,y') = \bra{\psi}\left(E_{xy} \otimes F_{x'y'}\right)\ket{\psi} \text{ for all } x, x' \in V(X), \ y, y' \in V(Y).
\end{equation}

When a winning quantum strategy for the $(X,Y)$-isomorphism game exists, the graphs $X$ and $Y$ are said to be \emph{quantum isomorphic}\footnote{Strictly speaking we should specify ``quantum tensor" strategy/isomorphism, as opposed to the more general ``quantum commuting" strategy/isomorphism. But we will always use ``quantum strategy/isomorphism" to refer to ``quantum tensor strategy/isomorphism" in this work. This matches the terminology of~\cite{qiso1}.}, which is denoted by $X \cong_q Y$.

As shown in~\cite{qiso1}, the conditions of the isomorphism game impose certain conditions/relations on the measurement operators of Alice and Bob and their shared state. This allows one to formulate quantum isomorphism in a more succinct manner, but requires the following definition:


\begin{definition}
    A matrix $\mathcal{P} = (p_{ij}) \in M_n(M_d(\mathbb{C}))$ is a \emph{quantum permutation matrix}\footnote{Generally, quantum permutation matrices are allowed to have entries from an arbitrary $C^*$-algebra. However, for the purposes of this paper, we will restrict to the finite dimensional case.} if its entries $p_{ij} \in M_d(\mathbb{C})$ satisfy the following:
    \begin{enumerate}
        \item $p_{ij} = p_{ij}^2 = p_{ij}^\dagger$, i.e., they are orthogonal projections, and
        \item $\sum_{k=1}^n p_{ik} = I = \sum_{\ell=1}^n p_{\ell j}$ for all $i,j \in [n]$.
    \end{enumerate}
    We remark that, assuming (1), condition (2) is equivalent to the matrix $\mathcal{P}$ being unitary.
\end{definition}

We then have the following:

\begin{theorem}[\cite{qiso1}]\label{thm:qiso1}
    Let $X$ and $Y$ be graphs. Then $X \cong_q Y$ if and only if there exists a quantum permutation matrix $\mathcal{P}$ such that 
    \[A_X \mathcal{P} = \mathcal{P}A_Y,\]
    where $A_X$ and $A_Y$ are the adjacency matrices of $X$ and $Y$ respectively.
\end{theorem}

It follows from the proof of Theorem~5.4 in~\cite{qiso1} that a quantum permutation matrix $\mathcal{P} = (p_{xy})_{x \in V(X), y \in V(Y)}$ satisfying the conditions of the above theorem results in a correlation $p$ that wins the $(X,Y)$-isomorphism game and is given by
\begin{equation}\label{eq:qcorr}
    p(x,x'|y,y') = \tr(p_{xy}p_{x'y'}),
\end{equation}
where $\tr$ denotes the normalized trace, i.e., $\tr(M) = \Tr(M)/d$ for $M \in M_d(\mathbb{C})$. It is well known that if all of the entries of a quantum permutation matrix pairwise commute, then the resulting correlation is classical. The fact that the converse is not true is the focus of~\cite{quantumvsnonlocal}.

In some cases, e.g.~Section~\ref{sec:supportgraphs}, it may be useful to allow our quantum permutation matrices to be over some self-adjoint subalgebra $\A \subseteq M_n(\mathbb{C})$. By this we mean that the entries of the quantum permutation matrix all live in $\A$ and the sum of any row or column is the identity of $\A$, which is some projection $q$ in $M_n(\mathbb{C})$. In this case, the normalized trace is defined as the usual trace of $M_n(\mathbb{C})$ divided by the trace of $q$, and the formula~\eqref{eq:qcorr} still produces a quantum correlation.

\begin{remark}\label{rem:qconvex}
It is likely, but unknown, that not every winning quantum correlation for the $(X,Y)$-isomorphism game arises from a quantum permutation matrix via Equation~\eqref{eq:qcorr}. 
However, it follows from known results in the literature, see e.g.~\cite{lauramaxent}, 
that every winning quantum correlation for the $(X,Y)$-isomorphism game is a convex combination of correlations of the form~\eqref{eq:qcorr}. Importantly, the sets of classical and quantum correlations are convex with the latter containing the former. Lastly, we note that it is well-known, and easy to see from~\eqref{eq:qcorr}, that all quantum correlations satisfy the so-called \emph{non-signalling} condition, i.e., that Alice's marginals $p_A(x|y) := \sum_{x'} p(x,x'|y,y')$ do not depend on Bob's input $y'$ and Bob's marginals $p_B(x'|y') := \sum_{x} p(x,x'|y,y')$ do not depend on Alice's input $y$. This condition enforces that no information is transmitted between the two parties.
\end{remark}

Notably, it is possible for graphs $X$ and $Y$ to be quantum isomorphic but not isomorphic~\cite{qiso1}. In this case, the correlation of any winning quantum strategy for the $(X,Y)$-isomorphism game is necessarily non-classical.  
However, the converse does not hold: there exist non-classical quantum correlations that win the $(X,Y)$-isomorphism game, but only for isomorphic $X$ and $Y$, e.g., the correlation used in~\cite[Theorem 3.19]{quantumvsnonlocal} to show that $K_5$ has ``nonlocal symmetry". Note that if $p$ is a winning correlation for the $(X,Y)$-isomorphism game, then it is also a winning correlation for the isomorphism game where $X$ and $Y$ are replaced by empty graphs on the same vertex set. Since we are mainly focused on group invariant quantum Latin squares and the correlations they produce, this least restrictive isomorphism game is the one we will consider. Since for empty graphs, two vertices can only be equal or distinct non-adjacent, in this case it makes sense to simply think of the isomorphism game as a \emph{bijection game} between two sets, in which the players must give the same answer if and only if they are given the same input. For us, the two sets will always be two groups $G$ and $G'$. As with the isomorphism game, we restrict to the case where $|G| = |G'|$, but remark that this assumption can be removed as in the isomorphism game case by passing to a larger game. It is then clear that the winning deterministic strategies for the $(G,G')$-bijection game correspond precisely to the bijections between $G$ and $G'$. Thus there is always a winning classical strategy, but we are interested in the winning quantum strategies that produce non-classical correlations. However, in Section~\ref{subsec:cayley} we will show how to take any winning group invariant correlation for the bijection game and determine all the graph isomorphism games for which it is also a winning strategy. Note that any quantum permutation matrix with rows indexed by $G$ and columns by $G'$ gives a winning quantum correlation for the $(G,G')$-bijection game, since the $A_X \mathcal{P} = \mathcal{P} A_Y$ condition of Theorem~\ref{thm:qiso1} is trivial for empty graphs.

\section{Group Invariant Correlations}\label{sec:GInvCorr}

Here we introduce group invariant correlations, in particular we investigate the set of such correlations that are classical. We also describe the connection between group invariant correlations and the isomorphism game on Cayley graphs. 

\begin{definition}
    Let $G$ and $G'$ be finite groups with $|G|=|G'|$. A \emph{$(G,G')$-invariant correlation} $p$ is a winning correlation for the $(G,G')$-bijection game such that the value of $p(b,d|a,c)$ depends only on the values of the quotients $b^{-1}d \in G$ and $a^{-1}c \in G'$, and additionally $\sum_{a,c \in G'} p(b,d|a,c) = 1$ for all  $b,d\in G$.
\end{definition}

\begin{remark}
    The constraint $\sum_{a,c \in G'} p(b,d|a,c) = 1$ in the definition above is to ensure that the correlation matrix of a group invariant correlation (see Definition~\ref{def:corrmat}) is doubly stochastic. Importantly, any winning quantum correlation for the $(G,G')$-bijection game must necessarily satisfy this condition as can be easily checked using Equation~\eqref{eq:qcorr}. Also note that the analogous sum over $G$ is necessarily equal to 1 for any correlation $p$. 
\end{remark}

For convenience, we will use $\hat{B}(G,G')$ to denote the set of all winning correlations for the $(G,G')$-bijection game that satisfy $\sum_{a,c \in G'} p(b,d|a,c) = 1$ for all $b,d \in G$.


Group invariant correlations were first introduced in~\cite{quantumvsnonlocal}, but there they were only defined for a single group, i.e., they defined the case where $G=G'$ in the definition above. In fact, in~\cite{quantumvsnonlocal} they also required that $p(b,d|a,c) = p(d,b|c,a)$ for all $b,d \in G$ and $a,c \in G'$, but we do not need this condition. In the case where $G=G'$, we will follow the convention of~\cite{quantumvsnonlocal} and refer to these correlations as $G$-invariant, as opposed to $(G,G)$-invariant.

In~\cite{quantumvsnonlocal}, they introduce an important specific $G$-invariant correlation, denoted $p_G$, which acts as a projection onto the set of $G$-invariant correlations. This correlation is defined as

\begin{equation}\label{eq:identity}
    p_G(b,d|a,c) = \begin{cases}
        \frac{1}{|G|} & \text{if } b^{-1}d = a^{-1}c \\
        0 & \text{otherwise}
    \end{cases}
\end{equation}

Importantly, as noted in~\cite{quantumvsnonlocal}, the correlation $p_G$ is always classical as it is a uniform convex combination (over $a \in G$) of the correlations arising from the deterministic strategies where both parties answer with $ax$ upon input $x \in G$. In order to understand how $p_G$ acts as a projection onto the set of $G$-invariant correlations, we must first introduce the composition of correlations.

\begin{definition}
    Given (winning) correlations $p_1$ and $p_2$ for the $(G,G')$-bijection game and the $(G',G'')$-bijection game respectively, their \emph{composition}, denoted $p_1 \circ p_2$, is the (winning) correlation for the $(G,G'')$-bijection game defined as
    \begin{equation}
        p_1 \circ p_2(b,d|a,c) = \sum_{x,y \in G'} p_1(b,d|x,y)p_2(x,y|a,c) \text{ for all } b,d \in G \text{ and } a,c \in G''.
    \end{equation}
    We remark that it is easy to verify that $p_1 \in \hat{B}(G,G')$, $p_2 \in \hat{B}(G',G'')$ implies that $p_1 \circ p_2 \in \hat{B}(G,G'')$.
\end{definition}

It is well known and not difficult to prove that composition of correlations preserves the property of being classical/quantum. As remarked in~\cite{quantumvsnonlocal}, we have that $p_G \circ p_G = p_G$. Moreover, they proved the following lemma in the case where  $G'=G$. The proof for the general case is not sufficiently different, so we omit it.

\begin{lemma}\label{lem:projection}
    Let $G$ and $G'$ be equicardinal finite groups, and let $p \in \hat{B}(G,G')$. Then $p$ is $(G,G')$-invariant if and only if $p = p_G \circ p \circ p_{G'}$. As a consequence, $p_G \circ p \circ p_{G'}$ is $(G,G')$-invariant for any $p \in \hat{B}(G,G')$.
\end{lemma}

The above lemma also leads to one of the motivations for studying group invariant correlations. As noted in~\cite{quantumvsnonlocal}, if $X$ and $Y$ are Cayley graphs for $G$ and $G'$ respectively and $p$ is a winning correlation for the $(X,Y)$-isomorphism game, then $p_G \circ p \circ p_{G'}$ is a winning correlation for the $(X,Y)$-isomorphism game that is $(G,G')$-invariant. Thus if one is searching for quantum isomorphisms of Cayley graphs, it suffices to search for ones that are group invariant. 

\begin{definition}\label{def:corrmat}
    If $p$ is a $(G,G')$-invariant correlation, then we define its \emph{correlation matrix} $D^p$ as the $G \times G'$ matrix with
    \[D^p_{x,y} = |G|p(b,d|a,c) \text{ where } b^{-1}d = x \ \& \ a^{-1}c = y.\]
\end{definition}

The following Lemma collects several basic facts about correlation matrices of group invariant correlations. All of these were proven in~\cite{quantumvsnonlocal} in the $G=G'$ case, and we omit the analogous proofs here. Note that we use $e$ to denote the identity of a group.

\begin{lemma}\label{lem:basic}
    Let $G$, $G'$, and $G''$ be equicardinal finite groups. Let $p_1$ and $p_2$ be $(G,G')$-invariant and $(G',G'')$-invariant correlations respectively. Then we have the following:
    \begin{enumerate}
        \item $D^{p_1}$ is doubly stochastic with $D^{p_1}_{e,e} = 1$. 
        \item $D^{p_G} = I$.
        \item $p_1 \circ p_2$ is $(G,G'')$-invariant and $D^{p_1 \circ p_2} = D^{p_1}D^{p_2}$.
        \item If $p_1$ is a quantum correlation, then $D^{p_1}_{x,y} = D^{p_1}_{x^{-1},y^{-1}}$.
    \end{enumerate}
\end{lemma}

Note that by definition, a $(G,G')$-invariant correlation is completely determined by its correlation matrix. Additionally, in converse to item~(1) above, it is easy to see that if $D$ is a $G \times G'$ doubly stochastic matrix with $D_{e,e} = 1$, then the correlation $p$ defined as
\[p(b,d|a,c) = \frac{1}{|G|}D_{x,y} \text{ where } x = b^{-1}d \text{ and } y = a^{-1}c\]
is $(G,G')$-invariant. It is also worth remarking that this shows that all group invariant correlations satisfy the no-signalling condition, i.e, that $\sum_{d \in G} p(b,d|a,c)$ is independent of $c \in G'$ and $\sum_{b \in G} p(b,d|a,c)$ is independent of $a \in G'$. Indeed, all of these so-called \emph{marginal probabilities} are equal to $1/|G|$.

\subsection{Classical group invariant correlations}\label{subsec:classicalGICs}

It follows from Lemma~\ref{lem:projection} and the convexity of the set of classical winning correlations for the $(G,G')$-bijection game that the set of all classical $(G,G')$-invariant correlations is convex. Moreover, Lemma~\ref{lem:projection} allows us to compute all of the extreme points of this set. Indeed, letting $\bij(G,G')$ be the set of bijections from $G'$ to $G$ and defining the correlation 
\[
p_\pi(b,d|a,c) = \begin{cases}
    1 & \text{if } b = \pi(a) \text{ and } d = \pi(c) \\ 0 & \text{otherwise}
\end{cases}
\]
for all $\pi \in \bij(G,G')$, it is immediate from Lemma~\ref{lem:projection} that every extreme point of the set of classical $(G,G')$-invariant correlations is contained in the set $\{p_G \circ p_\pi \circ p_{G'} : \pi \in \bij(G,G')\}$. Note however that it is not necessarily the case that every element of this set is an extreme point, and in fact this is not the case for some groups (see Section~\ref{subsec:Z6comps} for an example with $G = G' = \mathbb{Z}_6$). Determining which points are extreme can of course be done via linear programming, though the size of the linear program will be exponential in $|G|$ due to the size of $\bij(G,G')$. Another point worth mentioning is that it follows from the above that there are always \emph{classical} $(G,G')$-invariant correlations whenever $|G| = |G'|$. Motivated by our observations here, we define the following.

\begin{definition}
Given a bijection $\pi$ from $G'$ to $G$, define the correlation $\hat{p}_\pi = p_G \circ p_\pi \circ p_{G'}$. Then let $D^\pi$ be the correlation matrix of $\hat{p}_\pi$. Further, define the $G \times G'$ permutation matrix $P^\pi$ as
\[P^\pi_{x,y} = \begin{cases}1 & \text{if } x = \pi(y) \\ 0 & \text{otherwise} \end{cases}\]
Lastly, for any $a \in G$, define the bijection $\tau_a \in \bij(G) := \bij(G,G) = \sym(G)
$ by $\tau_a(x) = ax$ and define $P^a := P^{\tau_a}$ (here, $\sym(G)$ denotes symmetric group on the set $G$).
\end{definition}


Note that $P^\pi P^{\pi'} = P^{\pi \circ \pi'}$ when these products make sense. The above notation allows us to express the correlation matrices $D^\pi$ as the average of the permutation matrices of the elements of $\{\tau_{\pi(g)^{-1}} \circ \pi \circ \tau_g : g \in G'\}$:

\begin{lemma}\label{lem:avg}
    Let $\pi \in \bij(G,G')$, then
    \[D^\pi = \frac{1}{|G|} \sum_{g \in G'} P^{\pi(g)^{-1}} P^\pi P^g.\]
\end{lemma}
\begin{proof}
    Let $n = |G| = |G'|$. Let $s \in G$, $t \in G'$ and choose any $b,d \in G$ and $a,c \in G'$ such that $b^{-1}d = s$ and $a^{-1}c = t$ . Then by definition
    \begin{align*}
        D^\pi_{s,t} &= n p_G \circ p_\pi \circ p_{G'} (b,d|a,c) \\
        &= n \sum_{\substack{w,x \in G\\ \ y,z \in G'}}p_{G}(b,d|w,x)p_\pi(w,x|y,z)p_{G'}(y,z|a,c) \\
        &= n \sum_{y,z \in G'}p_{G}(b,d|\pi(y),\pi(z))p_{G'}(y,z|a,c) \\
        &= \frac{1}{n} \left|\left\{y,z \in G' : y^{-1}z = t \ \& \ \pi(y)^{-1}\pi(z) = s\right\}\right| \\
        &= \frac{1}{n} \left|\left\{y \in G' : \pi(y)^{-1}\pi(yt) = s\right\}\right|.
    \end{align*}
    On the other hand,
    \begin{align*}
        \left(\frac{1}{n} \sum_{g \in G'} P^{\pi(g)^{-1}} P^\pi P^g\right)_{s,t} &= \frac{1}{n} \sum_{x \in G, \ g,y \in G'} P^{\pi(g)^{-1}}_{s,x} P^\pi_{x,y} P^g_{y,t} \\
        &= \frac{1}{n} \sum_{g,y \in G'} P^{\pi(g)^{-1}}_{s,\pi(y)}P^g_{y,t} \\
        &= \frac{1}{n} |\{g \in G' : \pi(g)^{-1}\pi(gt) = s\}|,
    \end{align*}
    which is the same as the above.
\end{proof}

As a corollary we obtain the following.

\begin{cor}\label{cor:avg}
    Let $\pi \in \bij(G, G')$. 
    Then $D^{\tau_a \circ \pi \circ \tau_b} = D^\pi$ for all $a \in G, b \in G'$.
\end{cor}
\begin{proof}
    We have that
    \begin{align*}
        D^{\tau_a \circ \pi \circ \tau_b} &= \frac{1}{|G|} \sum_{g \in G'} P^{(a\pi(bg))^{-1}} P^{\tau_a \circ \pi \circ \tau_b} P^g \\
        &= \frac{1}{|G|} \sum_{g \in G'} P^{(a\pi(bg))^{-1}} P^a P^{\pi} P^b P^{g} \\
        &= \frac{1}{|G|} \sum_{g \in G'} P^{(a\pi(bg))^{-1}a} P^{\pi} P^{bg} \\
        &= \frac{1}{|G|} \sum_{g \in G'} P^{\pi(bg)^{-1}} P^{\pi} P^{bg} \\
        &= \frac{1}{|G|} \sum_{g \in G'} P^{(\pi(g))^{-1}} P^{\pi} P^{g} \\
        &= D^\pi.\qedhere
    \end{align*}
\end{proof}

\begin{remark}\label{rem:doublecosets}
    Note that the relation $\sim$ on $\bij(G,G')$ defined as $\pi_1 \sim \pi_2$ if there exists $a \in G$ and $b \in G'$ such that $\pi_2 = \tau_a \circ \pi_1 \circ \tau_b$ is an equivalence relation. Thus the above states that $D^\pi$ depends only on the particular equivalence class $\{\tau_a \circ \pi' \circ \tau_b : a \in G, \ b \in G'\}$ that $\pi$ is contained in. This reduces the set of possible extreme points, but only by a factor of approximately $|G|^2$. We can also use these equivalence classes to reinterpret Lemma~\ref{lem:avg}: it says that $D^\pi$ is the average of the matrices $\{P^{\pi'} : \pi' \sim \pi, \ \pi'(e) = e\}$. We also note that in the case $G = G'$, these equivalence classes are the double cosets of the subgroup $\{\tau_a : a \in G\}$ of $\bij(G)$.
\end{remark}


Next we will show that the permutation matrices contained in $\{D^\pi : \pi \in \bij(G,G')\}$ are precisely the isomorphisms of $G$ and $G'$.

\begin{lemma}\label{lem:classperms}
    Let $G$ and $G'$ be equicardinal finite groups. If $\pi \in \bij(G,G')$, then $D^\pi$ is a permutation matrix if and only if $\pi = \tau_a \circ \sigma \circ \tau_b$ for some $a \in G$, $b \in G'$, and some isomorphism $\sigma$ from $G'$ to $G$. In this case, $D^\pi = P^{\sigma}$. As a consequence, the only permutation matrices that are the correlation matrix of a classical $(G,G')$-invariant correlation are the matrices $P^\sigma$ where $\sigma$ is an isomorphism from $G'$ to $G$.
\end{lemma}
\begin{proof}
    Let $\pi \in \bij(G,G')$. By Lemma~\ref{lem:avg}, $D^\pi$ is a permutation matrix if and only if $P^{\pi(g)^{-1}}P^\pi P^g$ is constant for all $g \in G'$. A simple calculation shows that
    \begin{equation}\label{eq:ppp}
        (P^{\pi(g)^{-1}}P^\pi P^g)_{s,t} = \begin{cases}
            1 & \text{if } \pi(gt) = \pi(g)s \\
            0 & \text{otherwise}
            \end{cases}
    \end{equation}
    If $\pi$ is an isomorphism from $G'$ to $G$, then the above shows that $(P^{\pi(g)^{-1}}P^\pi P^g)_{s,t}$ is 1 if and only if $\pi(t)=s$ and is 0 otherwise. Thus if $\pi$ is an isomorphism, then $(P^{\pi(g)^{-1}}P^\pi P^g) = P^\pi$ as desired. Additionally, by Corollary~\ref{cor:avg}, if $\pi = \tau_a \circ \sigma \circ \tau_b$ for some $a \in G$, $b \in G'$, and some isomorphism $\sigma$ from $G'$ to $G$, then $D^\pi = D^\sigma$ the latter of which is equal to $P^\sigma$ by the above. Therefore we have proven the second claim of the lemma, as well as the if direction of the first statement.

    Now suppose that the set $\{\tau_a \circ \pi \circ \tau_b : a \in G, \ b \in G'\}$ does not contain an isomorphism. By Corollary~\ref{cor:avg}, every element in this set results in the same correlation matrix. Therefore, we may assume that $\pi(e) = e$. Further, since $\pi$ is not an isomorphism, there exists $\hat{g}, \hat{t} \in G'$ such that $\pi(\hat{g}\hat{t}) \ne \pi(\hat{g})\pi(\hat{t})$. Considering these particular values for Equation~\eqref{eq:ppp}, we see that $(P^{\pi(\hat{g})^{-1}}P^\pi P^{\hat{g}})_{\pi(\hat{t}),\hat{t}} = 0$. On the other hand,
    \[(P^{\pi(e)^{-1}}P^\pi P^e)_{\pi(\hat{t}),\hat{t}} = P^\pi_{\pi(\hat{t}),\hat{t}} = 1.\]
    Therefore, the matrix $P^{\pi(g)^{-1}}P^\pi P^g$ is not constant for all $g \in G'$ and thus $D^\pi$ is not a permutation matrix. So we have proven the only if the direction of the first claim.

    Lastly, since the set of correlation matrices of classical $(G,G')$-invariant correlations is equal to the convex hull of $\{D^\pi : \pi \in \bij(G,G')\}$, any permutation matarix in this convex hull must in fact be equal to one of the $D^\pi$. By the above, this proves the final claim of the lemma.
\end{proof}

Before moving on, we will prove one result concerning correlation matrices of quantum group invariant correlations. The following extends the final claim of the above lemma to this set.

\begin{cor}\label{cor:qperms}
    Let $G$ and $G'$ be equicardinal finite groups. If $P$ is a permutation matrix which is the correlation matrix of a quantum $(G,G')$-invariant correlation, then $P = P^\pi$ for some isomorphism $\pi$ from $G'$ to $G$.
\end{cor}
\begin{proof}
    Let $p$ be the quantum $(G,G')$-invariant correlation whose correlation matrix is $P$. By Remark~\ref{rem:qconvex}, the correlation $p$ is the convex combination of quantum correlations $p^i$ of the form
    \[p^i(b,d|a,c) = \tr(q^i_{a,b}q^i_{c,d})\]
    for some quantum permutation matrices $\mathcal{Q}^i = (q^i_{x,y})$. Let $\pi \in \bij(G,G')$ be the permutation such that $P = P^\pi$. By the above, we have that $\tr(q^i_{a,b}q^i_{c,d}) = 0$, and therefore $q^i_{a,b}q^i_{c,d} = 0$, if $b^{-1}d \ne \pi(a^{-1}c)$, for all $i$. Thus, for fixed $\hat{a},\hat{c} \in G'$ and $\hat{b},\hat{d} \in G$ such that $\hat{b}^{-1}\hat{d} = \pi(\hat{a}^{-1}\hat{c})$, we have that
    \begin{equation*}
        q^i_{\hat{a},\hat{b}} = q^i_{\hat{a},\hat{b}}\left(\sum_{c \in G'} q^i_{c,\hat{d}}\right) 
        = q^i_{\hat{a},\hat{b}}q^i_{\hat{c},\hat{d}} 
        = \left(\sum_{a \in G'}q^i_{a,\hat{b}}\right) q^i_{\hat{c},\hat{d}} 
        = q^i_{\hat{c},\hat{d}}
    \end{equation*}
    It follows that every row/column of $\mathcal{Q}^i$ contains the same multiset of projections and in particular all of its entries commute. Therefore, each correlation $p^i$ must be classical and thus $p$ is classical. The result then follows from Lemma~\ref{lem:classperms}.
\end{proof}

In general we do not have that $D^{\pi \circ \pi'} = D^{\pi} D^{\pi'}$, since this cannot hold for $\pi' = \pi^{-1}$ unless $\pi$ is an isomorphism. The next lemma shows the equation does hold if one of $\pi$ or $\pi'$ is an isomorphism.

\begin{lemma}
    Let $G, G'$, and $G''$ be equicardinal finite groups. If $\pi \in \bij(G,G')$ and $\pi' \in \bij(G',G'')$, then $D^{\pi \circ \pi'} = D^\pi D^{\pi'}$ as long as $\pi$ or $\pi'$ is an isomorphism.
\end{lemma}
\begin{proof}
    Suppose that $\pi'$ is an isomorphism. Then for any $g \in G'$ we have that
    \[P^{\pi'}P^g = P^{\pi' \circ \tau_g} = P^{\tau_{\pi'(g)} \circ \pi'} = P^{\pi'(g)}P^{\pi'},\]
    and $P^{(\pi \circ \pi'(g))^{-1}} = P^{\pi(\pi'(g))^{-1}}$. Applying Lemmas~\ref{lem:avg} and~\ref{lem:classperms} we have the following:
    \begin{align*}
        D^{\pi \circ \pi'} &= \sum_{g \in G'} P^{\pi(\pi'(g))^{-1}} P^\pi P^{\pi'} P^g \\
        &= \sum_{g \in G'} P^{\pi(\pi'(g))^{-1}} P^\pi P^{\pi'(g)} P^{\pi'} \\
        &= \left(\sum_{g \in G'} P^{\pi(\pi'(g))^{-1}} P^\pi P^{\pi'(g)}\right) P^{\pi'} \\
        &= \left(\sum_{g \in G'} P^{\pi(g)^{-1}} P^\pi P^{g}\right) D^{\pi'} \\
        &= D^\pi D^{\pi'}
    \end{align*}
    The case where $\pi$ is an isomorphism is similar.
\end{proof}

One useful consequence of the above is that the set of correlation matrices of $(G,G')$-invariant correlations is fixed under left (resp.~right) multiplication by automorphisms of $G$ (resp.~$G'$), and of course such actions map extreme points to extreme points. Therefore, once we identify some $\pi \in \bij(G,G')$ such that $D^\pi$ is an extreme point, then we can obtain many more such extreme points.

Next we will give another description of the correlation matrices $D^\pi$. For this we will need the following definition.

\begin{definition}
    Given a finite group $G$, define its \emph{reduction matrix} $R := R^G$ as the $G^2 \times G$ matrix such that
    \[R_{(a,b),c} = \begin{cases} 1 & \text{if } b^{-1}a = c\\ 0 & \text{o.w.} \end{cases}.\]
\end{definition}

\begin{lemma}\label{lem:redmtx}
    Let $G$ and $G'$ be finite groups with $|G| = |G'| = n$, and let $\pi \in \bij(G,G')$. Then
    \[D^\pi = \frac{1}{n}(R^G)^T\left(P^\pi \otimes P^\pi\right)R^{G'}.\]
\end{lemma}
\begin{proof}
    Let $s \in G$ and $t \in G'$. We have that
    \begin{align*}
       \frac{1}{n}\left((R^G)^T\left(P^\pi \otimes P^\pi\right)R^{G'}\right)_{s,t} &= \frac{1}{n}\sum_{b,d \in G, \ a,c \in G'} (R^G)^T_{s,(b,d)} (P^\pi \otimes P^\pi)_{(b,d),(a,c)} R^{G'}_{(a,c),t} \\
       &= \frac{1}{n}\sum_{b,d \in G, \ a,c \in G'} R^G_{(b,d),s} P^\pi_{b,a}P^\pi_{d,c} R^{G'}_{(a,c),t} \\
       &= \frac{1}{n}\sum_{d \in G, \ c \in G'} P^\pi_{ds,ct}P^\pi_{d,c} \\
       &= \frac{1}{n}|\{c \in G' : \pi(ct) = \pi(c)s\}|
    \end{align*}
    which is clearly equal to the expression derived for $D^\pi_{s,t}$ in the proof of Lemma~\ref{lem:avg}.
\end{proof}

The above will be used to prove that in the case of abelian groups, the convex hull of correlations produced by group invariant quantum Latin squares is a unitary transformation of the set of classical group invariant correlations (see Theorem~\ref{thm:C2Q}). 

\subsection{Cayley (Di)Graphs}\label{subsec:cayley}

There is a connection between group-invariant quantum correlations and quantum isomorphisms of graphs. For a finite group $G$ and a subset $C \subseteq G$, we denote by $\Cay(G,C)$ the \emph{Cayley (di)graph} for $G$ with connection set $C$, i.e., the digraph with vertex set $G$ and arc set $\{(g,gs): s\in C\}$. Thus there is an arc from $a$ to $b$ if and only if $a^{-1}b \in C$. Note that if $C$ is inverse-closed and does not contain the identity, then the Cayley digraph is an undirected loopless graph, in which case we call it a \emph{Cayley graph}.
We will show that given groups $G$ and $G'$, and a quantum $(G,G')$-invariant correlation, we can construct quantum isomorphic Cayley (di)graphs for $G$ and $G'$, respectively. Unfortunately, there is no guarantee that the digraphs will not also be isomorphic.

Let $G$ and $G'$ be finite groups with $|G|=|G'|$ and suppose there exists a $(G,G')$-invariant correlation $p$ with correlation matrix $D^p$. We will start by defining an auxiliary bipartite graph. Let $B$ be the graph with vertex set $G\cup G'$ where $x\in G$ and $y\in G'$ are adjacent if $D^p_{x,y}\neq 0$. Let $V_1,\dots, V_k$ be the vertex sets of the connected components of $B$. For each $i=1,\dots, k$ we define the sets 
\[C_i=G\cap V_i\,\text{ and } C_i' = G'\cap V_i\]
and for each $I\subseteq [k]$, define 
\[C_I = \bigcup_{i\in I}C_i\,\text{ and } C_I' = \bigcup_{i\in I}C_i'.\]

\begin{remark}\label{rem:Csizes}
    Since the matrix $D^p$ is doubly stochastic, we must have that $|C_i| = |C'_i|$ for all $i \in [k]$. This is because, by definition of the graph $B$, we have that $D^p_{a,b} = 0$ if $a \in C_i$ and $b \notin C'_i$ or vice versa, and therefore the submatrix of $D^p$ with rows from $C_i$ and columns from $C'_i$ must be a doubly stochastic matrix and thus square. Furthermore, it follows that $|C_I| = |C'_I|$ for all $I \subseteq [k]$.
\end{remark}

We will show that for a fixed $I\subseteq[k]$, the isomorphism game for the Cayley digraphs $\Cay(G,C_I)$ and $\Cay(G',C_I')$ can be won using the correlation $p$. Moreover, these are exactly the digraphs for which $p$ wins the isomorphism game.

\begin{theorem}\label{thm:CayConstruction}
    Let $p$ be a $(G,G')$-invariant correlation for finite groups $G$ and $G'$ and let $X$ and $Y$ be digraphs. Then $p$ wins the $(X,Y)$-isomorphism game if and only if they are a pair of Cayley digraphs that arise from the construction above, that is, defining $C_i$ and $C_i'$ as above using the correlation $p$, there is some index set $I \subseteq [k]$ satisfying $X = \Cay(G,C_I)$ and $Y = \Cay(G',C'_I)$.
\end{theorem}
\begin{proof}
    Suppose that $p$ is a winning correlation for the $(X,Y)$-isomorphism game. Then $X$ and $Y$ must have vertex sets $G$ and $G'$ respectively. First we show that $X$ and $Y$ are Cayley digraphs.
    
    Suppose that $a,b,c,d \in G'$ are such that $a^{-1}b = c^{-1}d$ and there is an arc from $a$ to $b$ but not from $c$ to $d$ in $Y$. Let $x,y \in G$ be such that $p(x,y|a,b) \ne 0$. Since $p$ is a winning correlation for the $(X,Y)$-isomorphism game, there must be an arc from $x$ to $y$. However, since $a^{-1}b = c^{-1}d$ and $p$ is $(G,G')$-invariant, $p(x,y|c,d) = p(x,y|a,b) \ne 0$ and therefore there must be an arc from $c$ to $d$, a contradiction. Therefore the existence of an arc from $a$ to $b$ depends only on the value $a^{-1}b$, i.e., $Y$ is a Cayley digraph for $G'$. Analogously, $X$ is a Cayley digraph for $G$.

    Now let $C$ and $C'$ be such that $X = \Cay(G,C)$ and $Y = \Cay(G',C')$. Suppose that $s \in G$ and $t \in G'$ are such that $D^p_{s,t} \ne 0$. We will show that $s \in C$ if and only if $t \in C'$. Let $a,b \in G$ and $c,d \in G'$ be such that $a^{-1}b = s$ and $c^{-1}d = t$, and note that $p(a,b|c,d) \ne 0$ and therefore $(a,b)$ satisfy the same relation (equal/adjacent/distinct non-adjacent) as $(c,d)$. So then $t \in C'$ if and only if there is an arc from $c$ to $d$ if and only if there is an arc from $a$ to $b$ if and only if $s \in C$. Since we have shown that $s \in C$ if and only if $t \in C'$ whenever $D^p_{s,t} \ne 0$, it follows from the definition of the bipartite graph $B$ and the notion of connectivity that there is some $I \subseteq [k]$ such that $C = C_I$ and $C' = C'_I$, as desired.

    The converse is similar.
\end{proof}

Next we give another, more algebraic, characterization of the digraphs for which a given $(G,G')$-invariant correlation wins the corresponding isomorphism game. Here we will use $\mathbf{e}_C \in \mathbb{C}^G$ to denote the characteristic vector of the subset $C \subseteq G$, and similarly for $C' \subseteq G'$.

\begin{theorem}\label{thm:DpeC}
    Let $G$ and $G'$ be finite groups and let $p$ be a $(G,G')$-invariant correlation. Then $p$ wins the $(X,Y)$-isomorphism game if and only if $X = \Cay(G,C)$ and $Y = \Cay(G',C')$ for some subsets $C\subseteq G$ and $C'\subseteq G'$ such that $D^p\mathbf{e}_{C'} = \mathbf{e}_C$.
\end{theorem}
\begin{proof}
    Let $B$ be the auxiliary bipartite graph with components $V_1,\dots V_k$ and $C_i$, $C_i'$ subsets of $G$, $G'$, as constructed above using the correlation $p$. By Theorem~\ref{thm:CayConstruction} it suffices to show that $D^p\mathbf{e}_{C'} = \mathbf{e}_C$ if and only if $C = C_I$ and $C' = C'_I$ for some $I \subseteq [k]$.

    First note that if the vectors $\mathbf{e}_C$ and $\mathbf{e}_{C'}$ do not have the same multiset of entries, i.e., $|C| \ne |C'|$, then $D^p\mathbf{e}_{C'} \ne \mathbf{e}_C$ since the sums of the entries of the vectors are different and $D^p$ is doubly stochastic. Moreover, it cannot be the case that $C = C_I$ and $C' = C'_I$ by Remark~\ref{rem:Csizes}. Thus the above if and only if holds when $\mathbf{e}_C$ and $\mathbf{e}_{C'}$ do not have the same multiset of entries, so for the remainder of the proof we may assume that they do.

    For any $n \times n$ doubly stochastic matrix $D$ and vectors $u, v \in \mathbb{C}^n$ with the same multiset of entries, it is well known that $Du = v$ if and only if $D_{ij} = 0$ whenever $u_j \ne v_i$ (see, e.g.,~\cite[Lemma~7.1]{mancinska_graph_2020}). 
    Therefore $D^p\mathbf{e}_{C'} = \mathbf{e}_C$ if and only if $D^p_{a,b} = 0$ whenever ($a \in C$ and $b \notin C'$) or ($a \notin C$ and $b \in C'$). By definition of the graph $B$, this is equivalent to there being no edges between $C \cup C'$ and $(G \cup G') \setminus (C \cup C')$. This is, in turn, equivalent to $C \cup C'$ being the vertex set of a union of connected components of $B$, which is equivalent to $C = C_I$ and $C' = C'_I$ for some $I \subseteq [k]$.
%
%
\end{proof}

We would like to use Theorems \ref{thm:CayConstruction} and \ref{thm:DpeC} to construct new pairs of non-isomorphic, quantum isomorphic (di)graphs. To do this, we would construct a $(G,G')$-invariant quantum correlation $p$ and apply the theorems to construct all of the pairs of Cayley (di)graphs for which $p$ wins the isomorphism game. Since $p$ is a quantum correlation, all these pairs are quantum isomorphic. If any of these pairs are non-isomorphic, then we have found the desired example. In later sections we will see how to construct $(G,G')$-invariant quantum correlations. However, despite some effort, we were unable to find any that yield non-isomorphic graphs using this method. We therefore propose the following problem.

\begin{prob}\label{prob:Cayley}
    Use this construction to come up with non-isomorphic, quantum isomorphic Cayley (di)graphs, or show that this is not possible.
\end{prob}

We remark that the results of Section~\ref{subsec:Z24comps} make us hopeful that the above open problem can be resolved positively.

\section{Quantum Latin Squares}\label{sec:QLS}

Recall that a Latin square of order $n$ is an $n \times n$ array of the numbers $1, 2, \ldots, n$ such that each number appears precisely once in each row and column. The following definition, introduced in~\cite{qlatin}, defines a quantum analog of these objects:

\begin{definition}
A \emph{quantum Latin square} $\Psi = (\ket{\psi_{i,j}})$ is an $n \times n$ array of vectors from an $n$-dimensional complex vector space $V$ such that the entries of each row and column form an orthonormal basis of $V$.
\end{definition}

Typically we will take $V = \mathbb{C}^n$, but in general we may want to let $V$ be an $n$-dimensional subspace of $\mathbb{C}^N$ for some $N \ge n$ (e.g., when we define the composition of quantum Latin squares).

\begin{definition}
    For finite groups $G$ and $G'$, we say that a quantum Latin square $\Psi = (\ket{\psi_{a,b}})$ is \emph{$(G,G')$-invariant} if its rows are indexed by $G$ and its columns by $G'$ and the inner product $\inner{\psi_{a,b}}{\psi_{c,d}}$ depends only on the values of $a^{-1}c \in G$ and $b^{-1}d \in G'$.
\end{definition}

Obviously, we must have $|G| = |G'|$ for a $(G,G')$-invariant quantum Latin square to exist. Later we will see that a $(G,G')$-invariant quantum Latin square exists if and only if $G$ and $G'$ have the same multiset of degrees of irreducible representations. In particular, this means that a $(G,G')$-invariant quantum Latin square always exists for abelian groups $G$ and $G'$ of the same order.

Any list of vectors $\ket{\psi_1}, \ldots, \ket{\psi_n} \in \mathbb{C}^d$ can, up to isometry, be determined by the pairwise inner products $\inner{\psi_i}{\psi_j}$, $i,j = 1, \ldots, n$. Thus, essentially all information about a quantum Latin square is contained in its pairwise inner products. 
By definition of a $(G,G')$-invariant quantum Latin square, this information can be contained in a matrix with rows indexed by $G$ and columns by $G'$:



\begin{definition}
    Suppose that $\Psi = (\ket{\psi_{a,b}})$ is a $(G,G')$-invariant quantum Latin square. Define its \emph{transformation matrix}, denoted $U^\Psi$, to be the $G \times G'$ matrix defined entrywise as
    \[U^\Psi_{x,y} = \inner{\psi_{a,b}}{\psi_{c,d}} \text{ for some } a,c \in G, \ b,d \in G' \text{ s.t. } a^{-1}c = x \ \& \ b^{-1}d = y.\]
\end{definition}

\begin{example}\label{ex:charconstruction}
Let $G$ and $G'$ be abelian groups of order $n$ with irreducible characters $\chi_1,\dots, \chi_n$ and $\chi_1',\dots,\chi_n'$, respectively. For each $g\in G$ and $h\in G'$, define the vector $\ket{\psi_{g,h}}$ by 
\[\ket{\psi_{g,h}} = (\chi_1(g^{-1})\chi'_1(h),\dots, \chi_n(g^{-1})\chi_n'(h))^T. \]
Then the array $(\ket{\psi_{g,h}})_{g\in G,h\in H}$ is a $(G,G')$-invariant quantum Latin square. We remark that when $G = G'$, this is equivalent to the construction of Definition 4.1 in~\cite{quantumvsnonlocal}. An example with $G=\Z_4$ and $G'=\Z_2\times \Z_2$ is given in Figure \ref{fig:charsQLS} and the corresponding transformation matrix is
\[
\begin{pmatrix}
1 & 0 & 0 & 0 \\
0 & 0 & \alpha & \beta \\
0 & 1 & 0 & 0 \\
0 & 0 & \beta & \alpha
\end{pmatrix},
\]
where $\alpha = \frac{1-i}2$ and $\beta = \frac{1+i}2.$
\begin{figure}[h]
\renewcommand{\arraystretch}{1.2}
\[\frac1{2}\,\,\begin{array}{|c|c|c|c|}
\hline
(1,1,1,1) & (1,-1,1,-1) & (1,1,-1,-1) & (1,-1,-1,1) \\
\hline
(1,-i,-1,i) & (1,i,-1,-i) & (1,-i,1,-i) &(1,i,1,i)\\
\hline
(1,-1,1,-1) & (1,1,1,1) & (1,-1,-1,1) & (1,1,-1,-1) \\
\hline
(1,i,-1,-i) & (1,-i,-1,i) & (1,i,1,i) & (1,-i,1,-i) \\
\hline
\end{array}\]
\caption{A $(\Z_4,\Z_2\times \Z_2)$-invariant quantum Latin square.}
\label{fig:charsQLS}
\end{figure}
\end{example}

Note that by definition of quantum Latin square, the transformation matrix of any group invariant quantum Latin square must of course be square. In fact, it turns out that the transformation matrix of a $(G,G')$-invariant quantum Latin square is unitary. Moreover, we prove the following:

\begin{theorem}\label{thm:qlstransf}
    Let $G$ and $G'$ be finite groups. Then a matrix $U \in \mathbb{C}^{G \times G'}$  is the transformation matrix of a $(G,G')$-invariant quantum Latin square if and only if
    \begin{enumerate}
        \item $U$ is unitary;
        \item $\overline{U_{a,b}} = U_{a^{-1},b^{-1}}$, for all $a \in G$, $b \in G'$, and
        \item $U_{ab,c} = \displaystyle\sum_{\substack{x,y \in G' \\ xy=c}} U_{a,x}U_{b,y}$ for all $a,b \in G$ and $c \in G'$.
    \end{enumerate}
\end{theorem}
\begin{proof}
Let $\Psi = (\ket{\psi_{a,b}})$ be a $(G,G')$-invariant quantum Latin square, and let $U:= U^\Psi$ be its transformation matrix. Then, letting $\delta_{x,y}$ denote the Kronecker delta function, we have that
    \begin{align*}
    \left(UU^\dagger\right)_{x,y} &= \sum_{z \in G'} U_{x,z}\overline{U_{y,z}} \\
    &= \sum_{z \in G'} \inner{\psi_{e,e}}{\psi_{x,z}}\overline{\inner{\psi_{xy^{-1},e}}{\psi_{x,z}}} \\
    &= \sum_{z \in G'} \inner{\psi_{e,e}}{\psi_{x,z}}\inner{\psi_{x,z}}{\psi_{xy^{-1},e}} \\
    &= \bra{\psi_{e,e}}\left( \sum_{z \in G'} \outerp{\psi_{x,z}}{\psi_{x,z}}\right) \ket{\psi_{xy^{-1},e}} \\
    &= \inner{\psi_{e,e}}{\psi_{xy^{-1},e}} \\
    &= \delta_{x,y},
    \end{align*}
    where the last two equalities follow from the fact that every row and column of $\Psi$ is an orthonormal basis. Since $U$ must be square, it follows that it is unitary, i.e., item $(1)$ holds.
    Next, we have that
    \[\overline{U_{x,y}} = \overline{\inner{\psi_{e,e}}{\psi_{x,y}}} = \inner{\psi_{x,y}}{\psi_{e,e}} = U_{x^{-1},y^{-1}},\]
    i.e., item $(2)$ holds.
    Finally,
    \begin{align*}
        \sum_{\substack{x,y \in G' \\ \text{ s.t. } xy=c}} U_{a,x}U_{b,y} &= \sum_{x \in G'} U_{a,x}U_{b,x^{-1}c} \\
        &= \sum_{x \in G'} \inner{\psi_{a^{-1},e}}{\psi_{e,x}} \inner{\psi_{e,x}}{\psi_{b,c}} \\
        &= \bra{\psi_{a^{-1}, e}} \left( \sum_{x \in G'}\outerp{\psi_{e,x}}{\psi_{e,x}} \right) \ket{\psi_{b,c}} \\
        &= \inner{\psi_{a^{-1}, e}}{\psi_{b,c}} \\
        &= U_{ab,c},
    \end{align*}
    i.e., item $(3)$ holds. Thus we have proven the forward direction of the theorem.

    Conversely, suppose that $U$ is a $G \times G'$ matrix satisfying $(1)$--$(3)$. We first show that $U_{e,e} = 1$ and thus $U_{a,e} = U_{e,b} = 0$ for all $a \in G\setminus{e}$ and $b \in G'\setminus{e}$. Using properties $(1)$--$(3)$ in reverse order, we have that
    \begin{align*}
    U_{e,e} &= \sum_{\substack{x,y \in G' \\ xy = e}} U_{e,x}U_{e,y} \\
    &= \sum_{x \in G'} U_{e,x}U_{e,x^{-1}} \\
    &= \sum_{x \in G'} U_{e,x}\overline{U_{e,x}} \\
    &= 1
    \end{align*}
    Since $U$ is unitary, the remaining entries of the $e$ row/column must be 0 as desired.
    
    Now let $\{\ket{\psi_{e,b}} : b \in G'\}$ be any orthonormal basis of $\mathbb{C}^{|G'|}$. For all $a \in G$ and $b \in G'$, define
    \begin{equation}\label{eq:defkets}
        \ket{\psi_{a,b}} = \sum_{x \in G'} U_{a,x^{-1}b} \ket{\psi_{e,x}}.
    \end{equation}
    Note that the above holds trivially if $a=e$, so we are not redefining our basis $\{\ket{\psi_{e,b}}:b\in G'\}$. We must show that these vectors form a $(G,G')$-invariant quantum Latin square whose inner product matrix is $U$. Towards this, we consider the inner products
    \begin{align*}
        \inner{\psi_{a,b}}{\psi_{c,d}} &= \sum_{x,y \in G'} \overline{U_{a,x^{-1}b}}U_{c,y^{-1}d}\inner{\psi_{e,x}}{\psi_{e,y}} \\
        &= \sum_{x \in G'} \overline{U_{a,x^{-1}b}}U_{c,x^{-1}d} \\
        &= \sum_{x \in G'} U_{a^{-1},b^{-1}x}U_{c,x^{-1}d} \\
        &= U_{a^{-1}c,b^{-1}d}.
    \end{align*}
    Since $U_{e,e} = 1$ as shown earlier, the above equation implies that every row and column of $\Psi = (\ket{\psi_{a,b}})$ is an orthonormal basis of $\mathbb{C}^{|G'|}$ and therefore $\Psi$ is a quantum Latin square. Moreover, $\Psi$ is $(G,G')$-invariant with transformation matrix $U$ by the same equation.
\end{proof}

Due to the above theorem, we will often refer to any $G \times G'$ matrix $U$ that satisfies conditions $(1)$--$(3)$ as a \emph{$(G,G')$-transformation matrix}, without referring to any $(G,G')$-invariant quantum Latin square.

\begin{remark}\label{rem:trans_mtx_extras}
    We note that any transformation matrix $U$ of a $(G,G')$-invariant quantum Latin square must satisfy $U_{e,e} = 1$ and thus $U_{a,e} = U_{e,b} = 0$ for $a \in G \setminus \{e\}$ and $b \in G' \setminus \{e\}$. Moreover, by symmetry of $G$ and $G'$ in the definition of $(G,G')$-invariant quantum Latin square, we have that any transformation matrix $U$ also satisfies
    \[U_{a,bc} = \sum_{\substack{x,y \in G \\ xy=a}} U_{x,b}U_{y,c}\text{ for all } a \in G \text{ and } b,c \in G'.\]
    Lastly, we leave it to the reader to verify that the identity matrix is a $(G,G)$-transformation matrix, and that if $U$ is a $(G,G')$-transformation matrix, then so is $\overline{U}$, and $U^\dagger$ is a $(G',G)$-transformation matrix.
\end{remark}

For any $n \times n$ quantum Latin square $\Psi = (\ket{\psi_{i,j}})$, there is a corresponding quantum permutation matrix
\begin{equation}\label{eq:qls2qpm}
    \mathcal{P}^\Psi = (\outerp{\psi_{i,j}}{\psi_{i,j}}),
\end{equation}
and the corresponding correlation defined as in Equation~\eqref{eq:qcorr} is given by
\[p(i,j|k, \ell) = \frac{1}{n}|\langle \psi_{i,k} | \psi_{j,\ell}\rangle|^2.\]
It follows immediately that if $\Psi$ is $(G,G')$-invariant, then so is $p$ and its correlation matrix $D^p$ is equal to $U \circ \overline{U}$, where $U$ is the transformation matrix of $\Psi$ and $\circ$ denotes the entrywise product. We remark that in general, not every quantum permutation matrix arises from a quantum Latin square in the manner of Equation~\eqref{eq:qls2qpm}, since $\mathcal{P}^\Psi$ necessarily only has entries of rank one. However, any quantum permutation matrix where every entry has rank one does arise in this way, but the correspondence is not one-to-one since multiplying any entry of $\Psi$ by a complex unit would not change $\mathcal{P}^\Psi$.

\section{Composition}\label{sec:composition}

Given $n \times n$ quantum permutation matrices $\mathcal{P} = (p_{ij})$ and $\mathcal{Q} = (q_{k\ell})$, one can \emph{compose} these to obtain a new quantum permutation matrix~\cite{qperms}, denoted $\mathcal{P} \circ \mathcal{Q}$, defined as
\[\left(\mathcal{P} \circ \mathcal{Q}\right)_{ik} = \sum_{j \in [n]} p_{ij} \otimes q_{jk}.\]
Moreover, if $p$ and $q$ are the correlations arising from $\mathcal{P}$ and $\mathcal{Q}$ via Equation~\eqref{eq:qcorr}, then the correlation that arises from $\mathcal{P} \circ \mathcal{Q}$ is the composition $p \circ q$.

Here we will introduce a notion of \emph{composition} of quantum Latin squares that is motivated by the above, but behaves differently in some key aspects.

\begin{definition}
    Let $\Psi = (\ket{\psi_{i,j}})$ and $\Psi' = (\ket{\psi'_{j,k}})$ be two $n \times n$ quantum Latin squares. Then we define the \emph{composition} of $\Psi$ and $\Psi'$, denoted by $\Psi \circ \Psi'$ as the $n \times n$ array with $i,k$-entry equal to 
    \[\frac1{\sqrt n}\sum_{j \in [n]} \ket{\psi_{i,j}} \otimes \ket{\psi'_{j,k}}.\] 
\end{definition}

Note that the underlying space of $\Psi \circ \Psi'$ in the above definition is $\mathbb{C}^n \otimes \mathbb{C}^n \cong \mathbb{C}^{n^2}$. The rows/columns of $\Psi \circ \Psi'$ clearly cannot be orthonormal bases of this space since they only contain $n$ vectors each. This does not preclude the possibility that every row and column spans the same $n$-dimensional subspace of $\mathbb{C}^n \otimes \mathbb{C}^n$, but this turns out not to be the case in general as can be seen by the following example:

\begin{example}
Let
\[\Psi = \begin{array}{|c|c|}
\hline
(1,0) & (0,1) \\
\hline
(0,1) & (1,0)\\
\hline
\end{array}
\quad \text{and} \quad
\Psi' = \begin{array}{|c|c|}
\hline
(1,0) & (0,1) \\
\hline
(0,1) & (-1,0)\\
\hline
\end{array}\]
Then
\[\Psi \circ \Psi'= \frac{1}{\sqrt{2}}\,\,\begin{array}{|c|c|}
\hline
(1,0,0,1) & (0,1,-1,0) \\
\hline
(0,1,1,0) & (-1,0,0,1)\\
\hline
\end{array}\]
whose rows do not span the same space.
\end{example}

Thus the main difference between composition of quantum Latin squares and composition of quantum permutation matrices is that the composition of quantum Latin squares is not necessarily a quantum Latin square.

The composition of quantum Latin squares will however always be an array in which each row and column is an orthonormal set of vectors. But the above example shows that these are not always bases for some common vector space.

What is interesting for us is that, for group invariant quantum Latin squares, the composition is in fact a quantum Latin square which is moreover group invariant:

\begin{theorem}
    Suppose that $G, G'$, and $G''$ are finite groups and that $\Psi = (\ket{\psi_{a,b}})_{a \in G, b \in G'}$ and $\Psi' = (\ket{\psi'_{b,c}})_{b \in G', c \in G''}$ are $(G,G')$-invariant and $(G',G'')$-invariant quantum Latin squares respectively, and let $U$ and $U'$ be their respective transformation matrices. Then the composition $\Psi \circ \Psi'$ is a $(G,G'')$-invariant quantum Latin square with transformation matrix equal to $UU'$. 
\end{theorem} 
\begin{proof}
Let $\Phi=(\ket{\vphi_{a,c}})_{a\in G,c\in G''}$ denote the composition of $\Psi$ and $\Psi'$, and let $U'' = UU'$. Then $\Phi$ is an $n\times n$ array of vectors with rows indexed by $G$ and columns by $G''$. We will first show that 
\begin{equation}\label{eq:compisGinv}
    \inner{\varphi_{a,c}}{\varphi_{b,d}} = \left(U''\right)_{a^{-1}b,c^{-1}d}.
\end{equation}
We have that
\begin{align*}
    \inner{\phi_{a,c}}{\phi_{b,d}} & = \frac1n \left(\sum_{x\in G'} \bra{\psi_{a,x}}\otimes \bra{\psi'_{x,c}}\right)\left(\sum_{y\in G'} \ket{\psi_{b,y}}\otimes \ket{\psi'_{y,d}}\right)\\
    & = \frac1n \left(\sum_{x,y\in G'} \inner{\psi_{a,x}}{\psi_{b,y}} \inner{\psi'_{x,c}}{\psi'_{y,d}}\right)\\
    & = \frac1n \left(\sum_{x,y\in G'} U_{a^{-1}b,x^{-1}y} U'_{x^{-1}y,c^{-1}d}\right)\\
    & = \sum_{z\in G'} U_{a^{-1}b,z} U'_{z,c^{-1}d}\\
    & = U''_{a^{-1}b,c^{-1}d}.
\end{align*}
Note that in particular we have that $U''_{e,e} = 1$ and $U''_{a,e} = U_{e,c} = 0$ for all nonidentity elements $a \in G$ and $c \in G''$ since the analogous statements hold for $U$ and $U'$. Therefore the above implies that each $\ket{\varphi_{a,c}}$ is a unit vector and moreover the vectors in a row/column of $\Phi$ are orthonormal.

Finally, for any $a,b \in G$, $c \in G''$, we have that
\begin{equation}\label{eq:inspan}
    \sum_{d \in G''} |\!\inner{\phi_{a,c}}{\phi_{b,d}}\!|^2 = \sum_{d \in G''} |U''_{a^{-1}b,c^{-1}d}|^2 = 1
\end{equation}
since $U''$ is unitary. Since $\ket{\varphi_{a,c}}$ is a unit vector and $\{\ket{\varphi_{b,d}} : d \in G''\}$ is an orthonormal set, Equation~\eqref{eq:inspan} implies that $\ket{\varphi_{a,c}} \in \spn\{\ket{\varphi_{b,d}} : d \in G''\}$. Thus every vector of $\Phi$ is contained in the span of the vectors in any row of $\Phi$, and the proof for columns is analogous. It follows that every row/column of $\Phi$ is an orthonormal basis of a common subspace of the underlying space and therefore $\Phi$ is indeed a quantum Latin square. Moreover, Equation~\eqref{eq:compisGinv} proves that $\Phi$ is $(G,G'')$-invariant with transformation matrix $UU'$.
\end{proof}



As a corollary we immediately obtain the following:

\begin{cor}
    The existence of a $(G,G')$-invariant quantum Latin square is an equivalence relation on groups.\qed
\end{cor}

The above also means that if $\mathcal{G}$ is an equivalence class of the relation in the corollary, then the transformation matrices in the set
\[\bigcup_{G,G' \in \mathcal{G}}\{U \in \mathbb{C}^{G \times G'} : U \text{ is a $(G,G')$-transformation matrix}\}\]
form a groupoid, and that the $(G,G)$-transformation matrices form a group.

As mentioned above, if $\mathcal{P}$ and $\mathcal{Q}$ are quantum permutation matrices with corresponding correlations $p$ and $q$, then the correlation produced by the composition $\mathcal{P} \circ \mathcal{Q}$ will be $p \circ q$. Recall also that in terms of correlation matrices we have that $D^{p \circ q} = D^p D^q$. Thus the correspondence between quantum permutation matrices and correlations commutes with their respective notions of composition. However, this is not the case for composing group invariant quantum Latin squares. Indeed, if $U$ and $U'$ are the transformation matrices of a $(G,G')$-invariant and $(G',G'')$-invariant quantum Latin squares respectively, then their correlation matrices are $U \circ \overline{U}$ and $U' \circ \overline{U'}$. But the transformation matrix of their composition is $UU'$ and thus its correlation matrix is $UU' \circ \overline{UU'}$, which is in general not equal to $\left(U \circ \overline{U}\right)\left(U' \circ \overline{U'}\right)$. Indeed, the correlation matrix $\left(U \circ \overline{U}\right)\left(U' \circ \overline{U'}\right)$ is not even necessarily the correlation matrix of some group invariant quantum Latin square, as the following example shows.

\begin{example}
The following correlation matrix of a $\mathbb{Z}_5$-invariant quantum Latin square was given in~\cite{quantumvsnonlocal} (here $\varphi$ denotes the golden ratio $\frac{1+\sqrt{5}}{2}$):
\begin{equation}\label{eq:Z5corrmtx1}
\frac{1}{5}\begin{pmatrix}
5 & 0 & 0 & 0 & 0 \\
0 & 1 + \varphi & 1 & 1 & 2 - \varphi \\
0 & 1 & 2 - \varphi & 1 + \varphi & 1 \\
0 & 1 & 1 + \varphi & 2 - \varphi & 1 \\
0 & 2 - \varphi & 1 & 1 & 1 + \varphi
\end{pmatrix}.
\end{equation}
Multiplying this matrix by its transpose, which is the correlation matrix of a $\mathbb{Z}_5$-invariant quantum Latin square by Remark~\ref{rem:trans_mtx_extras}, yields the following:
\begin{equation}\label{eq:Z5corrmtx2}
\frac{1}{25}\begin{pmatrix}
25 & 0 & 0 & 0 & 0 \\
0 & 9 & 6 & 6 & 4 \\
0 & 6 & 9 & 4 & 6 \\
0 & 6 & 4 & 9 & 6 \\
0 & 4 & 6 & 6 & 9
\end{pmatrix}.
\end{equation}
As we will see in Theorem~\ref{thm:abeliancase}, there are only finitely many transformation matrices of $\mathbb{Z}_5$-invariant quantum latin squares, and thus we can simply compute all of them and their corresponding correlation matrices. They are all either permutation matrices or the matrix in~\eqref{eq:Z5corrmtx1} multiplied on the left and/or right by a permutation matrix. Thus the matrix in~\eqref{eq:Z5corrmtx2} is not among them.
\end{example}




\section{Representation Theory}\label{sec:reptheory}

In this section we will see some connections between $(G,G')$-transformation matrices and representations of the groups $G$ and $G'$. 
We start by introducing some basic concepts in representation theory, but for more on group representations, we refer the reader to \cite{james2001reps,serre-reps}. 

Let $G$ be a finite group and let $\rho$ be a representation, that is, a homomorphism from $G$ to some general linear group. The dimension, $d,$ of this general linear group is the \emph{degree} of $\rho$. We say that representations $\rho$ and $\rho'$ of degree $d$ are \emph{equivalent} if there exists an invertible $d\times d$ matrix, $M$ such that $M\rho(g)M^{-1}=\rho'(g)$ for all $g\in G.$ The representation  $\rho$ is said to be \textit{unitary} if $\rho(g)$ is a unitary matrix for all $g\in G.$ We observe that every representation is equivalent to some unitary representation. Given a representation $\rho$, its corresponding \emph{character} is the map $\chi:G\to \C$ given by $g\mapsto \Tr(\rho(g)),$ where $\Tr$ is the usual trace.

For matrices $M$ and $N$, we denote by $\langle M,N\rangle$ their normalized Hilbert-Schmidt inner product, i.e., $\langle M, N \rangle = \tr(M^\dagger N) = \frac{1}{d}\Tr(M^\dagger N)$ where $d$ is the dimension of the matrices and $\tr$ is the trace scaled so that $\tr(I) = 1$. We will call a representation $\rho$ \textit{quasi-regular} if it is unitary and if for all distinct $g,h\in G$ we have $\langle \rho(g),\rho(h)\rangle = 0$.

We are in particular interested in one such representation, namely the left regular representation, $\lambda$, defined by $(\lambda(g))_{x,y} = \delta_{x,gy}$ for all $g,x,y\in G$. This representation is clearly quasi-regular.



\begin{lemma}\label{lem:reps}
    Let $G$ and $G'$ be finite groups and suppose that $\rho$ and $\rho'$ are quasi-regular representations of $G$ and $G'$ respectively. Define $\B:=\{\rho(g):g\in G\}$ and $\B':= \{\rho'(h):h\in G'\}$. Then the following are equivalent:
    \begin{enumerate}
        \item $\spn(\B) = \spn(\B')$;
        \item There exists a $(G,G')$-transformation matrix $U$ such that
        \[\rho'(h) = \sum_{x \in G} U_{x,h}\rho(x) \text{ for all } h \in G'.\]
    \end{enumerate}
\end{lemma}
\begin{proof}
    Suppose first that 1.\ holds. Note that with respect to our normalized trace inner product, $\B$ and $\B'$ are orthonormal bases, so the matrix $U$ defined by $U_{g,h}:=\langle \rho(g),\rho'(h)\rangle$ is a unitary change of basis matrix between them and satisfies
    \begin{equation}\label{eq:changebasis}
        \rho'(h) = \sum_{x \in G} U_{x,h}\rho(x) \text{ for all } h \in G'.
    \end{equation}

    Since $\rho$ and $\rho'$ are homomorphisms and each value is a unitary matrix, we have $\rho(g^{-1}) = \rho(g)^{-1} = \rho(g)^\dagger$ for all $g\in G$, and similarly for $\rho'$. Now we see that for $g\in G$, $h \in G'$ we get
    \begin{align*}
        U_{g^{-1},h^{-1}} &= \langle \rho(g^{-1}),\rho'(h^{-1}) \rangle \\
        &= \langle \rho(g)^\dagger,\rho'(h)^\dagger \rangle \\
        &= \tr\left(\rho(g) \rho'(h)^\dagger\right) \\
        &= \overline{\tr\left(\rho(g)^\dagger \rho'(h)\right)} \\
        &= \overline{U_{g,h}}
    \end{align*}
    as needed.

    Lastly, we will show that 
    \[U_{a,bc} = \sum_{\substack{x,y \in G\\ xy=a}} U_{x,b} U_{y,c},\] 
    which completes the proof that $U$ is a $(G,G')$-transformation matrix by Remark~\ref{rem:trans_mtx_extras}. We have that
    \begin{align*}
        U_{a,bc} &= \langle \rho(a), \rho'(bc)\rangle \\
        &= \langle \rho(a), \rho'(b)\rho'(c)\rangle \\
        &= \left\langle \rho(a), \left(\sum_{x \in G}U_{x,b}\rho(x)\right)\left(\sum_{y \in G}U_{y,c}\rho(y)\right)\right\rangle \\
        &= \left\langle \rho(a), \sum_{x,y \in G}U_{x,b}U_{y,c}\rho(x)\rho(y)\right\rangle \\
        &= \left\langle \rho(a), \sum_{x,y \in G}U_{x,b}U_{y,c}\rho(xy)\right\rangle \\
        &= \sum_{\substack{x,y\in G\\xy=a}}U_{x,b}U_{y,c},
    \end{align*}
    where the last equality is by quasi-regularity of $\rho$.

    Conversely, suppose that $(2)$ holds. Then clearly, $\spn(\B')\subseteq\spn(\B)$, and since $\B$ and $\B'$ are orthonormal sets of the same size, this implies equality.
\end{proof}

\begin{remark}\label{rem:construction}
    The representations $\rho$ and $\rho'$ from the above lemma can be used directly to construct a $(G,G')$-invariant quantum Latin square. Specifically, letting $\ket{\psi_{a,c}} = \rho'(c)\rho(a^{-1})$ and using the normalized Hilbert-Schmidt inner product gives a quantum Latin square with transformation matrix equal to the matrix $U$ from the lemma.
\end{remark}

Next we show that given a quasi-regular representation of $G$ and a $(G,G')$-transformation matrix, we can construct a quasi-regular representation of $G'$ satisfying the conditions of Lemma~\ref{lem:reps}.

\begin{lemma}\label{lem:quasi}
    If $\rho$ is a quasi-regular representation of $G$ and $U$  a $(G,G')$-transformation matrix, then 
    \[\rho'(b) := \sum_{a \in G}U_{a,b}\rho(a)\] 
    is a quasi-regular representation of $G'$.
\end{lemma}
\begin{proof}
    Since $U_{x,e}=\delta_{x,e}$ for all $x\in G$, it is clear that $\rho'(e) =I$, and for $g,h\in G'$, we have 
    \begin{align*}
        \rho'(g)\rho'(h) & = \left(\sum_{x\in G}U_{x,g}\rho(x)\right)\Bigg(\sum_{y\in G}U_{y,h}\rho(y)\Bigg)\\ 
        & = \sum_{x,y\in G}U_{x,g}U_{y,h}\rho(xy)\\ 
        & = \sum_{z\in G}\left(\sum_{x\in G}U_{x,g}U_{x^{-1}z,h}\right)\rho(z)\\ 
        & = \sum_{z\in G}U_{z,gh}\rho(z)\\
        & = \rho'(gh),
    \end{align*}
    so $\rho'$ is a homomorphism. We also have 
    \begin{align*}
        \rho'(g)^\dagger & = \sum_{x\in G}\overline{U_{x,g}}\rho(x)^\dagger\\ 
        & = \sum_{x\in G}U_{x^{-1},g^{-1}}\rho(x^{-1})\\
        & = \sum_{x\in G}U_{x,g^{-1}}\rho(x)\\
        & = \rho'(g^{-1}).
    \end{align*}
    It follows that
    \[\rho'(g)^\dagger\rho'(g) = \rho'(g^{-1})\rho'(g) = \rho'(g^{-1}g) = \rho'(e) = I,\]
    so $\rho'$ is unitary. Finally it is clear that $\langle \rho'(g),\rho'(h)\rangle = \tr(\rho'(g)^\dagger\rho'(h)) = \tr(\rho'(g^{-1}h))$, and 
    \begin{align*}
        \tr(\rho'(g^{-1}h)) & = \sum_{x \in G}U_{x,g^{-1}h} \tr(\rho(x))\\
        & = U_{e,g^{-1}h}\\
        & = \delta_{g,h}.
    \end{align*}
    Therefore, $\langle \rho'(g),\rho'(h)\rangle=\delta_{g,h}$, and we have shown that $\rho'$ is a quasi-regular representation of $G'.$ 
\end{proof}

The next lemma shows that without loss of generality, we may take $\rho$ (respectively $\rho'$) to be the left regular representation of $G$ (respectively $G'$). In this case, the change of basis matrix $U$ satisfies an additional interesting relation. Namely, if we let $\hat{\rho}'(b) = \sum_{x \in G} U_{x,b} \lambda(x)$ where $\lambda$ is the left regular representation of $G$, then $\hat{\rho}'(b) = U\lambda'(b)U^\dagger$ where $\lambda'$ is the left regular representation of $G'$. 

\begin{lemma}\label{lem:regreps}
    Let $G, G', \rho$ and $\rho'$ be as in Lemma \ref{lem:reps} and assume that the conditions of the Lemma hold with $(G,G')$-transformation matrix $U$.
    Let $\lambda, \lambda'$ be the left regular representations of $G,G'$, respectively. Then defining
    \begin{align*}
        \hat{\rho}(a) &:= \sum_{y \in G'} \overline{U_{a,y}}\lambda'(y) \\
        \hat{\rho}'(b) &:= \sum_{x \in G} U_{x,b }\lambda(x)
    \end{align*}
    satisfies
    \begin{align*}
        &\hat{\rho}(a) = U^\dagger \lambda(a) U, \\
        &\hat{\rho}'(b) = U \lambda'(b) U^\dagger,
    \end{align*}
    and
    \begin{align*}
        &\langle \lambda(a), \hat{\rho}'(b) \rangle = \langle \hat{\rho}(a), \lambda'(b) \rangle = U_{a,b} = \langle \rho(a), \rho'(b) \rangle.
    \end{align*}
\end{lemma}
\begin{proof}
We will show that $\hat{\rho}(a)= U^\dagger\lambda(a)U$; the second equality can be proven similarly. Let $g,h\in G'$. Then 
\begin{align*}
    \left(U^\dagger\lambda(a)U\right)_{g,h} &= \sum_{x,y\in G}U^\dagger_{g,y}\lambda(a)_{y,x}U_{x,h}\\
    & = \sum_{x\in G}\overline{U_{ax,g}}U_{x,h}\\
    & = \sum_{x\in G}\left(\sum_{y\in G}\overline{U_{a,y}U_{x,y^{-1}g}}\right)U_{x,h}\\
    & = \sum_{y\in G}\overline{U_{a,y}}\left(\sum_{x\in G}U^\dagger_{y^{-1}g,x}U_{x,h}\right)\\
    & = \sum_{y\in G}\overline{U_{a,y}}\delta_{y^{-1}g,h}\\
    & = \sum_{x\in G}\overline{U_{a,x}}\lambda'(x)_{g,h}
\end{align*}
and we have shown that $U^\dagger\lambda(a)U = \hat{\rho}(a)$. 
Further,
\begin{align*}
    \langle \lambda(a),\hat{\rho}'(b)\rangle & = \tr\left(\lambda(a)^\dagger\hat{\rho}'(b)\right) \\
    & = \tr\left(\lambda(a^{-1})\sum_{x\in G}U_{x,b}\lambda(x)\right)\\
    & = \sum_{x\in G}U_{x,b}\tr(\lambda(a^{-1}x))\\
    & = \sum_{x\in G}U_{x,b}\delta_{a,x}\\
    & = U_{a,b}.
\end{align*}
The other equalities can be shown similarly.
\end{proof}

As a corollary we obtain another characterization of $(G,G')$-transformation matrices:

\begin{cor}\label{cor:transformchar}
    Let $G$ and $G'$ be finite groups with left regular representations $\lambda$ and $\lambda'$ respectively. Then a unitary matrix $U \in \mathbb{C}^{G \times G'}$ is a $(G,G')$-transformation matrix if and only if
    \[U\lambda'(b)U^\dagger = \sum_{a \in G} U_{a,b} \lambda(a).\]
\end{cor}
\begin{proof}
    The forward direction follows immediately from Lemma~\ref{lem:regreps}. For the backwards implication, suppose that $U \in \mathbb{C}^{G \times G'}$ is a unitary matrix satisfying
    \[U\lambda'(b)U^\dagger = \sum_{a \in G} U_{a,b} \lambda(a).\]
    Following the construction of Remark~\ref{rem:construction}, we let
    \[\ket{\psi_{a,b}} = U\lambda'(b)U^\dagger\lambda(a^{-1})\]
    for $a \in G$ and $b \in G'$. Then
    \begin{align*}
        \inner{\psi_{a,b}}{\psi_{c,d}} &= \left\langle U\lambda'(b)U^\dagger\lambda(a^{-1}), U\lambda'(d)U^\dagger\lambda(c^{-1})\right\rangle \\
        &= \tr\left(\left(U\lambda'(b)U^\dagger\lambda(a^{-1})\right)^\dagger U\lambda'(d)U^\dagger\lambda(c^{-1})\right) \\
        &= \tr\left(\lambda(a^{-1})^\dagger U \lambda'(b)^\dagger U^\dagger U\lambda'(d)U^\dagger\lambda(c^{-1})\right) \\
        &= \tr\left(\lambda(c^{-1}a)U\lambda'(b^{-1}d)U^\dagger\right)\\
        &= \tr\left(\lambda(a^{-1}c)^\dagger \left(\sum_{x \in G} U_{x,b^{-1}d} \lambda(x)\right) \right) \\
        &= U_{a^{-1}c,b^{-1}d}
    \end{align*}
    as desired. From this it follows that $(\ket{\psi_{a,b}})_{a \in G, b \in G'}$ is a $(G,G')$-invariant quantum Latin square with transformation matrix $U$.
\end{proof}


\section{Existence of transformation matrices}\label{sec:transexist}

In this section, we will prove our main result: that a $(G,G')$-invariant quantum Latin square exists if and only if the multiset of degrees of the irreducible representations is the same for $G$ and $G'$. Moreover, the $(G,G')$-transformation matrices are in correspondence with the isomorphisms from the group algebra of $G'$ to the group algebra of $G$ that commute with trace and conjugate transpose. Again, we refer the reader to \cite{james2001reps} for definitions of irreducible representations and characters and for theorems on orthogonality relations of irreducible characters.

Given a finite group $G$ with left regular representation $\lambda,$ the group algebra of $G$ is canonically isomorphic to the matrix algebra
\[\Lambda_G:= \spn\{\lambda(g): g\in G\}.\]
We will therefore think of the group algebra as this matrix algebra. It is a semi-simple algebra, and for finite $G$ it is Artinian.

Let $\rho_1,\dots, \rho_r$ be a complete set of irreducible representations of $G$ with degrees $d_1,\dots, d_r$. Then as a consequence of Schur's Lemma, we have that $\spn\{\rho_i(g): g\in G\}$ is the full matrix algebra, $M_{d_i}(\C)$. The following lemma is an immediate consequence of some well known results in representation theory (see for example Section 6.2 in \cite{serre-reps}) and on semisimple, Artinian algebras (see, for example Section 1.3 in \cite{BensonReps}).

\begin{lemma}\label{lem:Wedderburn}
    Let $G$ be a finite group with left regular representation $\lambda$ and irreducible, unitary representations $\rho_1,\dots, \rho_r$ with degrees $d_1\dots, d_r.$ Then the map 
    \[\Phi_G:\Lambda_G\to \bigoplus_{i=1}^r\big(I_{d_i}\otimes M_{d_i}(\C)\big),\quad \lambda(g)\mapsto \bigoplus_{i=1}^r\big(I_{d_i}\otimes \rho_i(g)\big) \]
    (extended linearly) is an algebra isomorphism that preserves trace and commutes with the conjugate transpose. Further, there exists a unitary matrix $U$ such that for any $M\in \Lambda_G,$ we have $\Phi_G(M) = U^\dagger M U$.\qed
\end{lemma}

\begin{definition}
    We will call an algebra isomorphism that preserves trace and commutes with the conjugate transpose a \emph{unitary isomorphism}, since (assuming some conditions on the algebra which always apply in our case) such an isomorphism may be written as conjugation by a unitary matrix. 
\end{definition}

Now the next two lemmas follow quite easily.

\begin{lemma}
    Two groups, $G$ and $G'$ have isomorphic group algebras if and only if their multisets of degrees of irreducible representations are the same. 
\end{lemma}
\begin{proof}
    By the Wedderburn-Artin Theorem \cite[Theorem 1.3.5]{BensonReps}, any semisimple Artinian algebra can be written as a direct sum of full matrix algebras whose dimensions are uniquely determined. Now the result follows from Lemma \ref{lem:Wedderburn}.
\end{proof}




\begin{lemma}\label{lem:groupalgisos}
    Let $G$ and $G'$ be groups with isomorphic group algebras, $\Lambda_G$ and $\Lambda_{G'}$. Then there exists a unitary isomorphism, $\Psi:\Lambda_G\to \Lambda_{G'}.$ 
\end{lemma}
\begin{proof}
    This is obvious by taking $\Psi:= \Phi_{G'}^{-1}\circ\Phi_G$.
\end{proof}
  
Finally, we identify the $(G,G')$-transformation matrices with the unitary isomorphisms of the group algebras.

\begin{theorem}\label{thm:U-iso}
    Let $G$ and $G'$ be finite groups. Then $U \in \mathbb{C}^{G \times G'}$ is a $(G,G')$-transformation matrix if and only if the linear map $\Psi_U : \Lambda_{G'} \to \Lambda_G$ given by $\Psi_U(\lambda'(b)) = \sum_{a \in G} U_{a,b} \lambda(a)$ is a unitary isomorphism. Moreover, every unitary isomorphism from $\Lambda_{G'}$ to $\Lambda_G$ is attained in this way.
\end{theorem}
\begin{proof}
    Let $U \in \mathbb{C}^{G \times G'}$. If $U$ is a transformation matrix, then $\Psi_U(M) = UMU^\dagger$ for all $M \in \Lambda_{G'}$ by Corollary~\ref{cor:transformchar}. It is then immediate that $\Psi_U$ is a unitary isomorphism.

    Conversely, suppose that $\Psi_U$ is a unitary isomorphism. Define $\rho'(b) = \Psi_U(\lambda'(b))$ for $b \in G'$. Since $\lambda'$ is a homomorphism and $\Psi_U$ is an isomorphism, it is clear that $\rho'$ is a homomorphism from $G'$ to $\mathbb{C}^{G \times G}$. We can further show that $\rho'$ is quasi-regular: since $\lambda'$ is unitary and $\Psi_U$ is unitary, $\rho'$ must also be unitary, and since $\Psi_U$ is unitary it preserves inner products and thus $\langle\rho'(g),\rho'(h)\rangle = \delta_{g,h}$.

    Now since $\rho'(b) = \Psi_U(\lambda'(b))\in \Lambda_G$ for all $b\in G',$ we see that $\spn\{\rho'(b):b\in G'\}\subseteq \Lambda_G = \spn\{\lambda(a):a\in G\}$ and by a dimension argument we see that the spans must be equal. We can then apply Lemma \ref{lem:reps} to obtain a $(G,G')$-transformation matrix $U'$ that satisfies \begin{equation}\label{eq:UPsi}
        \Psi_U(\lambda'(b)) = \rho'(b) = \sum_{a \in G} U'_{a,b} \lambda(a) \text{ for each } b \in G'.
    \end{equation}
    But then of course $U=U'$ and thus $U$ is a $(G,G')$-transformation matrix as desired.

    Lastly, suppose that $\Psi: \Lambda_{G'} \to \Lambda_G$ is a unitary isomorphism. Then simply from the fact that $\Psi$ is a linear map there exist coefficients $U_{a,b}$ for $a \in G$ and $b \in G'$ such that
    \[\Psi(\lambda'(b)) = \sum_{a \in G} U_{a,b} \lambda(a) \text{ for all } b \in G'.\]
    In other words, $\Psi = \Psi_U$ and by the above we have that $U$ must be a $(G,G')$-transformation matrix.    
\end{proof}



We now obtain the following result directly from the previous lemmas.
\begin{cor}\label{cor:existence}
    Let $G$ and $G'$ be finite groups. There exists a $(G,G')$-transformation matrix if and only if the multisets of degrees of the irreducible representations of $G$ and $G'$ are the same.\qed
\end{cor}

The corollary implies that if $G$ and $G'$ are groups of the same order and $G$ is abelian, then a $(G,G')$-transformation matrix exists if and only if $G'$ is abelian. We will further look into the abelian case in Section \ref{sec:abelian}.

Before moving on, we prove the following additional characterization of transformation matrices. Here we use $\mathbf{e}_g \in \mathbb{C}^G$ to denote the standard basis vector indexed by $g \in G$.

\begin{theorem}
    Let $G$ and $G'$ be finite groups. Then $U \in \mathbb{C}^{G \times G'}$ is a $(G,G')$-transformation matrix if and only if $U$ is unitary, $U\Lambda_{G'}U^\dagger = \Lambda_G$, and $U\mathbf{e}_e = \mathbf{e}_e$.
\end{theorem}
\begin{proof}
    If $U$ is a $(G,G')$-transformation matrix, then it is unitary and it follows from Corollary~\ref{cor:transformchar} that $U\Lambda_{G'}U^\dagger \subseteq \Lambda_G$. Then the equality follows from the fact that they have the same dimension, and $U\mathbf{e}_e = \mathbf{e}_e$ follows from Remark~\ref{rem:trans_mtx_extras}.

    Conversely, suppose that $U \in \mathbb{C}^{G \times G'}$ is a unitary matrix such that $U\Lambda_{G'}U^\dagger = \Lambda_G$ and $U\mathbf{e}_e = \mathbf{e}_e$. Then the map $\Psi : \Lambda_{G'} \to \Lambda_G$ given by $\Psi(M) = UMU^\dagger$ is a unitary isomorphism. Thus by Theorem~\ref{thm:U-iso} and Corollary~\ref{cor:transformchar} there exists a $(G,G')$-transformation matrix $\hat{U}$ such that $\Psi(M) = \hat{U}M\hat{U}^\dagger$ for all $M \in \Lambda_{G'}$. Therefore, $U^\dagger\hat{U}M\hat{U}^\dagger U = M$ for all $M \in \Lambda_{G'}$. In particular, this implies that $U^\dagger\hat{U} \lambda'(g) = \lambda'(g)U^\dagger \hat{U}$ for all $g \in G'$. Therefore,
    \begin{align*}
        U^\dagger \hat{U} \mathbf{e}_g & = U^\dagger \hat{U} \lambda'(g) \mathbf{e}_e \\
        &= \lambda'(g) U^\dagger \hat{U} \mathbf{e}_e \\
        &= \lambda'(g) \mathbf{e}_e \\
        &= \mathbf{e}_g.
    \end{align*}
    Thus $U^\dagger \hat{U} = I$, i.e., $U = \hat{U}$ and so $U$ is a $(G,G')$-transformation matrix.
\end{proof}

\section{Constructing transformation matrices}\label{sec:constructing}

In the previous section, we saw that the $(G,G')$-transformation matrices are in correspondence with the unitary isomorphisms from $\Lambda_{G'}$ to $\Lambda_G$, and this means that $(G,G')$-transformation matrices exist if and only if $G$ and $G'$ have the same multiset of degrees of their irreducible representations, as stated in Corollary~\ref{cor:existence}. However, even if we know this is the case for two groups $G$ and $G'$, it is not obvious how to construct explicit examples of $(G,G')$-transformation matrices. In this section, we will show how to do this.

If $G$ and $G'$ are finite groups whose multisets of degrees of irreducible representations are the same, then both $\Lambda_G$ and $\Lambda_{G}'$ are isomorphic to the algebra
\begin{equation}\label{eq:Ahat}
\hat{\A} := \bigoplus_{i=1}^r\big(I_{d_i}\otimes M_{d_i}(\C)\big),\end{equation}
where $d_1, \ldots, d_r$ are the degrees of their irreducible representations. Moreover, by Lemma~\ref{lem:Wedderburn} there are unitary isomorphisms from $\Lambda_G$ and $\Lambda_{G'}$ to $\hat{\A}$. In other words, there exist unitaries $V$ and $W$ such that
\begin{equation}\label{eq:conjugate}
    V^\dagger \Lambda_G V = \bigoplus_{i=1}^r\big(I_{d_i}\otimes M_{d_i}(\C)\big) = W^\dagger \Lambda_{G'} W.
\end{equation}
The idea for constructing $(G,G')$-transformation matrices then is to pick a unitary $\hat{U}$ such that $\hat{U}\hat{\A}\hat{U}^\dagger = \hat{\A}$ and then computing $V\hat{U}W^\dagger$. However, this does not work in general, essentially because there are nontrivial unitaries that pointwise fix $\hat{\A}$, and therefore a unitary isomorphism of $\hat{\A}$ can be written as conjugation by a unitary in multiple different ways. So we must choose the unitary $\hat{U}$ more carefully. This is what we show how to do in the remainder of this section. We remark here that even after we know how to choose our unitaries $\hat{U}$ appropriately, it is still necessary to know unitaries $V$ and $W$ satisfying Equation~\eqref{eq:conjugate} in order to produce $(G,G')$-transformation matrices. For abelian groups one can always use the character table (normalized so that it is unitary), but for non-abelian groups finding such $V$ and $W$ seems to be less straightforward.

Consider a unitary $V$ that satisfies~\eqref{eq:conjugate}. We can index its columns by triples $(i,\ell,k)$ where the first coordinate corresponds to the terms in the summand, the second coordinate corresponds to indices of the first tensor factor, and the third coordinate corresponds to indices of the second tensor factor. Let $v^i_{\ell,k}$ denote the $(i,\ell,k)$ column of $V$. Letting $\lambda$ be the left regular representation of $G$, Equation~\eqref{eq:conjugate} implies that there are coefficients $\alpha^{a,i}_{k,t}$ for $a \in G$, $i \in [r]$ and $k,t \in [d_i]$ such that
\begin{equation}\label{eq:lambdavilk}
    \lambda(a)v^{i}_{\ell, k} = \sum_{t=1}^{d_i} \alpha^{a,i}_{k,t} v^{i}_{\ell, t}
\end{equation}
Importantly, the coefficients appearing in this sum do not depend on $\ell$ (this reflects the fact that each term in the direct sum of Equation~\eqref{eq:conjugate} is a tensor product with the identity). Recall that we denote by $\mathbf{e}_g$ the standard basis vector indexed by the element $g\in G$. Define a vector $x = V^\dagger \mathbf{e}_e$. Thus the $(i,\ell,k)$-entry of $x$ is $x^{i}_{\ell, k} := (v^{i}_{\ell, k})^\dagger \mathbf{e}_e$. Fixing the first index $i$, we define a $d_i \times d_i$ matrix $X^i$ entrywise as

\begin{equation}\label{eq:Xdef}
X^i_{\ell,k} = x^{i}_{\ell, k} = (v^{i}_{\ell, k})^\dagger \mathbf{e}_{e}. 
\end{equation}

We will need the following lemma:

\begin{lemma}\label{lem:Xunitary}
Let $G$ be a finite group and let $V$ be a unitary satisfying Equation~\eqref{eq:conjugate}. Further let $X^i$ be defined as in Equation~\eqref{eq:Xdef} for each $i$ indexing the irreducible representations of $G$. Then the matrix $\sqrt{\frac{|G|}{d_i}}X^i$ is unitary.
\end{lemma}
\begin{proof}
We will show that $(X^i(X^i)^\dagger))_{\ell, k} = \delta_{\ell,k}\frac{d_i}{|G|}$, thus proving the lemma. We have that
\begin{align*}
    \left(X^i(X^i)^\dagger)\right)_{\ell, k} &= \sum_{j} X^i_{\ell,j}\overline{X^i_{k,j}} \\
    &= \sum_{j} \left((v^i_{\ell,j})^\dagger \mathbf{e}_e\right) \left(\overline{(v^i_{k,j})^\dagger \mathbf{e}_e}\right) \\
    &= \mathbf{e}_e^\dagger \left(\sum_{j} v^i_{k,j}(v^i_{\ell,j})^\dagger\right)\mathbf{e}_e
\end{align*}
To show that this is equal to $\delta_{\ell,k}\frac{d_i}{|G|}$, we define the $|G|\times|G|$ matrix $A^i_{\ell,k} = \sum_{j} v^i_{k,j} (v^i_{\ell,j})^\dagger$. So $\left(X^i(X^i)^\dagger)\right)_{\ell, k} = \mathbf{e}_e^\dagger A^i_{\ell,k}\mathbf{e}_e$. We will show that $A^i_{k,\ell}$ has constant diagonal and trace equal to $\delta_{\ell,k} d_i$, which will prove the lemma.

First, we note that
\begin{equation}\label{eq:Av}
    A^i_{\ell,k}v^j_{s,t} = \delta_{i,j}\delta_{\ell,s}v^i_{k,t}
\end{equation}
since the $v^j_{s,t}$ form an orthonormal basis. Therefore, the trace of $A^i_{\ell,k}$ is
\[\Tr(A^i_{\ell,k}) = \sum_{j,s,t}(v^j_{s,t})^\dagger A^i_{\ell,k} v^j_{s,t} = \sum_{j,s,t} \delta_{i,j}\delta_{\ell,s}(v^j_{s,t})^\dagger v^i_{k,t} = \sum_{t} (v^i_{\ell,t})^\dagger v^i_{k,t} = \delta_{\ell,k} d_i,\]
as desired.

To show that $A^i_{\ell,k}$ has constant diagonal, it suffices to show that $\lambda(a)^\dagger A^i_{\ell,k} \lambda(a) = A^i_{\ell,k}$ for all $a \in G$. We will show that $A^i_{\ell,k} \lambda(a) = \lambda(a) A^i_{\ell,k}$, which is equivalent. Using Equations~\eqref{eq:lambdavilk} and~\eqref{eq:Av}, we have that
\[A^i_{\ell,k} \lambda(a) v^j_{s,t} = \sum_{x} \alpha^{a,j}_{t,x} A^i_{\ell,k} v^j_{s,x} = \delta_{i,j}\delta_{\ell,s} \sum_{x} \alpha^{a,i}_{t,x} v^i_{k,x},\]
and
\[\lambda(a) A^i_{\ell,k} v^j_{s,t} = \delta_{i,j}\delta_{\ell,s} \lambda(a) v^i_{k,t} = \delta_{i,j}\delta_{\ell,s} \sum_{x} \alpha^{a,i}_{t,x} v^i_{k,x}.\]
Since $A^i_{\ell,k}\lambda(a)$ and $\lambda(a)A^i_{\ell,k}$ agree on a basis they are equal. Therefore $A^i_{\ell,k}$ has constant diagonal equal to $\Tr(A^i_{\ell,k})/ |G| = \delta_{\ell,k}\frac{d_i}{|G|}$. It follows that
\[\left(X^i(X^i)^\dagger)\right)_{\ell, k} = \mathbf{e}_e^\dagger A^i_{\ell,k}\mathbf{e}_e = \delta_{\ell,k}\frac{d_i}{|G|},\]
and thus $\sqrt{\frac{|G|}{d_i}}X^i$ is unitary.
\end{proof}

Next we show that we can always find a unitary $V$ satisfying~\eqref{eq:conjugate} such that the matrices $X^i$ defined above are proportional to the identity.

\begin{lemma}
Let $G$ be a finite group and suppose $d_1, \ldots, d_r$ are the degrees of its irreducible representations. Then there exists a unitary $V$ satisfying~\eqref{eq:conjugate} such that each $X^i$ defined as in Equation~\eqref{eq:Xdef} is equal to $\sqrt{\frac{d_i}{|G|}}I_{d_i}$.
\end{lemma}
\begin{proof}
    Let $V$ be any unitary satisfying Equation~\eqref{eq:conjugate}, and let $X^i$ be defined as in Equation~\eqref{eq:Xdef}. Then $\sqrt{\frac{|G|}{d_i}}X^i$ is unitary by Lemma~\ref{lem:Xunitary}. Thus
    \[X = \bigoplus_{i=1}^r \sqrt{\frac{|G|}{d_i}} X^i \otimes I_{d_i}\]
    is unitary. Define $\hat{V} := VX$. Since $V$ and $X$ are unitary, so is $\hat{V}$. Moreover, conjugation by $X$ clearly (pointwise) preserves the algebra $\hat{\A}$, and therefore $\hat{V}$ also satisfies Equation~\eqref{eq:conjugate}.

    Now let us consider the analog of the matrices $X^i$ but for $\hat{V}$, which we will call $\hat{X}^i$. Let $\mathrm{vec}$ be the linear map given by $\mathrm{vec}(e_ie_j^T) = e_i \otimes e_j$ (and therefore satisfies $\mathrm{vec}(uv^T) = u \otimes v$ for arbitrary vectors $u$ and $v$). It is well known and not difficult to see that
    
    \begin{equation}\label{eq:vecid}
        \left(A \otimes B\right) \mathrm{vec}(C) = \mathrm{vec}(ACB^T).
    \end{equation}
    
    By definition, $\mathrm{vec}(X^i)$ is the part of the $e$ column of $V^\dagger$ corresponding to the $i^\text{th}$ summand of~\eqref{eq:conjugate}, and $\mathrm{vec}(\hat{X}^i)$ is the same but for $\hat{V}$. Since $\hat{V}^\dagger = X^\dagger V^\dagger$, using~\eqref{eq:vecid} and the definition of $X$ we see that
    \[\mathrm{vec}(\hat{X}^i) = \left(\left( \sqrt{\frac{|G|}{d_i}} X^i\right)^\dagger \otimes I_{d_i}\right) \mathrm{vec}(X^i) = \mathrm{vec}\left(\left(\sqrt{\frac{|G|}{d_i}} X^i\right)^\dagger X^i\right) = \sqrt{\frac{d_i}{|G|}}\mathrm{vec}\left(I_{d_i}\right).\]
    Thus $\hat{X}^i = \sqrt{\frac{d_i}{|G|}} I_{d_i}$ as desired.
\end{proof}

We will soon show how to construct $(G,G')$-transformation matrices assuming we have unitaries $V$ and $W$ satisfying~\eqref{eq:conjugate}. But first we need to understand the structure of the unitaries that preserve the algebra $\hat{\A}$ from~\eqref{eq:Ahat}. For this, we need to define some notation. Letting $\hat{\A}$ be as in~\eqref{eq:Ahat}, define $\mathrm{Sym}_{\hat{\A}}$ to be the group of permutations $\pi$ of $[r]$ such that $d_{\pi(i)} = d_i$ for each $i \in [r]$. In other words, $\mathrm{Sym}_{\hat{\A}}$ is the group of permutations of $[r]$ that only permute summands of the same dimension. Now given $\pi \in \mathrm{Sym}_{\hat{\A}}$, define $\widehat{P}^\pi$ to be the $|G| \times |G|$ permutation matrix such that
\[\widehat{P}^\pi\left(\bigoplus_{i=1}^r I_{d_i} \otimes A_i\right) (\widehat{P}^\pi)^\dagger = \bigoplus_{i=1}^r I_{d_{\pi^{-1}(i)}} \otimes A_{\pi^{-1}(i)}\]
for arbitrary matrices $A_i \in M_{d_i}(\mathbb{C})$.


The following lemma classifies the unitaries $U$ such that $U\hat{\A}U^\dagger = \hat{\A}$. It is precisely the ones you expect, and the result is certainly known. However, the authors were unable to find a reference and so we include its rather tedious proof in Appendix~\ref{app:proof}.

\begin{lemma}\label{lem:hatAautos}
    Let
    \[\hat{\A} = \bigoplus_{i=1}^r I_{d_i} \otimes M_{d_i}(\mathbb{C}).\]
    Then a unitary $U$ satisfies $U\hat{\A}U^\dagger = \hat{\A}$ if and only if
    \[U = \widehat{P}^\pi \left(\bigoplus_{i=1}^r M^i \otimes N^i\right)\]
    where $\pi \in \mathrm{Sym}_{\hat{\A}}$, and $M^i$ and $N^i$ are $d_i \times d_i$ unitaries for each $i \in [r]$.
\end{lemma}

We are now able to prove the main result of this section.

\begin{theorem}\label{thm:construct}
    Let $G$ and $G'$ be finite groups with isomorphic group algebras and suppose that $V$ and $W$ are unitaries satisfying~\eqref{eq:conjugate}. For $i = 1, \ldots, r$, let $X^i$ be the matrices defined by Equation~\eqref{eq:Xdef}, and let $Y^i$ be the analogous matrices defined using $W$ instead of $V$. Then a matrix $U$ is a $(G,G')$-transformation matrix if and only if $U = V\hat{U}W^\dagger$ where
    \begin{equation}\label{eq:Uhatform}
        \hat{U} = \widehat{P}^\pi \left(\bigoplus_{i=1}^r M^i \otimes N^i\right)
    \end{equation}
    such that $\pi \in \mathrm{Sym}_{\hat{\A}}$, and $M^i$ and $N^i$ are $d_i \times d_i$ unitaries satisfying $M^i = X^{\pi(i)}\overline{N^i}({Y^{i}})^{-1}$.
\end{theorem}
\begin{proof}
    Let $\hat{\A} = \bigoplus_i I_{d_i} \otimes M_{d_i}(\mathbb{C})$ be the algebra in the center of Equation~\eqref{eq:conjugate}. We will make use of the characterization of transformation matrices from Corollary~\ref{cor:transformchar}. Let $U \in \mathbb{C}^{G \times G'}$ and define $\hat{U} = V^\dagger U W$. Then of course $U = V\hat{U}W^\dagger$ and $U$ is unitary if and only if $\hat{U}$ is unitary. Continuing in the case where they are unitary, we see that
    \begin{align*}
        \hat{U}\hat{\A}\hat{U}^\dagger = \hat{\A} \quad &\Leftrightarrow \quad V^\dagger U W \hat{\A}W^\dagger U^\dagger V = \hat{\A} \\
        &\Leftrightarrow \quad U W \hat{\A}W^\dagger U^\dagger = V \hat{\A} V^\dagger \\
        &\Leftrightarrow \quad U\Lambda_{G'}U^\dagger = \Lambda_G.
    \end{align*}
    By Lemma~\ref{lem:hatAautos}, this says that $U\Lambda_{G'}U^\dagger = \Lambda_G$ if and only if $\hat{U} = \widehat{P}^\pi \left(\oplus_{i=1}^r M^i \otimes N^i\right)$ for some $\pi \in \mathrm{Sym}_{\hat{\A}}$ and $d_i \times d_i$ unitaries $M^i$ and $N^i$ for all $i=1, \ldots, r$. Continuing in the case where this holds, we see that to prove the theorem we must show that for such $U$ and $\hat{U} = V^\dagger U W$, the condition $U\mathbf{e}_e = \mathbf{e}_e$ is equivalent to $M^i = X^{\pi(i)}\overline{N^i}\left(Y^i\right)^{-1}$ for each $i = 1, \ldots, r$. This we proceed to do.

    We have that $U\mathbf{e}_e = \mathbf{e}_e$ if and only if $V\hat{U}W^\dagger \mathbf{e}_e = \mathbf{e}_e$ if and only if $\hat{U}W^\dagger \mathbf{e}_e = V^\dagger\mathbf{e}_e$. Letting $x = V^\dagger \mathbf{e}_e$ and $y = W^\dagger \mathbf{e}_e$, the latter states that $\hat{U}y = x$. Letting $x^i$ and $y^i$ denote the portion of $x$ and $y$ corresponding to the $i^\text{th}$ summand in the definition of $\hat{U}$, and recalling the definitions of $X^i$ and $Y^i$, we see that $\mathrm{vec}(X^i) = x^i$ and $\mathrm{vec}(Y^i) = y^i$. The $\pi(i)^\text{th}$ block of the equation $\hat{U}y = x$ can be written as
    \[\left(M^i \otimes N^i\right)y^i = x^{\pi(i)}.\]
    Using~\eqref{eq:vecid}, we see that this is equivalent to $M^i Y^i (N^i)^T = X^{\pi(i)}$. By unitarity of $N^i$ this is equivalent to $M^i = X^{\pi(i)} \overline{N^i} (Y^i)^{-1}$, as desired.
\end{proof}

We remark that if $N^i$ is unitary then $M^i = X^{\pi(i)} \overline{N^i} (Y^i)^{-1}$ will necessarily be unitary since both $X^{\pi(i)}$ and $Y^{i}$ are equal to $\sqrt{\frac{d_i}{|G|}}$ times a unitary by Lemma~\ref{lem:Xunitary}. Thus the above theorem shows that for any choice of unitaries $N^i \in M_{d_i}(\mathbb{C})$ for $i \in [r]$ and permutation $\pi \in \mathrm{Sym}_{\hat{\A}}$, the matrix
\begin{equation}\label{eq:constructfinal}
    U = V \left( \widehat{P}^\pi \left( \bigoplus_{i=1}^r \left( X^{\pi(i)} \overline{N^i} (Y^i)^{-1} \right) \otimes N^i \right) \right) W^\dagger
\end{equation}
is a $(G,G')$-transformation matrix, and moreover all $(G,G')$-transformation matrices can be obtained in this way. Note however that the correspondence is not one-to-one, since multiplying any $N^i$ by a complex unit does not change $U$. However, it is clear that up to this degree of freedom, distinct choices for the $N^i$ and $\pi$ result in distinct $U$. This implies that when $G$ and $G'$ are non-abelian and have isomorphic group algebras there are uncountably many $(G,G')$-transformation matrices. This is because in this case at least one of the irreducible representations must have degree greater than one, and then there are uncountably many choices for the corresponding $N^i$. What is less clear is when the corresponding correlation matrices $U \circ \overline{U}$ are distinct. Since it is simple enough to choose some unitaries and a permutation, the above theorem allows us to construct arbitrarily many transformation matrices.

\section{The abelian case}\label{sec:abelian}

As noted at the end of Section~\ref{sec:transexist}, if $G$ is a finite abelian group, then there exist $(G,G')$-transformation matrices if and only if $G'$ is abelian and of the same order as $G$. In this case, using the results of the previous section allows us to describe the $(G,G')$-transformation matrices quite explicitly. In this section we give this description and additionally show that the correlation matrices resulting from these $(G,G')$-transformation matrices are closely related to the correlation matrices of the classical $(G,G')$-invariant correlations.

Given an abelian group $G$, we let $\widehat{G}$ denote the group of characters of $G$; it is isomorphic to $G$. We consider the character table of an abelian group to be a matrix with rows indexed by $\widehat{G}$ and columns by $G$ such that the $\chi,a$-entry is equal to $\chi(a)$. The \emph{normalized character table} has $\chi,a$-entry equal to $\frac{1}{\sqrt{|G|}}\chi(a)$ and is well known to be a unitary matrix. Using the results of the previous section, we prove the following theorem:

\begin{theorem}\label{thm:abeliancase}
    Let $G$ and $G'$ be equicardinal abelian groups with normalized character tables $\mathcal{C}$ and $\mathcal{C}'$ respectively. Then the set of $(G,G')$-transformation matrices is
    \[\left\{\mathcal{C}^\dagger P^\pi \mathcal{C}' : \pi \in \bij(\widehat{G},\widehat{G}')\right\}.\]
    In particular, this set is finite.
\end{theorem}
\begin{proof}
    Since all irreducible characters of abelian groups have degree one, the algebra $\hat{\A}$ of Equation~\eqref{eq:conjugate} is simply the algebra of diagonal matrices, and $\mathrm{Sym}_{\hat{\A}}$ is simply all permutations of $[n]$ for $n=|G|$. Moreover, since $d_i = 1$ for all $i$ in this case, the matrix $\widehat{P}^\pi$ is simply $P^\pi$.
    
    It is well known that the normalized character table satisfies Equation~\eqref{eq:conjugate}. Since every entry of the $e$-column of $\mathcal{C}^\dagger$ is $\frac{1}{\sqrt{|G|}}$, all of the $X^i$ defined as in Equation~\eqref{eq:Xdef} are proportional to identity, and similarly for the analogous matrices for $G'$. Therefore each $N^i$ in Equation~\eqref{eq:Uhatform} is simply a complex unit and $M^i = \overline{N^i}$. Therefore $M^i \otimes N^i = 1$ for all $i$. Thus the present theorem follows from Theorem~\ref{thm:construct}.
\end{proof}

\begin{remark}\label{rem:all}
    It is straightforward to check that the transformation matrix of the $(G,G')$-invariant quantum Latin square described in Example~\ref{ex:charconstruction} is precisely the matrix $\mathcal{C}^\dagger P^\pi \mathcal{C}'$ of the above theorem, where $\pi$ is the bijection taking $\chi'_i$ to $\chi_i$. Thus the construction described in that example, which is the generalization of the construction of~\cite[Definition 4.1]{quantumvsnonlocal} to the case of possibly different groups $G$ and $G'$, produces all $(G,G')$-invariant quantum Latin squares, up to isometry.
\end{remark}

As remarked at the end of the previous section, in the non-abelian case there are uncountably many $(G,G')$-transformation matrices, assuming there are any.


\begin{definition}\label{def:Qpi}
    Due to Theorem~\ref{thm:abeliancase}, for $\pi \in \bij(\widehat{G},\widehat{G}')$ we define $U^\pi := \mathcal{C}^\dagger P^\pi \mathcal{C}'$, and we let $Q^\pi := U^\pi \circ \overline{U^\pi}$ be the corresponding correlation matrix.
\end{definition}

Recall from Section~\ref{sec:GInvCorr} that for equicardinal groups $G$ and $G'$, every bijection $\pi \in \bij(G,G')$ results in a correlation matrix $D^\pi$ of the $(G,G')$-invariant correlation $p_G \circ p_\pi \circ p_{G'}$, and moreover the convex hull of these $D^\pi$ is the full set of correlation matrices of \emph{classical} $(G,G')$-invariant correlations. For abelian groups $G$ and $G'$, we see that we obtain a correlation matrix $Q^{\hat{\pi}}$ which corresponds to a $(G,G')$-invariant quantum Latin square for every bijection $\hat{\pi} \in \bij(\widehat{G},\widehat{G'})$, and moreover these are all of the correlation matrices of $(G,G')$-invariant quantum Latin squares. Remarkably, the matrices $Q^{\hat{\pi}}$ and $D^\pi$ are closely related:

\begin{theorem}\label{thm:C2Q}
    Let $G$ and $G'$ be finite abelian groups of the same order with normalized character tables $\mathcal{C}$ and $\mathcal{C}'$ respectively. Fix isomorphisms $\phi: \widehat{G} \to G$ and $\phi': \widehat{G}' \to G'$, and let $\Phi, \Phi'$ be the permutation matrices encoding these isomorphisms. Then for $\pi \in \bij(G,G')$ and $\hat{\pi} = \phi^{-1} \circ \pi \circ \phi' \in \bij(\widehat{G},\widehat{G}')$, we have that
    \[Q^{\hat{\pi}} = \mathcal{C}^\dagger\Phi^\dagger D^\pi \Phi'\mathcal{C}'.\]
\end{theorem}
\begin{proof}
    We begin by deriving a slightly different formula for $Q^{\hat{\pi}}$. First, define the $(G \times G) \times G$ matrix $S^G$ entrywise as
    \[S^G_{(a,b),c} = \begin{cases}
        1 & \text{if } a = b = c\\
        0 & \text{o.w.}
    \end{cases}\]
Define $S^{G'}$ analogously. It is then easy to check that for any two $G \times G'$ matrices $M$ and $N$, we have that
\[\left(S^G\right)^\dagger\left(M \otimes N\right)S^{G'} = M \circ N,\]
and therefore
\begin{align*}
Q^{\hat{\pi}} &= \left(S^G\right)^\dagger\left(U^{\hat{\pi}} \otimes \overline{U^{\hat{\pi}}}\right)S^{G'} \\
&= \left(S^G\right)^\dagger\left(\mathcal{C}^\dagger P^{\hat{\pi}} \mathcal{C}' \otimes \overline{\mathcal{C}^\dagger P^{\hat{\pi}} \mathcal{C}'}\right)S^{G'} \\
&= \left(S^G\right)^\dagger\left(\mathcal{C}^\dagger \otimes \overline{\mathcal{C}^\dagger}\right) \left(P^{\hat{\pi}} \otimes P^{\hat{\pi}}\right) \left(\mathcal{C}' \otimes \overline{\mathcal{C}'}\right)S^{G'} \\
&= \left(S^G\right)^\dagger\left(\mathcal{C}^\dagger \otimes \overline{\mathcal{C}^\dagger}\right) \left(\Phi^\dagger P^{\pi} \Phi' \otimes \Phi^\dagger P^{\pi} \Phi'\right) \left(\mathcal{C}' \otimes \overline{\mathcal{C}'}\right)S^{G'} \\
&= \left(S^G\right)^\dagger\left(\mathcal{C}^\dagger \otimes \overline{\mathcal{C}^\dagger}\right) \left(\Phi^\dagger \otimes \Phi^\dagger\right) \left(P^{\pi} \otimes P^{\pi} \right) \left(\Phi' \otimes \Phi' \right) \left(\mathcal{C}' \otimes \overline{\mathcal{C}'}\right)S^{G'}.
\end{align*}
On the other hand, Lemma~\ref{lem:redmtx} states that
\[D^\pi = \frac{1}{n} \left(R^G\right)^\dagger \left(P^\pi \otimes P^\pi \right) R^{G'},\]
where $n = |G| = |G'|$. Thus
\[\mathcal{C}^\dagger \Phi^\dagger D^\pi \Phi'\mathcal{C}' = \frac{1}{n} \mathcal{C}^\dagger \Phi^\dagger \left(R^G\right)^\dagger \left(P^\pi \otimes P^\pi \right) R^{G'} \Phi'\mathcal{C}'.\]
So it suffices to show that
\[(\Phi \otimes \Phi)(\mathcal{C} \otimes \overline{\mathcal{C}})S^G = \frac{1}{\sqrt{n}}R^G \Phi \mathcal{C},\]
or equivalently that
\[(\mathcal{C} \otimes \overline{\mathcal{C}})S^G\mathcal{C}^\dagger = \frac{1}{\sqrt{n}}(\Phi \otimes \Phi)^\dagger R^G \Phi.\]
It is easy to see from the definition of $R^G$ and the fact that $\Phi$ is the permutation matrix encoding an isomorphism $\varphi: \widehat{G} \to G$, that
\[\frac{1}{\sqrt{n}}(\Phi \otimes \Phi)^\dagger R^G \Phi = \frac{1}{\sqrt{n}} R^{\widehat{G}}. \]

On the other hand, for $\alpha, \beta, \chi \in \widehat{G}$, we have that
\begin{align*}
    \left(\left(\mathcal{C} \otimes \overline{\mathcal{C}}\right) S^G \mathcal{C}^\dagger\right)_{(\alpha,\beta),\chi} &= \sum_{a,b,c \in G} \mathcal{C}_{\alpha,a} \overline{\mathcal{C}_{\beta,b}} S^G_{(a,b),c} \overline{\mathcal{C}_{\chi,c}} \\
    &= \frac{1}{n^{3/2}} \sum_{a \in G} \alpha(a)\overline{\beta(a)}\overline{\chi(a)} \\
    &= \frac{1}{n^{3/2}} \sum_{a \in G} \beta(a)^{-1}\alpha(a)\overline{\chi(a)} \\
    &= \frac{1}{n^{3/2}} \sum_{a \in G} \left(\beta^{-1}\alpha(a)\right)\overline{\chi(a)} \\
    &= \begin{cases}
        \frac{1}{\sqrt{n}} & \text{if } \beta^{-1}\alpha = \chi \\
        0 &\text{otherwise}
    \end{cases}
\end{align*}
which is equal to $(1/\sqrt{n}) R^{\widehat{G}}_{(\alpha,\beta),\chi}$ and therefore we are done.
\end{proof}

The above shows that $\conv\{Q^{\hat{\pi}} : \hat{\pi} \in \bij(\widehat{G},\widehat{G}')\}$ and $\conv\{D^\pi : \pi \in \bij(G, G')\}$ are related by a unitary map (in particular they have the same volume and this map takes extreme points to extreme points). As a consequence, neither of these sets contains the other unless they are equal. In the case where they are not equal (which appears to be most cases) this implies that the group invariant quantum Latin squares produce at least some non-classical correlations, but their convex hull is not the set of all  of the quantum group invariant correlations as this convex hull does not even contain every classical correlation.

\section{Isomorphisms}\label{sec:isomorphisms}

Here we show that the collection of rows (resp.\ columns) of a $(G, G')$-transformation matrix $U$ containing only a single nonzero entry correspond to subgroups of $G$ (resp.\ $G'$). Moreover, the submatrix of $U$ corresponding to these rows and columns gives an isomorphism between these subgroups and a representation of degree one of each. 

A character of a group is called \emph{linear} if the degree of the corresponding representation is one. Note that in this case, the character is equal to the representation.

\begin{theorem}\label{thm:submtxiso}
    Suppose that $U$ is a $(G,G')$-transformation matrix and let $S \subseteq G$ and $S' \subseteq G'$ be the indices of all of the rows/columns of $U$ that have a single nonzero entry. Further define $\sigma: S \to S'$, $\chi:S \to \mathbb{C}$ and $\chi': S' \to \mathbb{C}$ as
    \begin{align*}
        U_{a,\sigma(a)} &\ne 0 \\
        \chi(a) &= U_{a,\sigma(a)} \\
        \chi'(b) &= U_{\sigma^{-1}(b),b}
    \end{align*}
    Then $S$, $S'$ are subgroups of $G$, $G'$ respectively, $\sigma$ is an isomorphism of these subgroups, and $\chi$, $\chi'$ are linear characters of $S$, $S'$.
\end{theorem}
\begin{proof}
Note first that $\sigma$ is clearly bijective and $|S|=|S'|.$
    We have seen before that $U_{e,e}=1$, so $e\in S$ with $\sigma(e) =e\in S'.$ Let $a\in S$ and $a':=\sigma(a)$. Then $|U_{a,b}|=0$ for all $b\in G'\sm\{a'\}$ and so
    $|U_{a^{-1},b^{-1}}| = |\overline{U_{a,b}}| = 0$ unless $b^{-1}= (a')^{-1}$. It follows that $a^{-1}\in S$ and that $\sigma(a^{-1}) = \sigma(a)^{-1}$. Similarly we see that $S'$ is inverse-closed. Now let $a,b\in S$ with $a':=\sigma(a)$ and $b':=\sigma(b).$ Then 
    \begin{equation}\label{eq:chariso}
        U_{ab,c} =\sum_{\substack{x,y\in G'\\xy = c}}U_{a,x}U_{b,y} = 
        \begin{cases}
            U_{a,a'}U_{b,b'}& \text{if } c=a'b'\\
            0&\text{otherwise.}
        \end{cases}
    \end{equation}
    Therefore, $ab\in S$ with $\sigma(ab) =a'b' = \sigma(a)\sigma(b)\in S'$ and we have shown that $S$ and $S'$ are subgroups of $G$ and $G',$ respectively and that $\sigma$ is an isomorphism between them. 

    To show that $\chi$ and $\chi'$ are linear characters of $G$ and $G'$, it suffices to show that they are homomorphisms to the complex unit circle. Clearly we have $|\chi(g)|=|\chi'(g')|=1$ for all $g\in G$ and $g'\in G'$ and the fact that they are homomorphisms follows easily from~\eqref{eq:chariso}. This concludes the proof.
\end{proof}
In the case where $G$ and $G'$ can be written as internal direct products containing the subgroups $S$ and $S'$, respectively, a version of the converse also holds. By Corollary \ref{cor:existence}, the existence of a $(G,G')$-transformation matrix depends on the degrees of the representations of the groups being the same, so we need that condition.

\begin{theorem}\label{thm:subgroupsiso}
    Let $G$ and $G'$ be groups, and let $H\leq G$ and $H'\leq G'$ be subgroups such that $H\cong H',$ $G\cong H\times K$ and $G'\cong H'\times K'$ for some $K,K'$. Suppose further that the multisets of degrees of irreducible representations of $G$ and $G'$ are the same. Let $\sigma:H\to H'$ be an isomorphism. Then for each linear character $\chi$ of $H$, there exists a $(G,G')$-transformation matrix, $U$ such that for all $h\in H$, we have $U_{h,x} = \delta_{x,\sigma(h)}\chi(h)$.
\end{theorem}
\begin{proof}
    Let $\chi$ be a linear character of $H$ and define $\chi':H'\to \C$ by letting $\chi'(h):=\chi(\sigma^{-1}(h))$. Since $\sigma$ is an isomorphism, it is easy to see that $\chi'$ is a linear character of $H'.$ Define the matrix $V\in\C^{H\times H'}$ by letting 
    \[V_{h,h'}:=
    \begin{cases}
        \chi(h)&\text{if } h'=\sigma(h)\\
        0&\text{otherwise.}
    \end{cases}\]
    Since $|\chi(h)| = 1$ for all $h \in H$, it is easy to see that $V$ is a unitary matrix.

    Now consider the groups $K$ and $K'$. It is well known that the irreducible characters of a direct product of groups are precisely the Kronecker products of the irreducible characters of the factors (see for example \cite[Theorem 19.18]{james2001reps}). It follows that the multisets of degrees of the irreducible representations of $K$ and $K'$ are the same. Therefore, there exists a $(K,K')$-transformation matrix, $W$. We will show that that the matrix $U:=V\otimes W$ is a $(G,G')$-transformation matrix satisfying the required conditions. Clearly, $U\in \C^{(H\times K)\times (H'\times K')} = \C^{G\times G'}$, and it is a unitary matrix since $V$ and $W$ are unitary.

    Note that since $G$ is the internal direct product of $H$ and $K$, we have for all $h\in H$ and $k\in K$ that $hk = kh$. Further, every element, $g$, of $G$ can be written uniquely as a product $g=hk$ with $h\in H$ and $k\in K$. We have for $g=hk\in G$ and $g'=h'k'\in G'$
    \begin{align*}
        U_{g^{-1},g'^{-1}} & = U_{(hk)^{-1},(h'k')^{-1}}\\ 
        & = U_{h^{-1}k^{-1},h'^{-1}k'^{-1}} & (\text{since } hk = kh)\\
        & = V_{h^{-1},h'^{-1}} W_{k^{-1},k'^{-1}} \\ 
        &= V_{h^{-1},h'^{-1}} \overline{W_{k,k'}},
    \end{align*}
    which is zero if $h'^{-1} \neq \sigma(h^{-1}) = \sigma(h)^{-1}$ or equivalently if $\sigma(h) \neq h'$. If $h' = \sigma(h)$ then 
    \[V_{h^{-1},h'^{-1}} = \chi(h^{-1}) = \chi(h)^{-1} = \overline{\chi(h)} = \overline{V_{h,h'}}.\]
    In both cases we see that $U_{g^{-1},g'^{-1}} = \overline{U_{g,g'}}$.

    Now let 
    $a=a_1a_2, b=b_1b_2\in G$ with $a_1,b_1\in H$ and $a_2,b_2\in K$ and let $c = c_1c_2\in G'$ with $c_1\in H'$ and $c_2\in K'$. Then
    \begin{align*}
        \sum_{\substack{x,y\in G'\\xy = c}}U_{a,x}U_{b,y} & = \sum_{\substack{x_1x_2,y_1y_2\in G'\\x_1x_2y_1y_2 = c}}U_{a_1a_2,x_1x_2}U_{b_1b_2,y_1y_2}\\ 
        & = \sum_{\substack{x_1x_2,y_1y_2\in G'\\x_1y_1x_2y_2 = c}} V_{a_1,x_1}W_{a_2,x_2}V_{b_1,y_1}W_{b_2,y_2} \\ 
        & = \sum_{\substack{x_2,y_2\in K'\\\sigma(a_1)\sigma(b_1)x_2y_2 = c}} \chi(a_1)\chi(b_1) W_{a_2,x_2}W_{b_2,y_2}\\ 
        & = \chi(a_1)\chi(b_1)\sum_{\substack{x_2,y_2\in K'\\\sigma(a_1b_1)= c_1\\ x_2y_2 =c_2}}  W_{a_2,x_2}W_{b_2,y_2}\\
        & = 
        \begin{cases}
            \chi(a_1b_1)W_{a_2b_2,c_2} &\text{if } \sigma(a_1b_1) = c_1\\
            0&\text{otherwise.}
        \end{cases}
    \end{align*}
    On the other hand we have 
    \[U_{ab,c} = V_{a_1b_1,c_1}W_{a_2b_2,c_2} = \chi(a_1b_1)W_{a_2b_2,c_2},\]
    if $c_1=\sigma(a_1b_1)$ and $0$ otherwise. Thus 
    \[U_{ab,c} = \sum_{\substack{x,y\in G'\\xy = c}}U_{a,x}U_{b,y}\]
    for all $a,b\in G$ and $c\in G'$ and we have shown that $U$ is a $(G,G')$-transformation matrix.
\end{proof}
\begin{remark}
    Note that the trivial character is linear, and so every group has a linear character.
\end{remark}



\begin{lemma}
    Suppose that $G$ and $G'$ are finite groups that are isomorphic via a map $\sigma: G \to G'$, and $\chi: G \to \mathbb{C}$ is a linear character of $G$. Then $U \in \mathbb{C}^{G \times G'}$ defined as
    \[U_{a,b} = \begin{cases}\chi(a) & \text{if } b = \sigma(a) \\ 0 & \text{o.w.} \end{cases}\]
    is a $(G,G')$-transformation matrix. Moreover, these are precisely the transformation matrices of $(G,G')$-invariant quantum Latin squares whose corresponding quantum permutation matrices have commuting entries. 
\end{lemma}
\begin{proof}
    The fact that such $U$ are transformation matrices follows immediately from Theorem~\ref{thm:subgroupsiso}. Moreover, if $U$ is such a transformation matrix, then its corresponding correlation matrix $U \circ \overline{U}$ is a permutation matrix and thus it follows from Corollary~\ref{cor:qperms} that all entries of the corresponding quantum permutation matrix commute.

    Conversely, suppose that $\Psi = (\psi_{a,b})$ is a $(G,G')$-invariant quantum Latin square such that all of the projections $\outerp{\psi_{a,b}}{\psi_{a,b}}$ commute. It follows that any two vectors in $\Psi$ are either orthogonal to each other or they differ only by multiplication by a complex unit. Therefore, the transformation matrix $U$ of $\Psi$ has precisely one nonzero entry in every row/column and thus by Theorem~\ref{thm:submtxiso} it has the form given in the lemma statement.
\end{proof}

Recall from Corollary~\ref{cor:qperms} that the only permutation matrices contained in the set of correlation matrices of quantum $(G,G')$-invariant correlations are the isomorphisms from $G'$ to $G$. Using the above, we can show that these can all be obtained from $(G,G')$-invariant quantum Latin squares.

\begin{cor}
    The permutation matrices contained in the set
    \[\{U \circ \overline{U} : U \text{ is a } (G, G')\text{-transformation matrix}\}\]
    are precisely the permutation matrices encoding isomorphisms between $G$ and $G'$.\qed
\end{cor}

\section{Support Graphs}\label{sec:supportgraphs}


In this section we will introduce the notion of the \emph{support graph} of a correlation for the bijection game and see some connections between properties of this graph and properties of the correlation. We will start with general non-signalling correlations, then specialize to quantum correlations, and finally to group invariant correlations. We will see that the support graph of a correlation $p$ arising from a group invariant quantum Latin square $\Psi$ contains interesting information about $\Psi$. We begin with the definition of a support graph:

\begin{definition}
    Given a winning correlation $p$ for the $(G,G')$-bijection game, the \emph{support (di)graph} of $p$, denoted $X^p$, is the (di)graph with vertex set $\{(a,b) \in G \times G' : p(a,a|b,b) \ne 0\}$ such that there is an arc from $(a,b)$ to $(c,d)$ if they are distinct and $p(a,c|b,d) \ne 0$.
\end{definition}

Note that the sets $\{(x,b) : x \in G\} \cap V(X^p)$ and $\{(a,y) : y \in G'\} \cap V(X^p)$ are independent sets in the support graph for all $a \in G$ and $b \in G'$, since $p(x,x'|b,b) = 0$ if $x \ne x'$ and $p(a,a|y,y') = 0$ if $y \ne y'$. For quantum correlations the support (di)graph is in fact always a graph by Equation~\eqref{eq:qcorr} and the cyclicity of trace. 

\subsection{Support graphs of non-signalling correlations}\label{subsec:nonsigsupp}

Recall that a correlation $p$ is \emph{non-signalling} if both $\sum_{x} p(a,x|b,y)$ and $\sum_x p(x,a|y,b)$ are independent of $y$. If $p$ is a winning non-signalling correlation for the $(G,G')$-bijection game, then both of these so-called ``marginal probabilities" are equal to $p(a,a|b,b)$. In other words, Alice and Bob have the same marginal probabilities and so we may simply denote this by $p(a|b)$. This fact will be useful in the proofs of Lemma~\ref{lem:suppflow} and Theorem~\ref{thm:NSred}.

A common alternative term for ``classical correlation" is ``local correlation", and so a nonlocal correlation is simply a non-classical correlation. A correlation $\hat{p}$ is said to be \emph{strongly nonlocal}~\cite{strongnonlocal} if there is no classical correlation $p$ such that $p(a,b|x,y) > 0 \Rightarrow \hat{p}(a,b|x,y)$ for all $a,b,x,y$. In other words, $\hat{p}$ is strongly nonlocal if there is no classical correlation that is zero everywhere $\hat{p}$ is zero. Suppose that $\hat{p}$ is strongly nonlocal. We can then define a nonlocal game such that Alice and Bob win when responding with $a$ and $b$ upon inputs $x$ and $y$ if $\hat{p}(a,b|x,y) > 0$. Clearly, $\hat{p}$ will win this game with probability one. Moreover, any correlation that wins this game perfectly must be zero everywhere $\hat{p}$ is zero, and thus no classical correlation can win the game since $\hat{p}$ is strongly nonlocal. Conversely, if $\hat{p}$ can perfectly win some nonlocal game that cannot be won perfectly by any classical correlation, this implies $\hat{p}$ is strongly nonlocal. Thus the strongly nonlocal correlations are precisely those that can win some nonlocal game that cannot be won by any classical correlation. The following lemma shows that strong nonlocality of a group invariant correlation is captured by its support graph. Here we say a subset $C$ of vertices in a digraph is a clique if there is an arc from $x$ to $y$ for every distinct $x,y \in C$. 

\begin{lemma}\label{lem:strongnonlocal}
    Let $p$ be a winning correlation for the $(G,G')$-bijection game. Then $p$ is strongly nonlocal if and only if its support (di)graph has no clique of size $|G|$.
\end{lemma}
\begin{proof}
    Let $X$ be the support graph of $p$. We will prove that $p$ is not strongly nonlocal if and only if $X$ has a clique of size $|G|$.

    Suppose that $C$ is a clique of size $|G|$ in $X$. Since no two vertices $(a,b)$ and $(c,d)$ that agree in exactly one coordinate can be adjacent, there must exist a bijection $\pi \in \bij(G,G')$ such that $C = \{(\pi(y),y): y \in G'\}$. Then $p_\pi(a,b|c,d)$ is nonzero only when $a = \pi(c)$ and $b = \pi(d)$, i.e., only when $(a,c), (b,d) \in C$. In this case, by definition of the support (di)graph, we have that $p(a,b|c,d)$ is nonzero. Thus $p_\pi$ is a classical correlation that is zero everywhere $p$ is zero, i.e., $p$ is not strongly nonlocal.

    Conversely, suppose that $p$ is not strongly nonlocal. Then there exists a classical correlation that is zero everywhere $p$ is zero, and by convexity there then must exist a \emph{deterministic} classical correlation with this property. In other words, there are functions $f_A, f_B : G' \to G$ such that $p(f_A(x),f_B(y)|x,y) > 0$ for all $x,y \in G'$. However, since $p$ is a winning correlation for the $(G,G')$-bijection game, it follows that $f_A = f_B$ and this function is a bijection $\pi \in \bij(G,G')$. It is then easy to see that $C := \{(\pi(y),y) : y \in G'\}$ is a clique of size $|G|$ in the support (di)graph of $p$.
\end{proof}

\begin{remark}
    The above lemma is very closely related to the result of~\cite{MRV} showing that the existence of a winning classical strategy for a ``synchronous" nonlocal game is equivalent to the existence of an independent set of a certain size in its ``game graph". The main difference is that in~\cite{MRV} one constructs the graph from a game rather than a correlation, and the graph is essentially the complement of the support graph defined here.
\end{remark}

In Section~\ref{subsec:Z24comps}, we will give an example of a $\mathbb{Z}_2^4$-invariant quantum Latin squares whose correlations are strongly nonlocal, answering a question of~\cite{quantumvsnonlocal}.

We will soon show that the connected components of the support digraph can be used to decompose a non-signalling correlation into ``smaller" correlations. First we will show that all of the connected components of the support digraph of a non-signalling correlation are strongly connected. We will do this by showing that the values a correlation takes give a flow on the support digraph. Recall that a real-valued flow in a digraph is an assignment of real numbers to its arcs such that for each vertex $v$ the sum of the values assigned to the arcs leaving $v$ is equal to the sum of the values assigned to the arcs entering $v$. 

\begin{lemma}\label{lem:suppflow}
    Let $p$ be a winning non-signalling correlation for the $(G,G')$-bijection game. Then assigning the value $p(a,c|b,d)$ to the arc from $(a,b)$ to $(c,d)$ in the support digraph gives a real-valued nowhere-zero flow in $X^p$. As a consequence, every weakly connected component of $X^p$ is strongly connected.
\end{lemma}
\begin{proof}
    First note that by definition of the support graph, the assignment described is nowhere-zero. We will show that the total flow leaving a vertex is equal to the total flow entering a vertex, which proves that the assignment is a flow. The total flow leaving vertex $(a,b)$ is
    \[\sum_{y \in G'} \sum_{x \in G} p(a,x|b,y) = \sum_{y \in G'} p(a|b) = |G'|p(a|b).\]
    Similarly, the total flow entering vertex $(a,b)$ is
    \[\sum_{y \in G'} \sum_{x \in G} p(x,a|y,b) = \sum_{y \in G'} p(a|b) = |G'|p(a|b).\]
    Therefore the assignment is a nowhere-zero flow.

    Now suppose that $Y$ is a weakly connected component of $X^p$. If $Y$ is not strongly connected, then its vertex set can be partitioned into two nonempty parts $A$ and $B$ such that there are arcs going from $A$ to $B$, but no arcs going from $B$ to $A$. However, since the flow described above is positive on all arcs, this would imply that the net flow leaving $A$ is nonzero, a contradiction. Therefore, every weakly connected component of $X^p$ is strongly connected.
\end{proof}

\begin{theorem}\label{thm:NSred}
    Let $p$ be a winning non-signalling correlation for the $(G,G')$-bijection game, and let $V_1, V_2, \ldots, V_k$ be the vertex sets of the connected components of its support graph $X^p$. Then there exist $s_1, \ldots, s_k \in \mathbb{R}^+$ such that for all $i=1, \ldots, k$
    \[\sum_{a,c \in G: (a,b),(c,d) \in V_i} p(a,c|b,d) = s_i \text{ for all } b,d \in G'.\]
    It then follows that functions $p_1, \ldots, p_k : G \times G \times G' \times G' \to \mathbb{R}$ defined as
    \[p_i(a,c|b,d) = \begin{cases}
        \frac{1}{s_i}p(a,c|b,d) & \text{if } (a,b),(c,d) \in V_i \\
        0 & \text{o.w.}
    \end{cases} \]
    are \emph{distinct} non-signalling winning correlations for the $(G,G')$-isomorphism game and moreover $p$ is equal to the convex combination $\sum_{i=1}^k s_ip_i$. It follows that $p$ is non-classical if and only if at least one of the $p_i$ are non-classical, and $p$ is strongly nonlocal if and only if all of the $p_i$ are strongly nonlocal.
\end{theorem}
\begin{proof}
    For any $b,d \in G'$, we have that
    \begin{align}\label{eq:Visum}
        \sum_{a,c \in G: (a,b),(c,d) \in V_i} p(a,c|b,d) &= \sum_{a \in G: (a,b) \in V_i}\sum_{c \in G: (c,d) \in V_i} p(a,c|b,d) \\
        &= \sum_{a \in G: (a,b) \in V_i}\sum_{c \in G} p(a,c|b,d) \\
        &= \sum_{a \in G: (a,b) \in V_i} p(a|b).
    \end{align}
    So we see that the sum only depends on $b$. However, it can be analogously shown that the sum only depends on $d$, and therefore the sum must be equal to some constant $s_i$ for all $b,d \in G'$, as desired. To see that $s_i > 0$, note that $p(a|b) = p(a,a|b,b) > 0$ for all $(a,b) \in V(X^p)$ by definition of the support graph.

    We now let $p_1, \ldots, p_k$ be defined as in the theorem statement. These are all correlations, since they only take nonnegative values and
    \[\sum_{a,c \in G} p_i(a,c|b,d) = \frac{1}{s_i}\sum_{a,c \in G: (a,b),(c,d) \in V_i} p(a,c|b,d) = \frac{1}{s_i}s_i = 1.\]
    Also note that Equation~\eqref{eq:Visum} shows that they are non-signalling. Lastly, they win the $(G,G')$-bijection game because they are zero everywhere $p$ is zero. They are also clearly distinct correlations since $p_i(a,a|b,b) \ne 0$ if and only if $(a,b) \in V_i$.

    The fact that $p$ is equal to the convex combination of $p_1, \ldots, p_k$ given in the theorem statement is easy to check.

    Since $p$ is a convex combination of the $p_i$, it immediately follows that if they are all classical then $p$ must be classical. Now suppose that $p$ is classical. Then $p = \sum_{\pi \in \bij(G,G')} \alpha_\pi p_\pi$ for some nonnegative coefficients $\alpha_\pi$ that sum to 1. It is easy to see that $X^p$ is the union of the support digraphs $X^{p_\pi}$ such that $\alpha_\pi > 0$. Thus every $X^{p_{\pi}}$ is a subgraph of $X^p$. Also, it is straightforward to check that $X^{p_\pi}$ is simply the complete graph on the vertex set $\{(\pi(b),b) : b \in G'\}$. Therefore, each $X^{p_\pi}$ with $\alpha_\pi > 0$ is a subgraph of one of the connected components of $X^p$. So for each $i \in [k]$, define $S_i = \{\pi \in \bij(G,G') :  V(X^{p_\pi}) \cap V_i \ne \varnothing\}$. We will show that for all $i \in [k]$,
    \[p_i = \left(\sum_{\pi \in S_i} \alpha_\pi \right)^{-1} \sum_{\pi \in S_i} \alpha_\pi p_\pi.\]
    Without loss of generality we can consider the $i = 1$ case. Let $p'_1$ denote the righthand side of the above equation and note that it is a convex combination of classical correlations and therefore is a classical correlation. For $\pi \in S_1$, we have that $V(X^{p_\pi})$ nontrivially intersects $V_1$ and therefore it is contained in $V_1$ since $X^{p_\pi}$ is a connected subgraph of $X^p$. Therefore, if $(a,b) \notin V_1$, then $a \ne \pi(b)$ and thus $p_\pi(a,c|b,d) = 0$ for all $c \in G$ and $d \in G'$. Similarly, $p_\pi(a,c|b,d) = 0$ if $(c,d) \notin V_1$. Thus it only remains to check that $p_1(a,c|b,d) = p'_1(a,c|b,d)$ for $(a,b),(c,d) \in V_1$. In this case, we have that $p_\pi(a,c|b,d) = 0$ if $\pi \notin S_1$ by the same argument as above. Therefore, for $(a,b),(c,d) \in V_1$ we have that
    \[\sum_{\pi \in S_1} \alpha_\pi p_\pi(a,c|b,d) = \sum_{\pi \in \bij(G,G')} \alpha_\pi p_\pi(a,c|b,d) = p(a,c|b,d) = s_1 p_1(a,c|b,d).\]
    Thus $p'_1 = s_i\left(\sum_{\pi \in S_i} \alpha_\pi \right)^{-1}p_1$. However, since both $p'_1$ and $p_1$ are correlations, the coefficient on the righthand side must be equal to 1 and therefore we have our desired equality. So we have shown that $p_1$ (and thus all $p_i$) are a convex combination of classical correlations and is therefore classical. Thus $p$ is classical if and only if every $p_i$ is classical as desired.
    
    Lastly, note that the support graph of $p_i$ is simply the connected component of $X^p$ with vertex set $V_i$. Since a digraph has a clique of a given size if and only if one of the connected components has a clique of this size, the final claim of the theorem follows from Lemma~\ref{lem:strongnonlocal}.
\end{proof}

    


One of the consequences of the above theorem is that if the support graph of a winning non-signalling correlation for the $(G,G')$-bijection game is not connected, then it is not an extreme point of the set of all winning non-signalling correlations for the $(G,G')$-bijection game, since it can be written as a nontrivial convex combination of other elements of this set. Note that the converse is not true. For example, taking a uniform combination of all deterministic classical winning correlations for the $(G,G')$-bijection game yields a classical (and thus non-signalling) correlation $p$ whose support graph is isomorphic to the complement of the cartesian product $K_{|G|} \square K_{|G|}$, which is connected for $|G| > 2$.


Before moving on to quantum correlations, we prove the following about the size of the connected components of a support graph:

\begin{lemma}\label{lem:cliques}
    Let $p$ be a winning non-signalling correlation for the $(G,G')$-bijection game, let $V_1, \ldots, V_k$ be the vertex sets of the connected components of $X^p$, and let $p_1, \ldots, p_k$ be the correlations defined in Theorem~\ref{thm:NSred}. Then, for all $i \in [k]$, we have that $|V_i| \ge |G|$ and the following are equivalent
    \begin{enumerate}
        \item $|V_i| = |G|$;
        \item $V_i$ induces a clique in $X^p$;
        \item $p_i$ is a deterministic classical correlation, i.e., $p_i = p_\pi$ for some $\pi \in \bij(G,G')$.
    \end{enumerate}
\end{lemma}
\begin{proof}
    By Theorem~\ref{thm:NSred}, it suffices prove the lemma in the case where $X^p$ is connected. Let $n = |G| = |G'|$. First, we show that the out/in-degree of every vertex in $X^p$ is at least $n-1$. Let $(a,b) \in V(X^p)$, i.e., $p(a|b) = p(a,a|b,b) \ne 0$. Then for each $d \in G'$, we have that $\sum_{c \in G} p(a,c|b,d) = p(a|b) \ne 0$, and therefore there exists some $c \in G$ such that there is an arc from $(a,b)$ to $(c,d)$. So the out-degree of $(a,b)$ is at least $n-1$, and the same holds for the in-degree by an analogous argument. It immediately follows that $|V(X^p)| \ge n$, and moreover if equality holds then we must have that $X^p$ is a clique, so we have shown $(1) \Rightarrow (2)$. Conversely, any clique in $X^p$ has size at most $n$, since $X^p$ can be partitioned into the $n$ independent sets $\{\{(a,b) : b \in G'\} : a \in G\}$. Thus if $X^p$ is a clique, then it contains exactly $n$ vertices, i.e., $(2) \Rightarrow (1)$.
    
    Now suppose that $(1)$ and $(2)$ hold. Then there exists $\pi \in \bij(G,G')$ such that $V(X^p) = \{(\pi(b),b) : b \in G'\}$. Therefore, by definition of support graph, $p(a,c|b,d) \ne 0$ if and only if $a = \pi(b)$ and $c = \pi(d)$. This moreover implies that $p(\pi(b),\pi(d)|b,d) = 1$ for all $b,d \in G'$. Thus $p = p_\pi$ as desired. The converse was already shown in the proof of Theorem~\ref{thm:NSred}.
\end{proof}

\subsection{Support graphs of quantum correlations}\label{subsec:qsupp}

We now will focus on support graphs of quantum correlations for the bijection game, which always arise from quantum permutation matrices up to convex combinations. We begin by showing that the support graph of such a correlation can be described in terms of the quantum permutation matrix.

\begin{lemma}\label{lem:qpermsupp}
    Let $\mathcal{Q} = (q_{a,b})_{a \in G, b \in G'}$ be a quantum permutation matrix and let $p$ be the corresponding correlation, i.e.,
    \[p(a,c|b,d) = \tr(q_{a,b}q_{c,d}).\]
    Then $(a,b) \in V(X^p)$ if and only if $q_{a,b} \ne 0$ and $(a,b) \sim (c,d)$ if and only if $q_{a,b}q_{c,d} \ne 0$.
\end{lemma}
\begin{proof}
    This is immediate from the definition of support graph and the fact that $\tr(MN) = 0$ if and only if $MN = 0$ for positive semidefinite matrices $M$ and $N$.
\end{proof}

Note that if a quantum permutation matrix $\mathcal{Q}$ comes from a quantum Latin square, i.e., $q_{a,b} = \outerp{\psi_{a,b}}{\psi_{a,b}}$ for all $a \in G$, $b \in G'$, then the vertex set of the support graph of the corresponding correlation is $G \times G'$.

Given a finite multiset $S \in \mathbb{C}^d$ of nonzero vectors, we define the \emph{non-orthogonality graph of $S$} to be the graph with vertex set $S$ such that two vectors in $S$ are adjacent if and only if they are not orthogonal. As a corollary to the above lemma, we see that the support graph of a correlation arising from a quantum Latin square is isomorphic to the non-orthogonality graph of the vectors in the quantum Latin square. 

\begin{cor}\label{cor:nonortho}
    Let $\Psi = (\ket{\psi_{a,b}})_{a \in G, b \in G'}$ be a quantum Latin square and let $p$ be the corresponding correlation, i.e.,
    \[p(a,c|b,d) = \frac{1}{|G|}|\!\inner{\psi_{a,b}}{\psi_{c,d}}\!|^2.\]
    Then the map $(a,b) \mapsto \ket{\psi_{a,b}}$ is an isomorphism from the support graph of $p$ to the non-orthogonality graph of $\{\ket{\psi_{a,b}} : (a,b) \in G \times G'\}$. 
\end{cor}
\begin{proof}
    This is immediate from Lemma~\ref{lem:qpermsupp} and the fact that $\outerp{\psi}{\psi}\outerp{\varphi}{\varphi} = 0$ if and only if $\ket{\psi}$ and $\ket{\varphi}$ are orthogonal.
\end{proof}

\begin{remark}
    In the case of a $(G,G')$-invariant quantum Latin square, the map $(a,b) \mapsto \ket{\psi_{a,b}}$ from the above corollary is not only an isomorphism. In the parlance of~\cite{jjortho}, it is also a \emph{symmetric orthogonal embedding} of the complement of the support graph. Conversely, any symmetric orthogonal embedding of the Cayley graph $\Cay(G \times G', \{(a,b) \in G \times G' \setminus \{(e,e)\} : a = e \text{ or } b = e\}$ in $\mathbb{C}^{|G|}$ yields a $(G,G')$-invariant quantum Latin square.
\end{remark}

Given a subset $S \subseteq M_n(\mathbb{C})$, we refer to the self-adjoint algebra
\[\spn\left\{\prod_{i=1}^k M_i : k \in \mathbb{N}, M_i \in S \text{ or } M_i^\dagger \in S \text{ for all } i \in [k]\right\}\]
as the \emph{algebra generated by $S$}. The next lemma shows that the algebra generated by the entries of a quantum permutation matrix is reducible if the support graph of its correlation is disconnected.

\begin{lemma}\label{lem:disconn2red}
    Let $\mathcal{Q} = (q_{a,b})_{a \in G, b \in G'}$ be a quantum permutation matrix and let $p$ be the corresponding correlation. Let $V_1, \ldots, V_c$ be the vertex sets of the connected components of the support graph of $p$, let $\A_i$ be the algebra generated by $\{q_{a,b} : (a,b) \in V_i)\}$. Then each $\A_i$ is nontrivial and $\A_i\A_j = \delta_{ij} \A_i$. In particular, this implies that if $X^p$ is disconnected, then the algebra generated by the entries of $\mathcal{Q}$ is reducible.
\end{lemma}
\begin{proof}
    By Lemma~\ref{lem:qpermsupp}, we have that the product of any element of $\{q_{a,b} : (a,b) \in V_i)\}$ with any element of $\{q_{a,b} : (a,b) \in V_j)\}$ is zero if $i \ne j$. Thus $\A_i\A_j = \delta_{ij}\A_i$ as desired. Since $(a,b) \in V(X^p)$ implies $q_{a,b} \ne 0$, it follows that each $\A_i$ is nontrivial and thus if $X^p$ is disconnected then the algebra generated by the entries of $\mathcal{Q}$ is reducible.
\end{proof}

Note that the subalgebras $\A_i$ are not necessarily irreducible, and so it may be possible to break them down even further. However, in the case where the quantum permutation matrix comes from a quantum Latin square, this cannot happen and we can make a much stronger statement. In particular, irreducibility is equivalent to connectedness in this case.


\begin{lemma}\label{lem:irreducibility}
    Let $\Psi = (\ket{\psi_{a,b}})_{a \in G, b \in G'}$ be a quantum Latin square and let $p$ be the corresponding correlation. Let $V_1, \ldots, V_k$ be the vertex sets of the connected components of the support graph of $p$, let $\A_i$ be the algebra generated by $\{\outerp{\psi_{a,b}}{\psi_{a,b}} : (a,b) \in V_i\}$, and let $S_i = \spn\{\ket{\psi_{a,b}} : (a,b) \in V_i\}$. Then
    \begin{equation}\label{eq:AiSi}
        \A_i = \spn\{\outerp{\varphi_1}{\varphi_2} : \ket{\varphi_1}, \ket{\varphi_2} \in S_i\},
    \end{equation}
    which is irreducible, and $\A_i \A_j = \delta_{ij}\A_i$. In particular, this means that the algebra generated by $\{\outerp{\psi_{a,b}}{\psi_{a,b}} : a \in G, b \in G'\}$ is irreducible if and only if the support graph is connected.
\end{lemma}
\begin{proof}
    Let $\A'_i$ denote the righthand side of Equation~\eqref{eq:AiSi} for each $i$. Note that it is easy to see that each $\A'_i$ is in fact a self-adjoint algebra, and moreover $\A_i$ is clearly contained in $\A'_i$. So it suffices to show that $\A_i \supseteq \A'_i$. To do this, it suffices to show that $\outerp{\varphi_1}{\varphi_2} \in \A_i$ for all $\ket{\varphi_1}, \ket{\varphi_2} \in S_i$. Furthermore, to show this we only need to show that $\outerp{\psi_{a,b}}{\psi_{c,d}} \in \A_i$ for all $(a,b), (c,d) \in V_i$. This we proceed to do.

    Suppose that $(a,b), (c,d) \in V_i$. Then by connectivity, there exists some sequence of vertices $(a,b) = (a_0,b_0), (a_1,b_1), \ldots, (a_k,b_k) = (c,d)$ in $V_i$ such that $(a_{\ell-1}, b_{\ell-1}) \sim (a_{\ell},b_{\ell})$ for all $\ell = 1, \ldots, k$. This implies that $\inner{\psi_{a_{\ell-1},b_{\ell-1}}}{\psi_{a_{\ell},b_{\ell}}} \ne 0$ for $\ell=1, \ldots, k$. Therefore,
    \begin{align*}
        \outerp{\psi_{a_0, b_0}}{\psi_{a_0,b_0}}\outerp{\psi_{a_1, b_1}}{\psi_{a_1,b_1}} \ldots \outerp{\psi_{a_k, b_k}}{\psi_{a_k,b_k}} &= \left(\prod_{\ell = 1}^{k}\inner{\psi_{a_{\ell-1},b_{\ell-1}}}{\psi_{a_{\ell},b_{\ell}}}\right) \outerp{\psi_{a,b}}{\psi_{c,d}}
    \end{align*}
    and the scalar on the righthand side is nonzero. Thus $\outerp{\psi_{a,b}}{\psi_{c,d}} \in \A_i$ as desired. Note that the righthand side of Equation~\eqref{eq:AiSi} is clearly irreducible

    Lastly, note that by definition of the support graph we have that $\inner{\psi_{a,b}}{\psi_{c,d}} = 0$ whenever $(a,b) \in V_i$ and $(c,d) \in V_j$ with $i \ne j$. This implies that the subspaces $S_i$ and $S_j$ are orthogonal and the rest of the lemma follows easily from this and the above.
\end{proof}

One of the reasons that reducibility/irreducibility is relevant, is that if the algebra $\A$ generated by the entries of a quantum permutation matrix $\mathcal{Q}$ is reducible, then the correlation $p$ produced by $\mathcal{Q}$ can be written as a convex combination of quantum correlations corresponding to the irreducible subalgebras of $\A$. In particular, we have the following analog of Theorem~\ref{thm:NSred}:

\begin{lemma}\label{lem:suppconv}
    Let $\mathcal{Q} = (q_{a,b})_{a \in G, b \in G'}$ be a quantum permutation matrix whose entries are elements of $M_d(\mathbb{C})$, let $p$ be the corresponding correlation, and let $V_1, \ldots, V_k$ be the vertex sets of the connected components of $X^p$. Then there exist projections $q^i \ne 0$ for each $i = 1, \ldots, k$ such that for all $a \in G$ and $b \in G'$ we have $\sum_{y \in G' : (a,y) \in V_i} q_{a,y} = q^i$ and $\sum_{x \in G : (x,b) \in V_i} q_{x,b} = q^i$. Let $d_i = \rk(q^i)$ for $i = 1, \ldots, k$ and define
    \[p_i(a,c|b,d) = \begin{cases}
        \frac{1}{d_i}\Tr(q_{a,b}q_{c,d}) & \text{if } (a,b),(c,d) \in V_i \\
        0 & \text{o.w.}
    \end{cases}\]
    Then each $p_i$ is a distinct winning quantum correlation for the $(G,G')$-bijection game, and $p$ is equal to the convex combination $\sum_{i=1}^k \frac{d_i}{d} p_i$.
\end{lemma}
\begin{proof}
    For each $a \in G$ and $i \in [k]$, let $q^i_a = \sum_{y \in G' : (a,y) \in V_i} q_{a,b}$. It is then clear that $q^i_a$ is a projection, and moreover $\sum_{i=1}^k q^i_a = \sum_{y \in G'} q_{a,y} = I$. Furthermore, by Lemma~\ref{lem:disconn2red}, we have that $q^i_a q^j_c = 0$ if $i \ne j$. Therefore, $q^i_a = q^i_a \sum_{j} q^j_c = q^i_a q^i_c$ and similarly $q^i_c = q^i_a q^i_c$. Thus $q^i_a = q^i_c$ and we will simply denote this projection as $q^i$. Analogously, there is some projection $\hat{q}^i$ such that $\sum_{x \in G: (x,b) \in V_i} q_{x,b} = \hat{q}^i$ for all $b \in G'$. But $|G|q^i = \sum_{(a,b) \in V_i} q_{a,b} = |G'|\hat{q}^i$, and thus $q^i = \hat{q}^i$.

    Using the above, it is easy to see that $\mathcal{Q}^i = (q^i_{a,b})$ where
    \[q^i_{a,b} := \begin{cases}
        q_{a,b} & \text{if } (a,b) \in V_i \\
        0 & \text{o.w.}
    \end{cases}\]
    is a quantum permutation matrix over the algebra $\A_i$ generated by $\{q_{a,b} : (a,b) \in V_i\}$ whose identity element is $q^i$. Thus $\frac{1}{d_i}\Tr(\cdot)$ is the normalized trace on $\A_i$ and therefore $p_i$ is a winning quantum correlation for the $(G,G')$-bijection game by Equation~\eqref{eq:qcorr}. Note that
    \begin{align*}
    \sum_{a,c \in G: (a,b),(c,d) \in V_i} p(a,c|b,d) &= \sum_{a,c \in G: (a,b),(c,d) \in V_i} \frac{1}{d}\Tr(q_{a,b}q_{c,d}) \\
    &= \frac{1}{d}\Tr\left(\sum_{a \in G : (a,b) \in V_i}  q_{a,b} \sum_{c \in G : (c,d) \in V_i} q_{c,d}\right) \\
    &= \frac{1}{d}\Tr(q^i) = \frac{d_i}{d}.
    \end{align*}
    Thus the $s_i$ from Theorem~\ref{thm:NSred} in this case is equal to $\frac{d_i}{d}$ and it follows that the $p_i$ defined here are the same as the $p_i$ defined in Theorem~\ref{thm:NSred}. Thus the remaining claims of the lemma follow from Theorem~\ref{thm:NSred}.

\end{proof}

We remark that in the case of a quantum Latin square, the $d$ in Lemma~\ref{lem:suppconv} is equal to $|G|$ and each $d_i$ is just the dimension of the subspace $S_i$ defined in Lemma~\ref{lem:irreducibility}.

As mentioned above, whenever the algebra $\A$ generated by the entries of a quantum permutation matrix is reducible, the resulting correlation can be written as a convex combination of quantum correlations corresponding to the irreducible subalgebras of $\A$. However, it is possible that this convex combination is trivial. For example, let $\mathcal{Q} = (q_{a,b})$ be a quantum permutation matrix whose entries generate an irreducible algebra with corresponding correlation $p$, and define $\mathcal{Q}' := (q_{a,b} \oplus q_{a,b})$. Then $\mathcal{Q}'$ is also a quantum permutation matrix and the correlation it produces is also $p$. Of course, the algebra generated by the entries of $\mathcal{Q}'$ is not irreducible by design. But the convex combination given by its irreducible subalgebras is simply $p = \frac{1}{2}p + \frac{1}{2}p$. Furthermore, it is possible that $p$ cannot be written as a nontrivial convex combination of any non-signalling correlations. This happens, for instance, if $\mathcal{Q}$ is simply a classical permutation matrix. This example also shows that, unlike in the quantum Latin square case, the ``$X^p$ disconnected implies reducibility" result cannot be reversed.

\subsection{Support graphs of group-invariant correlations}\label{subsec:Ginvsupp}

Now we arrive to the group-invariant correlations. Recall that all group-invariant correlations are non-signalling, so all of our previous results still apply here. Also, since all of the marginals $p(a|b)$ are equal to $1/|G|$ in this case, the vertex set of $X^p$ is always $G \times G'$. In fact, we show that in this case the support digraph is actually a Cayley digraph on $G \times G'$:

\begin{lemma}\label{lem:suppiscay}
    Let $p$ be a $(G,G')$-invariant correlation. Then the support (di)graph of $p$ is the Cayley (di)graph on $G \times G'$ with connection set $\{(a,b) \in G \times G': D^p_{a,b} \ne 0\}$.
\end{lemma}
\begin{proof}
    First note that for a $(G,G')$-invariant correlation $p$, we have that $p(a,a|b,b) = \frac{1}{|G|} \ne 0$ for all $(a,b) \in G \times G'$, and so $V(X^p) = G \times G'$, as needed. Let $C = \{(a,b) \in G \times G': D^p_{a,b} \ne 0\}$. Now consider vertices $(a,b)$ and $(c,d)$ of $X^p$. We have that there is an arc from $(a,b)$ to $(c,d)$ if and only if $p(a,c|b,d) \ne 0$ if and only if $D^p_{a^{-1}c, b^{-1}d} \ne 0$ if and only if $(a,b)^{-1}(c,d) \in C$, and thus the lemma is proven.
\end{proof}

Before proving more results on support graphs of group-invariant correlations, we need a purely group-theoretic result. For finite groups $G$ and $G'$, and a subgroup $K \subseteq G \times G'$, we let $\pi_G$ and $\pi_{G'}$ denote the \emph{canonical projections} from $K$ to $G$ and $G'$ respectively, i.e., $\pi_G(a,b) = a$ and $\pi_{G'}(a,b) = b$.

\begin{lemma}\label{lem:groupthing}
    Let $G$ and $G'$ be finite groups and suppose that $K$ is a subgroup of $G \times G'$ such that the canonical projections, $\pi_G:K\to G$ and $\pi_{G'}:K\to G'$ 
    are surjective. Then the following hold:
    \begin{enumerate}
        \item The subgroups $H := \{a \in G : (a,e) \in K\}$ and $H' := \{b \in G' : (e,b) \in K\}$ are normal subgroups of $G$ and $G'$ respectively.
        \item The product $H \times H'$ is a normal subgroup of both $K$ and $G \times G'$.
        \item The quotients $G/H$, $G'/H'$ and $K/(H\times H')$ are isomorphic.
        \item $K = G \times G'$ if and only if $H = G$ if and only if $H' = G'$.
    \end{enumerate}
\end{lemma}
\begin{proof}
    Clearly, $\ker(\pi_G) = \{e\}\times H'$ and so $\{e\}\times H'$ is a normal subgroup of $K$. It then follows easily that $H'$ is normal in $G'$ and similarly $H\lhd G$, proving the first claim. It follows directly that $H\times H'$ is normal in  $G\times G'$ and moreover, since $K$ contains $H\times \{e\}$ and $\{e\}\times H'$, it also contains $H\times H'$. Therefore $H\times H'$ is normal in $K$, proving the second claim.
    
    By the first isomorphism theorem we have $G\cong K/\ker(\pi_G) = K/(\{e\}\times H')$, and similarly, $G'\cong K/(H\times \{e\})$. Let $\phi$ and $\phi'$ be the respective canonical isomorphisms. It is easy to verify that the map given by $gH\mapsto (a,b)(H\times H')=(aH,bH')$ where $(a,bH') := \phi(g)$ is an isomorphism from $G/H$ to $K/(H\times H')$ and similarly we can show that $g'H'\mapsto \phi'(g')(e,H')$ is an isomorphism. Thus 
    \begin{equation}\label{eq:quotients}
        G/H\cong K/(H\times H')\cong G'/H',
    \end{equation}
    proving the third claim. If $K=G\times G'$, we clearly have $H=G$ and $H'=G'$ by definition of $H$ and $H'$. The rest follows from \eqref{eq:quotients}.
\end{proof}

Note that the subgroup $K$ does not need to be a normal subgroup of $G \times G'$. For example, if $G=G'$ is a non-abelian group, then taking $K = \{(a,a) : a \in G\}$ meets the conditions of the above lemma and is not a normal subgroup of $G \times G$.

\begin{lemma}\label{lem:Dpblocks}
    Let $p$ be a $(G,G')$-invariant correlation, and let $K$ be the vertex set of the connected component of its support digraph containing the identity. Then $K$ is a subgroup of $G \times G'$ such that the canonical projections $\pi_G : K \to G$ and $\pi_{G'} : K \to G'$ are surjective. Moreover, if $\pi:= \phi^{-1} \circ \phi'$ where $\phi,\phi':G/H,G'/H'\to K/(H\times H')$ are the isomorphisms 
    defined in the proof of Lemma~\ref{lem:groupthing}, then 
    \[D^p_{a,b} = 0 \text{ unless } a \in \pi(bH').\]
\end{lemma}
\begin{proof}
    The vertex set of the connected component of a Cayley digraph is the subgroup generated by its connection set. It follows that $K$ is a subgroup of $G \times G'$. Furthermore, since the correlation matrix $D^p$ is doubly stochastic, for every $a \in G$ there exists $b \in G'$ such that $D^p_{a,b} \ne 0$. It follows that $(a,b)$ is a neighbour of the identity in the support digraph and in particular it is contained in $K$. Thus $a \in \pi_G(K)$ for all $a \in G$, and the analogous statement for $G'$ holds similarly.

    The statement $a \in \pi(bH')$ is equivalent to the statement $(a,b) \in K$. Obviously, if $(a,b) \notin K$, then $(a,b)$ is not in the connected component of the support digraph of $p$ containing the identity, which in particular implies that $(a,b)$ is not in the connection set of the support digraph. This last statement is equivalent to $D^p_{a,b} = 0$ by definition.
\end{proof}

\begin{remark}\label{rem:Dblockdiag}
One way to view the above lemma is that if we partition the columns of $D^p$ into the cosets of the subgroup $H'$ and partition its rows into the cosets of $H$ such that the $bH'$ column block and $\pi(bH')$ row block appear at the same position, then this block-diagonalizes $D^p$. This is not necessarily the finest possible block diagonalization; that is given by the connected components of the bipartite graph defined in Section~\ref{subsec:cayley}.    
\end{remark}

It is trivial to see that the subgroup $K$ from the above lemma is equal to $G \times G'$ if and only if the support graph of the group-invariant correlation is connected. On the other extreme, we have the following:

\begin{lemma}
    Let $p$ be a $(G,G')$-invariant correlation, and let $K$ be the vertex set of the connected component of its support digraph containing the identity. If 
    $H$ and $H'$ are the subgroups of $G$ and $G'$ defined as in Lemma~\ref{lem:groupthing}, then the following are equivalent:
    \begin{enumerate}
        \item $H$ is the subgroup containing only the identity, i.e., $G/H =G$.
        \item $H'$ is the subgroup containing only the identity, i.e., $G'/H' =G'$.
        \item $|K| = |G|$.
        \item $D^p$ is a permutation matrix encoding an isomorphism from $G'$ to $G$.
        \item $X^p$ is the disjoint union of $|G|$ complete graphs of size $|G|$.
    \end{enumerate}
\end{lemma}
\begin{proof}
    By Lemma~\ref{lem:groupthing}, $G/H \cong G'/H'$ and in this case $|G| = |G'|$ and therefore $|H| = |H'|$. The equivalence of items (1) and (2) follows immediately.

    Since $K/(H \times H') \cong G/H$, it follows that if $(1)$ (and therefore $(2)$) holds, then $|K| = |G|$. Conversely, if $|K| = |G|$, then $K/(H \times H')$ can only have the same size as $G/H$ (which it is isomorphic to) if $|H'| = 1$. Thus $(3) \Rightarrow (1)$ and also $(3) \Rightarrow (2)$ since $(1)$ and $(2)$ are equivalent.

    If items $(1)-(3)$ hold, then $G/H = G$ and $G'/H' = G'$ and so the isomorphism $\pi$ from the proof of Lemma~\ref{lem:groupthing} is an isomorphism from $G'$ to $G$. Then by Lemma~\ref{lem:Dpblocks} the matrix $D^p$ is just the permutation matrix encoding this isomorphism. Conversely, if $D^p$ is a permutation matrix encoding an isomorphism $\pi$ from $G'$ to $G$, then the connection set of the support graph of $p$ is $\{(\pi(b),b) : b \in G' \setminus \{e\}\}$. The subgroup of $G \times G'$ generated by this set is easily seen to be $\{(\pi(b),b) : b \in G' \}$ and thus this is the vertex set of the connected component of the support graph containing the identity, which by definition is $K$. This set has size $|G'| = |G|$ and so we have shown that $(4)$ is equivalent to $(1)-(3)$.

    Lastly, if $(1)-(4)$ hold, then $|K| = |G|$ and therefore the connected component of $X^p$ containing the identity contains $|G|$ vertices. It then follows from Lemma~\ref{lem:cliques} that this connected component is a clique. Since $X^p$ is a Cayley graph, all of its connected components are isomorphic, and therefore they are all cliques of size $|G|$. Since the vertex set of $X^p$ is $G \times G'$, there are exactly $|G|$ such components. Thus we have shown that $(1)-(4) \Rightarrow (5)$. Conversely, suppose that $X^p$ is a disjoint union of $|G|$ cliques of size $|G|$. Then the connected component of $X^p$ containing the identity has $|G|$ vertices and therefore $|K| = |G|$, and so $(1)-(4)$ hold.
\end{proof}

Now we show that in the group-invariant case, the correlations $p_1, \ldots, p_k$ defined in Theorem~\ref{thm:NSred} are all the same up to left multiplication of their arguments. Recall that for $x \in G$, the bijection $\tau_x \in \bij(G)$ is defined as $\tau_x(a) = xa$.

\begin{lemma}\label{lem:cosetcorrs}
    Let $p$ be a $(G,G')$-invariant correlation, let $K$ be the vertex set of the connected component of $X^p$ containing the identity, let $H$ and $H'$ be defined as in Lemma~\ref{lem:groupthing}. For any $(x,y) \in G \times G'$, let $p_{(x,y)}$ be the correlation corresponding to the component of $X^p$ containing $(x,y)$ as defined in Theorem~\ref{thm:NSred}. Then
    \begin{equation}\label{eq:cosetcorrs1}
        p_{(x,y)}(a,c|b,d) = \begin{cases}
        \frac{|G \times G'|}{|K|} p(a,c|b,d) & \text{if } (a,b),(c,d) \in (x,y)K \\
        0 & \text{o.w.}
    \end{cases}
    \end{equation}
    and
    \begin{equation}\label{eq:cosetcorrs2}
    p_{(x,y)}(a,c|b,d) = p_{\tau_x} \circ p_{(e,e)} \circ p_{\tau_{y^{-1}}}(a,c|b,d) = p_{(e,e)}(x^{-1}a,x^{-1}c|y^{-1}b,y^{-1}d).
    \end{equation}
    It follows that if $(x_1,y_1), \ldots, (x_k,y_k)$ are a full set of representatives of the left cosets of $K$ in $G \times G'$, then
    \begin{equation}\label{eq:cosetcorrs3}
    p = \frac{|K|}{|G \times G'|} \sum_{i=1}^k p_{(x_i,y_i)} = \frac{|K|}{|G \times G'|} \sum_{i=1}^k p_{\tau_{x_i}} \circ p_{(e,e)} \circ p_{\tau_{{y_i}^{-1}}}
    \end{equation}
\end{lemma}
\begin{proof}
    For fixed $a \in G$, $b,d \in G'$ such that $(a,b) \in (x,y)K$, we have that
    \[\sum_{c \in G: (c,d) \in (x,y)K} p(a,c|b,d) = \sum_{c \in G} p(a,c|b,d) = p(a|b) = \frac{1}{|G|}.\]
    Since $H \times H'$ is a normal subgroup of $K$, there are exactly $|H| = |K|/|G|$ elements $a \in G$ such that $(a,d) \in K$ for any fixed $d \in G'$. Thus summing the above expression over all $a \in G$ such that $(a,d) \in K$ yields $|K|/|G \times G'|$. Therefore, all of the $s_i$ from Theorem~\ref{thm:NSred} are equal to $|K|/|G \times G'|$, and so we have proven~\eqref{eq:cosetcorrs1}.

    Now note that $(a,b),(c,d) \in (x,y)K$ if and only if $(x^{-1}a,y^{-1}b), (x^{-1}c,y^{-1}d) \in K$, and moreover, $p(x^{-1}a,x^{-1}c|y^{-1}b,y^{-1}d) = p(a,c|b,d)$ since $p$ is $(G,G')$-invariant. Thus the first and last expressions of~\eqref{eq:cosetcorrs2} are in fact equal. The middle expression is equal to the last expression through a straightforward calculation.

    If $(x_1,y_1), \ldots, (x_k,y_k)$ are a full set of representatives of the left cosets of $K$, then the correlations $p_{(x_1,y_1)}, \ldots, p_{(x_k,y_k)}$ are precisely the correlations $p_1, \ldots, p_k$ defined in Theorem~\ref{thm:NSred} and so~\eqref{eq:cosetcorrs3} follows immediately from that theorem and our computation of the $s_i$ above.
\end{proof}

One way to view the above lemma is that a group-invariant correlation $p$ whose support graph is disconnected can be decomposed into correlations that are equal up to pre- and post-composing with some deterministic classical correlations, and moreover these all have disjoint supports. So it is natural to expect that many properties of the correlation $p$ are determined by the properties of the correlation $p_{(e,e)}$ defined above. Indeed, we have the following:

\begin{cor}
    Let $p$ be a $(G,G')$-invariant correlation, and let $p_{(e,e)}$ be as defined in Lemma~\ref{lem:cosetcorrs}. Then $p$ is classical if and only if $p_{(e,e)}$ is classical. Also $p$ is strongly nonlocal if and only if $p_{(e,e)}$ is strongly nonlocal.
\end{cor}
\begin{proof}
    By Theorem~\ref{thm:NSred} and Lemma~\ref{lem:cosetcorrs}, $p$ is classical if and only if $p_{(x,y)}$ is classical for all $(x,y) \in G \times G'$, so the forward direction is immediate.
    
    Now suppose that $p_{(e,e)}$ is classical. By Lemma~\ref{lem:cosetcorrs}, we have that
    \[p_{(x,y)} = p_{\tau_x} \circ p_{(e,e)} \circ p_{\tau_{y^{-1}}},\]
    where recall that $\tau_x \in \bij(G)$ is defined as $\tau_x(a) = xa$, and $\tau_{y^{-1}} \in \bij(G')$ is defined analogously. Since $p_{\tau_x}$ and $p_{\tau_{y^{-1}}}$ are classical, and composition of classical correlations results in a classical correlation, we have that $p_{(x,y)}$ is classical for all $(x,y) \in G \times G'$ and thus $p$ is classical.    
    By Theorem~\ref{thm:NSred}, $p$ is strongly nonlocal if and only if every $p_{(x,y)}$ is strongly nonlocal if and only if none of the $X^{p_{(x,y)}}$ contain a clique of size $|G|$. But each $X^{p_{(x,y)}}$ is a connected component of $X^p$ which are all isomorphic since it is a Cayley digraph. Therefore they either all contain such a clique or none do. Thus $p$ is strongly nonlocal if and only if $p_{(e,e)}$ is strongly nonlocal.
\end{proof}

The correlation $p_{(e,e)}$ from the above lemma and corollary can also in some sense be decomposed, but not to the extent of the initial correlation $p$ above.

\begin{lemma}
    Let $p$ be a $(G,G')$-invariant correlation, and let $K$, $H$, $H'$, and $p_{(e,e)}$ be as defined in Lemma~\ref{lem:cosetcorrs}. Then the restriction of $p_{(e,e)}$ to inputs from $H'$ (and thus outputs from $H$) is an $(H,H')$-invariant correlation. Moreover, if $(x,y) \in K$, then $p_{(e,e)}(xa,xc|yb,yd) = p_{(e,e)}(a,c|b,d)$
\end{lemma}
\begin{proof}
    First note that if $b \in H'$, then $(a,b) \in K$ if and only if $a \in H$. Therefore, if $b,d \in H'$, then $p_{(e,e)}(a,c|b,d) = 0$ unless $a,c \in H$. Thus if we define $\tilde{p} : H \times H \times H' \times H' \to \mathbb{R}$ as
    \[\tilde{p}(a,c|b,d) := p_{(e,e)}(a,c|b,d) = \frac{|G \times G'|}{|K|} p(a,c|b,d) \text{ for } (a,b),(c,d) \in H \times H',\]
    then this is a correlation which we can think of as the restriction of $p_{(e,e)}$ to inputs from $H'$. Moreover, it is immediate that it is $(H,H')$-invariant, since $p$ was $(G,G')$-invariant.

    For the final claim, note that if $(x,y) \in K$, then $(xa,yb) = (x,y)(a,b) \in K$ if and only if $(a,b) \in K$. Thus, if either $(a,b) \notin K$ or $(c,d) \notin K$, then
    \[p_{(e,e)}(a,c|b,d) = 0 = p_{(e,e)}(xa,xc|yb,yd) \ \forall (x,y) \in K.\]
    But if $(a,b) \in K$ and $(c,d) \in K$, then for all $(x,y) \in K$,
    \[p_{(e,e)}(xa,xc|yb,yd) = \frac{|G \times G'|}{|K|}p(xa,xc|yb,yd) = \frac{|G \times G'|}{|K|}p(a,c|b,d) = p_{(e,e)}(a,c|b,d),\]
    as desired.
\end{proof}


Now we will consider the relationship between a $(G,G')$-transformation matrix and the support graph of its corresponding correlation.

\begin{lemma}
    Let $U$ be the transformation matrix of a $(G,G')$-invariant quantum Latin square $\Psi = (\ket{\psi_{a,b}})$ with corresponding correlation $p$. Let $K,H$, and $H'$ be defined as in previous lemmas. Then the submatrix $\hat{U}$ of $U$ with rows indexed by $H$ and columns by $H'$ is a $(H,H')$-transformation matrix. Moreover, the subarray $(\ket{\psi_{a,b}})_{a \in H, b \in H'}$ of $\Psi$ is an $(H,H')$-invariant quantum Latin square with transformation matrix $\hat{U}$.
\end{lemma}
\begin{proof}
    Since $U_{a,b} = 0$ if $a \in H$ and $b \notin H'$ or if $a \notin H$ and $b \in H'$, it follows that $\hat{U}$ is unitary. The condition $\hat{U}_{a^{-1},b^{-1}} = \overline{\hat{U}_{a,b}}$ holds trivially. Lastly, for all $a,b \in G$ and $c \in G'$, 
    \begin{align*}
        \sum_{x,y \in H': xy = c} \hat{U}_{a,x}\hat{U}_{b,y} &= \sum_{x,y \in H': xy = c} U_{a,x} U_{b,y} \\
        &= \sum_{x,y \in G': xy = c} U_{a,x} U_{b,y} \\
        &= U_{ab,c} = \hat{U}_{ab,c}.
    \end{align*}
    Thus the first claim of the lemma is proven.

    Now let $\hat{\Psi}$ denote the subarray of $\Psi$ with rows indexed by $H$ and columns by $H'$. For all $a,c \in H$ and $b,d \in H'$, we have that
    \[\inner{\psi_{a,b}}{\psi_{c,d}} = U_{a^{-1}c,b^{-1}d} = \hat{U}_{a^{-1}c,b^{-1}d}.\]
    Note that since $\hat{U}$ is unitary, this implies that every row/column of $\hat{\Psi}$ spans some common vector space, and furthermore the rows and columns must be orthonormal bases of this space since the vectors they contain are orthonormal. Thus $\hat{\Psi}$ is a quantum Latin square and it is immediately $(H,H')$-invariant with transformation matrix $\hat{U}$ by the equation displayed above.
\end{proof}

As mentioned in Remark~\ref{rem:Dblockdiag} regarding the correlation matrix $D^p$, by partitioning the rows and columns of a $(G,G')$-transformation matrix $U$ into the cosets of $H$ and $H'$ appropriately, we can block diagonalize $U$. The above Lemma shows that the block corresponding to $H$ and $H'$ is in fact a $(H,H')$-invariant transformation matrix. It would be very interesting if one were able to somehow characterize what the remaining blocks of $U$ could be.

\section{Computational Results}\label{sec:computations}

Here we will present some results that rely at least partially on computations.

\subsection{A non-extreme $D^\pi$ for $\mathbb{Z}_6$}\label{subsec:Z6comps}

As mentioned in Section~\ref{subsec:classicalGICs}, every extreme point of the set of classical $(G,G')$-invariant correlations is contained in the set $\{D^\pi : \pi \in \bij(G,G')\}$. We give an example for $G = G' = \mathbb{Z}_6$ of a $D^\pi$ that is not an extreme point.

We will let $\mathrm{id}$ denote the identity permutation of $\mathbb{Z}_6$ and $\mathrm{inv}$ denote the permutation that maps every element to its inverse. Note that both of these are automorphisms of $\mathbb{Z}_6$ and so by Lemma~\ref{lem:classperms} we have that $D^\mathrm{id} = P^\mathrm{id} = I$ and $D^\mathrm{inv} = P^\mathrm{inv}$ which is a 01-matrix with 1's on the ``backwards" diagonal. Let us fix the permutation $\pi = (1,3)$ written in cycle notation, and also let $\pi_1 = (0,1)(2,3)(4,5)$ and $\pi_2 = (1,3)(2,4)$. Then, letting $\tilde{D}^\pi$ denote the principal submatrix of $D^\pi$ consisting of all but the identity row and column, straightforward computation (using, e.g.,~Lemma~\ref{lem:redmtx}) yields
\[\tilde{D}^\pi = \begin{pmatrix}
    \frac{1}{3} & 0 & \frac{1}{3} & 0 & \frac{1}{3}\\
    0 & \frac{1}{2} & 0 & \frac{1}{2} & 0\\
    \frac{1}{3} & 0 & \frac{1}{3} & 0 & \frac{1}{3}\\
    0 & \frac{1}{2} & 0 & \frac{1}{2} & 0\\
    \frac{1}{3} & 0 & \frac{1}{3} & 0 & \frac{1}{3}\\
\end{pmatrix},
\quad \quad
\tilde{D}^{\pi_1} = \begin{pmatrix}
0 & 0 & \frac{1}{2} & 0 & \frac{1}{2}\\
0 & 1 & 0 & 0 & 0\\
\frac{1}{2} & 0 & 0 & 0 & \frac{1}{2}\\
0 & 0 & 0 & 1 & 0\\
\frac{1}{2} & 0 & \frac{1}{2} & 0 & 0
\end{pmatrix},
\quad \quad
\tilde{D}^{\pi_2} = \begin{pmatrix}
    \frac{1}{2} & 0 & \frac{1}{2} & 0 & 0\\
    0 & 0 & 0 & 1 & 0\\
    \frac{1}{2} & 0 & 0 & 0 & \frac{1}{2}\\
    0 & 1 & 0 & 0 & 0\\
    0 & 0 & \frac{1}{2} & 0 & \frac{1}{2}
\end{pmatrix}.\]
It is then straightforward to check that
\[D^\pi = \frac{1}{6} D^{\mathrm{id}} + \frac{1}{6}D^{\mathrm{inv}} + \frac{1}{3} D^{\pi_1} + \frac{1}{3} D^{\pi_2}.\]
Since $D^\pi$ can be written as a nontrivial convex combination of other correlation matrices of classical $\mathbb{Z}_6$-invariant correlations, it cannot be an extreme point of this set.

\subsection{A non-classical $\mathbb{Z}_2^4$-invariant correlation}\label{subsec:Z24comps}

In~\cite{quantumvsnonlocal}, it was shown that the quantum correlations produced by $\mathbb{Z}_2^d$-invariant quantum Latin squares are always classical for $d = 1,2,3$. They then asked whether this pattern continued: are the correlations produced by $\mathbb{Z}_2^d$-invariant quantum Latin squares always classical for all $d \in \mathbb{N}$?



Here we answer this question in the negative, by producing a $\mathbb{Z}_2^4$-transformation matrix whose resulting correlation is non-classical. In fact, the correlation we produce will be strongly nonlocal. In terms of correlation matrices, we find $\hat{\pi} \in \bij(\widehat{\mathbb{Z}_2^4})$ such that there is no $\pi \in \bij(\mathbb{Z}_2^4)$ satisfying $D^\pi_{a,b} \ne 0 \Rightarrow Q^{\hat{\pi}}_{a,b} \ne 0$. Since all the $D^\pi$ are entrywise nonnegative matrices, this implies that this $Q^{\hat{\pi}}$ cannot be contained in their convex hull. 
To find such a $Q^{\hat{\pi}}$, we generate random permutations $\hat{\pi}$ of $\widehat{\mathbb{Z}_2^4}$ (that fix the identity) and construct the corresponding $Q^{\hat{\pi}}$ according to definition~\ref{def:Qpi}. We then recursively search for $\pi \in \bij(\mathbb{Z}_2^4)$ such that $D^\pi_{a,b}$ is zero everywhere $Q^{\hat{\pi}}$ is. To do this, we make use of the penultimate expression for $D^\pi_{a,b}$ given in the proof of Lemma~\ref{lem:avg}. Writing things in additive notation since we are in $\mathbb{Z}_2^4$, it is not difficult to see that the nonzero entries of $D^\pi$ are precisely the entries $D^{\pi}_{\pi(x)+\pi(y),x+y}$ for some $x,y \in \mathbb{Z}_2^4$. Now given a bijection $\tau$ from $S \subseteq \mathbb{Z}_2^4$ to $T \subseteq \mathbb{Z}_2^4$, we consider possible extensions of $\tau$ to one more element of $\mathbb{Z}_2^4$ such that this extension does not already guarantee that we will have nonzero entries of our final $D^\pi$ that are zero in $Q^{\hat{\pi}}$. We do this by considering each $y \in \mathbb{Z}_2^4 \setminus S$ and constructing the set
\[\mathrm{pos\_vals(y)} := \{z \in \mathbb{Z}_2^4 \setminus T : Q^{\hat{\pi}}_{\tau(x) + z, x + y} \ne 0 \text{ for all } x \in S\}\]
which are the possible values for $\tau(y)$. We then select $y \in \mathbb{Z}_2^4 \setminus S$ such that $\mathrm{pos\_vals}(y)$ has minimum size and we consider all extensions of $\tau$ to $S \cup \{y\}$ such that $\tau(y) \in \mathrm{pos\_vals}(y)$. We continue this recursion until we find a permutation $\pi$ of $\mathbb{Z}_2^4$ such that $D^\pi$ is zero everywhere $Q^{\hat{\pi}}$ is, or we are able to conclude that no such $\pi$ exists. In the latter case we have found a $\hat{\pi} \in \bij(\widehat{\mathbb{Z}_2^4})$ such that $Q^{\hat{\pi}}$ is the correlation matrix of a strongly nonlocal quantum correlation.

The approach described above turns out to be more efficient than constructing the support graph of the correlation corresponding to $Q^{\hat{\pi}}$ and checking if it has a clique of size $|\mathbb{Z}_2^4| = 16$. It only takes about 90 seconds to run the recursive search on 10,000 randomly generated permutations of $\widehat{\mathbb{Z}_2^4}$, and this yields about a dozen of the desired examples of $Q^{\hat{\pi}}$ on average. One particular example of $\hat{\pi} \in \widehat{\mathbb{Z}_2^4}$ such that there is no $D^\pi$ satisfying $D^\pi_{a,b} \ne 0 \Rightarrow Q^{\hat{\pi}}_{a,b} \ne 0$ is $\hat{\pi} = (1,11)(2,8,3,7,10,15,5,6,14,4)$. Here we are identifying the numbers $0,1,\ldots, 15$ with their binary representations which we think of as members of $\mathbb{Z}_2^4$ in the obvious way. Moreover, we identify an element $a \in \mathbb{Z}_2^4$ with the character $\chi \in \widehat{\mathbb{Z}_2^4}$ defined as $\chi(x) = (-1)^{x \cdot a}$ for all $x \in \mathbb{Z}_2^4$, where $x \cdot a$ is the usual dot product. The resulting matrix $Q^{\hat{\pi}}$ is given in Figure~\ref{fig:Z24}.

\begin{figure}[h!]
    \centering
    \setcounter{MaxMatrixCols}{20}
\[
\begin{pmatrix}
   1  &  0  &  0  &  0  &  0  &  0  &  0  &  0  &  0  &  0  &  0  &  0  &  0  &  0  &  0  &  0 \\
   0 & \frac{1}{4}  &  0 & \frac{1}{4}  &  0  &  0  &  0  &  0  &  0  &  0  &  0  &  0 & \frac{1}{4}  &  0 & \frac{1}{4}  &  0\\
   0  &  0  &  0  &  0 & \frac{1}{16} & \frac{1}{16} & \frac{1}{16} & \frac{1}{16}  &  0 & \frac{1}{4} & \frac{1}{4}  &  0 & \frac{1}{16} & \frac{1}{16} & \frac{1}{16} & \frac{1}{16}\\
   0  &  0  &  0  &  0 & \frac{1}{16} & \frac{1}{16} & \frac{1}{16} & \frac{1}{16} & \frac{1}{4} & \frac{1}{4}  &  0  &  0 & \frac{1}{16} & \frac{1}{16} & \frac{1}{16} & \frac{1}{16}\\
   0 & \frac{1}{4}  &  0 & \frac{1}{4} & \frac{1}{4}  &  0&  \frac{1}{4}&    0&    0&    0&    0&    0&    0&    0&    0&    0\\
   0&    0&    0&    0&    0&  \frac{1}{4}&    0 & \frac{1}{4}&    0&    0&    0&    0&    0&  \frac{1}{4}&    0&  \frac{1}{4}\\
   0 &   0&    0&    0& \frac{1}{16}& \frac{1}{16}& \frac{1}{16}& \frac{1}{16}&    0&    0&  \frac{1}{4} & \frac{1}{4} &\frac{1}{16} &\frac{1}{16} &\frac{1}{16} &\frac{1}{16}\\
   0&    0&    0&    0& \frac{1}{16}& \frac{1}{16} &\frac{1}{16}& \frac{1}{16}&  \frac{1}{4}&    0&    0&  \frac{1}{4}& \frac{1}{16}& \frac{1}{16}& \frac{1}{16}& \frac{1}{16}\\
   0&  \frac{1}{4} & \frac{1}{4}&    0& \frac{1}{16} &\frac{1}{16}& \frac{1}{16}& \frac{1}{16}&    0&    0&    0&    0& \frac{1}{16}& \frac{1}{16}& \frac{1}{16} &\frac{1}{16}\\
   0&    0&  \frac{1}{4}&  \frac{1}{4}& \frac{1}{16}& \frac{1}{16}& \frac{1}{16}& \frac{1}{16}&    0&    0&    0&    0& \frac{1}{16}& \frac{1}{16}& \frac{1}{16}& \frac{1}{16}\\
   0&    0&    0&    0&    0&    0&    0&    0&  \frac{1}{4}&  \frac{1}{4}&  \frac{1}{4}&  \frac{1}{4}&    0&    0 &   0&    0\\
   0&    0 &   0&    0&    0&  \frac{1}{4}&    0&  \frac{1}{4} &   0&  \frac{1}{4}&    0&  \frac{1}{4} &   0 &   0&    0&    0\\
   0 &   0&  \frac{1}{4} & \frac{1}{4}& \frac{1}{16}& \frac{1}{16} &\frac{1}{16}& \frac{1}{16}&    0 &   0&    0&    0& \frac{1}{16}& \frac{1}{16}& \frac{1}{16}& \frac{1}{16}\\
   0 & \frac{1}{4} & \frac{1}{4}&    0 &\frac{1}{16}& \frac{1}{16} &\frac{1}{16}& \frac{1}{16}&    0 &   0&    0&    0& \frac{1}{16}& \frac{1}{16} & \frac{1}{16} &\frac{1}{16}\\
   0&    0&    0&    0&    0&    0&    0&    0&  \frac{1}{4}&    0&  \frac{1}{4} &   0&  \frac{1}{4}&    0 & \frac{1}{4}&    0\\
   0&    0&    0&    0&  \frac{1}{4}&    0 & \frac{1}{4}&    0 &   0 &   0 &   0 &   0 &   0 & \frac{1}{4} &   0 & \frac{1}{4}
\end{pmatrix}\]
    \caption{The correlation matrix of a strongly nonlocal quantum correlation arising from a $\mathbb{Z}_2^4$-invariant quantum Latin square.}
    \label{fig:Z24}
\end{figure}

It is important to note that this example also allows us to construct similar examples for $\mathbb{Z}_2^d$ for $d >4$. Let $\hat{\pi} \in \bij(\widehat{\mathbb{Z}_2^4})$ be the permutation such that $Q^{\hat{\pi}}$ is the matrix in Figure~\ref{fig:Z24}. Then $P^{\hat{\pi}} \otimes I_{2}$ is a permutation matrix encoding an element $\pi'$ of $\bij(\widehat{\mathbb{Z}_2^5})$. Since the normalized character table of $\mathbb{Z}_2^d$ is $H^d$ where
\[H = \frac{1}{\sqrt{2}}\begin{pmatrix}
    1 & 1\\ 1 & -1
\end{pmatrix}
\]
is the Hadamard matrix, we see that $U^{\pi'} = U^{\hat{\pi}} \otimes I_{2}$ and thus
\[Q^{\pi'} = \left( U^{\hat{\pi}} \otimes I_{2}\right) \circ \left( \overline{U^{\hat{\pi}} \otimes I_{2}}\right) = \left( U^{\hat{\pi}} \circ \overline{U^{\hat{\pi}}}\right) \otimes \left( I_{2} \circ \overline{I_{2}}\right) = Q^{\hat{\pi}} \otimes I_{2}\]
Now suppose that $\pi \in \bij(\mathbb{Z}_2^5)$ is such that $D^\pi$ is zero everywhere $Q^{\pi'}$ is. Let $G$ be the copy of $\mathbb{Z}_2^4$ in $\mathbb{Z}_2^5$ corresponding to the upper left block of $D^\pi$, and fix $a \in \mathbb{Z}_2^5 \setminus G$. Consider playing the $G$-bijection game as follows. Upon input $x \in G$, act as if you are playing the $\mathbb{Z}_2^5$-bijection game according to the correlation whose correlation matrix is $D^\pi$ to obtain an output $y \in \mathbb{Z}_2^5$. If $y \in G$, then respond with $y$. If $y \notin G$, then respond with $y+a \in G$. It is straightforward to check that this is a classical $G$-invariant correlation whose correlation matrix is the upper left block of $D^\pi$, which is of course zero every $Q^{\hat{\pi}}$ is zero, a contradiction. Thus there is no classical $\mathbb{Z}_2^5$-invariant correlation whose correlation matrix is zero everywhere $Q^{\pi'}$ is zero, and of course this can be extended to $\mathbb{Z}_2^d$ for any larger $d$.

Finally, we recall that since $Q^{\hat{\pi}}$ is the correlation matrix of a strongly nonlocal correlation, there is a corresponding nonlocal game that it wins perfectly, but which has no perfect classical strategy. The game is played as follows: the players Alice and Bob are sent elements $x,y \in \mathbb{Z}_2^4$ and respond with $a,b \in \mathbb{Z}_2^4$ respectively. They win if $Q^{\hat{\pi}}_{a+b,x+y} \ne 0$. 
Unfortunately, this game is not necessarily an isomorphism game for some pair of graphs, and in fact it is not for any of the matrices $Q^{\hat{\pi}}$ we found using the above described procedure.

\subsection{$S_3$-invariant correlations}\label{subsec:S3comps}

For the symmetric group $S_3$ we can construct a unitary $V$ satisfying Equation~\eqref{eq:conjugate} as follows. Let $\one \in \mathbb{C}^3$ be the all ones vector. Define $v,u \in \mathbb{C}^3$ as
\[v = \begin{pmatrix}1 \\ \omega \\ \omega^2 \end{pmatrix}, \quad u = \begin{pmatrix}1 \\ \omega^2 \\ \omega \end{pmatrix}\]
where $\omega$ is a primitive cube root of unity. Let $\hat{V}$ be the matrix with columns
\[\one \oplus \one, \quad \one \oplus -\one, \quad v \oplus v, \quad v \oplus -v, \quad u \oplus u, \quad -u \oplus u.\]
Then $V = \frac{1}{\sqrt{6}}\hat{V}$ is a unitary satisfying Equation~\eqref{eq:conjugate}. The rows of $V$ are indexed by the elements of $S_3$, and the order we have chosen for these elements is $e, (123), (321), (12), (23), (13)$.

The group $S_3$ has two 1-dimensional representations: the trivial one and the representation that maps every permutation to its sign. Additionally, it has one irreducible 2-dimensional representation which can be viewed as the restriction of the \emph{natural representation} (as $3 \times 3$ permutation matrices) to the orthogonal complement of the constant vector. For our choice of $V$, the matrices $X^i$ associated to the two 1-dimensional representations of $S_3$ are scalar multiples of the identity matrix. The $X^i$ corresponding to the single 2-dimensional representation of $S_3$ is equal to
\[\frac{1}{\sqrt{6}}\begin{pmatrix}1 & 1 \\ 1 & -1\end{pmatrix}.\]
Let $H$ be the $2 \times 2$ Hadamard matrix (which is $\sqrt{3}$ times the matrix displayed above), and let $X = \begin{pmatrix}0 & 1 \\ 1 & 0\end{pmatrix}$ be the Pauli $X$ matrix. Then for any $2 \times 2$ unitary $N$, we obtain two $S_3$-transformation matrices.
\begin{equation}\label{eq:S3UG}
    V\left(I_2 \oplus \left(H\overline{N}H \otimes N\right)\right) V^\dagger \quad \text{and} \quad V\left(X \oplus \left(H\overline{N}H \otimes N\right)\right)V^\dagger.
\end{equation}

The question then is whether this produces any non-classical correlations? It turns out that it does. Since there will be uncountably many transformation matrices in this case, we cannot simply compute all of them as in the abelian case. Instead, we will pick a finite subgroup of the transformation matrices and compute all of the corresponding correlation matrices. Specifically, we will let the unitary $N$ from~\eqref{eq:S3UG} vary over the icosahedral group (realized as a subgroup of the $2 \times 2$ unitaries). This group has order 120, but the elements come in pairs of the form $N, -N$, which produce the same transformation matrices. This gives 120 correlations, 60 for each choice of either $I_2$ or $X$ on the lefthand side of the direct sum in~\eqref{eq:S3UG}. Both the transformation matrices $V\left(I_2 \oplus \left(H\overline{N}H \otimes N\right)\right) V^\dagger$ and $V\left(X \oplus \left(H\overline{N}H \otimes N\right)\right)V^\dagger$ result in classical correlations for $N$ being the identity matrix or the matrix $\begin{pmatrix}0 & 1 \\ -1 & 0\end{pmatrix}$ (or their negations). If $N$ is chosen to be any of the remaining diagonal matrices in the icosahedral group, then $V\left(I_2 \oplus \left(H\overline{N}H \otimes N\right)\right) V^\dagger$ results in classical correlations but $V\left(X \oplus \left(H\overline{N}H \otimes N\right)\right)V^\dagger$ produces non-classical ones, and if $N$ is chosen to be any of the remaining matrices with zero diagonal then $V\left(X \oplus \left(H\overline{N}H \otimes N\right)\right)V^\dagger$ results in classical correlations but $V\left(I_2 \oplus \left(H\overline{N}H \otimes N\right)\right) V^\dagger$ produces non-classical ones. For all other choices for $N$ being an element of the icosahedral group, both transformation matrices result in non-classical correlations. 
One also obtains non-classical correlations when choosing $N$ to be the $2 \times 2$ Hadamard matrix. In this case, the two matrices of Equation~\eqref{eq:S3UG} are equal to
\[
\begin{pmatrix}
1 & 0 & 0 & 0 & 0 & 0 \\
0 & \frac{1}{2} & \frac{1}{2} & \frac{i}{\sqrt{3}} & \frac{-i}{2\sqrt{3}} & \frac{-i}{2\sqrt{3}} \\
0 & \frac{1}{2} & \frac{1}{2} & \frac{-i}{\sqrt{3}} & \frac{i}{2\sqrt{3}} & \frac{i}{2\sqrt{3}} \\
0 & \frac{-i}{\sqrt{3}} & \frac{i}{\sqrt{3}} & \frac{1}{3} & \frac{1}{3} & \frac{1}{3} \\
0 & \frac{i}{2\sqrt{3}} & \frac{-i}{2\sqrt{3}} & \frac{1}{3} & \frac{-1}{6} & \frac{5}{6} \\
0 & \frac{i}{2\sqrt{3}} & \frac{-i}{2\sqrt{3}} & \frac{1}{3} & \frac{5}{6} & \frac{-1}{6}
\end{pmatrix}
\quad \text{and} \quad
\begin{pmatrix}
1 & 0 & 0 & 0 & 0 & 0 \\
0 & \frac{1}{2} & \frac{1}{2} & \frac{i}{\sqrt{3}} & \frac{-i}{2\sqrt{3}} & \frac{-i}{2\sqrt{3}} \\
0 & \frac{1}{2} & \frac{1}{2} & \frac{-i}{\sqrt{3}} & \frac{i}{2\sqrt{3}} & \frac{i}{2\sqrt{3}} \\
0 & \frac{-i}{\sqrt{3}} & \frac{i}{\sqrt{3}} & \frac{-1}{3} & \frac{-1}{3} & \frac{-1}{3} \\
0 & \frac{i}{2\sqrt{3}} & \frac{-i}{2\sqrt{3}} & \frac{-1}{3} & \frac{-5}{6} & \frac{1}{6} \\
0 & \frac{i}{2\sqrt{3}} & \frac{-i}{2\sqrt{3}} & \frac{-1}{3} & \frac{1}{6} & \frac{-5}{6}
\end{pmatrix}
\]
respectively. These produce the following correlation matrices:
\begin{equation}\label{eq:Hcharmat}
\begin{pmatrix}
1 & 0 & 0 & 0 & 0 & 0 \\
0 & \frac{1}{4} & \frac{1}{4} & \frac{1}{3} & \frac{1}{12} & \frac{1}{12} \\
0 & \frac{1}{4} & \frac{1}{4} & \frac{1}{3} & \frac{1}{12} & \frac{1}{12} \\
0 & \frac{1}{3} & \frac{1}{3} & \frac{1}{9} & \frac{1}{9} & \frac{1}{9} \\
0 & \frac{1}{12} & \frac{1}{12} & \frac{1}{9} & \frac{1}{36} & \frac{25}{36} \\
0 & \frac{1}{12} & \frac{1}{12} & \frac{1}{9} & \frac{25}{36} & \frac{1}{36}
\end{pmatrix}
\quad \text{and} \quad
\begin{pmatrix}
1 & 0 & 0 & 0 & 0 & 0 \\
0 & \frac{1}{4} & \frac{1}{4} & \frac{1}{3} & \frac{1}{12} & \frac{1}{12} \\
0 & \frac{1}{4} & \frac{1}{4} & \frac{1}{3} & \frac{1}{12} & \frac{1}{12} \\
0 & \frac{1}{3} & \frac{1}{3} & \frac{1}{9} & \frac{1}{9} & \frac{1}{9} \\
0 & \frac{1}{12} & \frac{1}{12} & \frac{1}{9} & \frac{25}{36} & \frac{1}{36} \\
0 & \frac{1}{12} & \frac{1}{12} & \frac{1}{9} & \frac{1}{36} & \frac{25}{36}
\end{pmatrix}
\end{equation}

We can compute a hyperplane that separates these correlations from the classical set of $S_3$-invariant correlation matrices. This is given in the form of a $6 \times 6$ matrix below:
\[
M = \begin{pmatrix}
1 & 0 & 0 & 0 & 0 & 0 \\
0 & 0 & 0 & -1 & 1 & 1 \\
0 & 0 & 0 & -1 & 1 & 1 \\
0 & -1 & -1 & 1 & 0 & 0 \\
0 & 1 & 1 & 0 & -1 & -1 \\
0 & 1 & 1 & 0 & -1 & -1
\end{pmatrix}
\]

Given any classical $S_3$-invariant correlation $p$ with correlation matrix $D^p$, the inner product $\Tr(M^\dagger D^p)$ is at least zero. However, both characteristic matrices presented in~\eqref{eq:Hcharmat} have inner product $-1$ with $M$. This is the largest separation we were able to find, where we restricted our hyperplane coefficients to lie between $-1$ and $1$. 

\section{Discussion}

We have introduced the notion of group invariant quantum Latin squares and provided a complete characterization in terms of unitary isomorphisms of the corresponding group algebras. Moreover, assuming we know unitaries that block diagonalize the group algebras, we can use the results of Section~\ref{sec:constructing} to actually construct group invariant quantum Latin squares. However, there are several natural questions that this work leaves unanswered. One such question is mentioned at the end of Section~\ref{subsec:cayley}, and we rephrase it here:
\begin{prob}
    Find a group invariant quantum Latin square that provides a quantum isomorphism of non-isomorphic Cayley graphs.
\end{prob}
\noindent We believe that the answer is yes, but it is likely that this is a rare phenomenon and thus random approaches may not be useful.

The other major open question 
is the following:
\begin{prob}
    Characterize when the correlation produced by a group invariant quantum Latin square is non-classical.
\end{prob}
\noindent Even an answer for the special case of abelian groups would be of great interest. In this case the $(G,G')$-invariant quantum Latin squares are in bijection with the bijections between $\widehat{G}$ and $\widehat{G'}$, and it is unclear what properties of these bijections could possibly influence whether the resulting correlation is non-classical.

There are a few possibilities for generalizing the notion of group invariant quantum Latin squares. The most obvious is to only require that the squared modulus of the inner products have the group invariance property, i.e., that the resulting correlation is group invariant. But the characterization of such quantum Latin squares may be difficult. However, in a similar direction one could take the quasi-regular representations used in Section~\ref{sec:reptheory} and extend this notion to allow for \emph{projective representations}. 
\begin{prob}
    Which results of the present work have analogs in the more general setting of projective representations?
\end{prob}
\noindent The resulting quantum Latin squares would no longer be group invariant, but only because inner products that should be equal differ by a unit complex factor. Thus the correlations would still be group invariant. It seems likely that analogs of many of the results from this work would also hold in this setting, but with the representation theory replaced with projective representation theory.

\section*{Acknowledgments}
Both authors were supported by the Carlsberg Foundation Young Researcher Fellowship CF21-0682 -- ``Quantum Graph Theory". Additionally, the first author is supported by CNPq, National Council for Scientific and Technological Development (Brazil).

\bibliographystyle{plainurl}
\bibliography{GIQLS}

\appendix

\section{Proof of Lemma~\ref{lem:hatAautos}}\label{app:proof}

\begin{proof}[Proof of Lemma~\ref{lem:hatAautos}]
    If $U$ has the form given in the lemma statement, it is easy to see that $U\hat{\A}U^\dagger = \hat{\A}$. Conversely, suppose that $U\hat{\A}U^\dagger = \hat{\A}$, and write $U$ as a block matrix whose blocks correspond to the summands in the expression for $\hat{\A}$, i.e.,
    \[U = \begin{pmatrix} 
    U_{11} & \dots & U_{1r} \\
    \vdots & \ddots & \vdots \\
    U_{r1} & \dots & U_{rr} 
    \end{pmatrix}.\]
    An arbitrary element $A \in \hat{\A}$ can be written as
    \begin{equation}\label{eq:Aform}
        A = \bigoplus_{i=1}^r I_{d_i} \otimes A_i
    \end{equation}
    for some $A_i \in M_{d_i}(\mathbb{C})$ for each $i \in [r]$. The $s,t$-block of $UAU^\dagger$ is then 
    \begin{equation}\label{eq:UAUform}
        \left(UAU^\dagger\right)_{s,t} = \sum_{i \in [r]} U_{si} \left(I_{d_i} \otimes A_i \right) U^\dagger_{ti}.
    \end{equation}
    
    For any $j \in [r]$ and $\ell,k \in [d_j]$, define $E^j_{\ell,k} \in \hat{\A}$ to be the matrix of the form given in~\eqref{eq:Aform} such that $A_i = 0$ if $i \ne j$ and $A_j = e_{\ell}e_k^\dagger$. Clearly $E^i_{\ell,k}$ has rank $d_i$ and
    \begin{equation}\label{eq:Eformula}
    E^i_{\ell,k}E^j_{s,t} = \delta_{i,j}\delta_{k,s}E^i_{\ell,t}.
    \end{equation}
    Now choose $\hat{i} \in [r]$ such that $d_{\hat{i}}$ is minimum. Note that then $E^{\hat{i}}_{\ell,k}$ has rank $d_{\hat{i}}$ which is the minimum nonzero rank of an element of $\hat{\A}$. Since $UE^{\hat{i}}_{\ell,k}U^\dagger \in \hat{\A}$, we have that
    \[B^{\hat{i},\ell,k} := UE^{\hat{i}}_{\ell,k}U^\dagger = \bigoplus_{i=1}^r I_{d_i} \otimes B^{\hat{i},\ell,k}_i\]
    for some matrices $B^{\hat{i},\ell,k}_i$ for $i \in [r]$. Since $UE^{\hat{i}}_{\ell,k}U^\dagger$ must have the same rank as $E^{\hat{i}}_{\ell,k}$, it follows that there is a unique $\pi(\hat{i},\ell,k) \in [r]$ such that $B^{\hat{i},\ell,k}_{\pi(\hat{i},\ell,k)}$ is nonzero, and moreover $d_{\pi(\hat{i},\ell,k)}$ must be minimum. We aim to show that $\pi(\hat{i},\ell,k)$ depends only on $\hat{i}$. Since $E^{\hat{i}}_{\ell,k}E^{\hat{i}}_{k,s} = E^{\hat{i}}_{\ell,s}$, we have that $B^{\hat{i},\ell,k}B^{\hat{i},k,s} = B^{\hat{i},\ell,s} \ne 0$. It follows that $\pi(\hat{i},\ell,k) = \pi(\hat{i},k,s)$ for all $\ell,k,s \in [r]$, and from this it is immediate that $\pi(\hat{i},\ell,k)$ depends only on $\hat{i}$, so we denote it as $\pi(\hat{i})$.

    If we let $\hat{\A}_i = \spn\{E^i_{\ell,k} : \ell,k \in [d_i]\}$, then by the above and a dimension argument we must have that $U\hat{\A}_{\hat{i}}U^\dagger = \hat{\A}_{\pi(\hat{i})}$ and in particular, for any two distinct $i,j \in [r]$ with $d_i = d_j$ minimum, we have that $\pi(i) \ne \pi(j)$, i.e., $\pi$ gives a permutation of the indices $i \in [r]$ with $d_i$ minimum. Considering next the indices $i \in [r]$ with $d_i$ equal to its second smallest value, we conclude that $\pi$ also gives a permutation of these indices, and continuing on we see that $\pi \in \mathrm{Sym}_{\hat{\A}}$.

    Applying~\eqref{eq:UAUform} to $E^i_{\ell,k}$, we have that
    \begin{equation}\label{eq:UEUblocks}
        \left(UE^i_{\ell,k}U^\dagger\right)_{s,t} = U_{si}\left(I_{d_i} \otimes e_\ell e_k^\dagger \right) U^\dagger_{ti}
    \end{equation}
    and by the above this is zero unless $s = t = \pi(i)$. If we let $I^i = \sum_{\ell \in [d_i]} E^i_{\ell,\ell}$, then $I^i$ is the identity in the subalgebra $\hat{\A}_i$. From the above it is then easy to see that $UI^i U^\dagger = I^{\pi(i)}$ and therefore
    \[U_{\pi(i),i} (I_{d_i} \otimes I_{d_i}) U^\dagger_{\pi(i),i} = I_{d_i} \otimes I_{d_i}.\]
    Since $U_{\pi(i),i}$ is square, it is therefore unitary. Since the expression in~\eqref{eq:UEUblocks} is zero if $t \ne \pi(i)$, by taking $s=\pi(i) \ne t$ and summing the expression over all $\ell = k \in [d_i]$ we obtain that $U_{\pi(i),i}U^\dagger_{t,i} = 0$ and therefore $U^\dagger_{t,i} = 0$ for $t \ne \pi(i)$.

    Now, if we let $U_i = U_{\pi(i),i}$, then the above shows that $U = \widehat{P}^\pi \hat{U}$ where
    \[\hat{U} = \bigoplus_{i=1}^r U_i.\]
    Furthermore, since $(\widehat{P}^\pi)^\dagger \hat{\A} \widehat{P}^\pi = \hat{\A}$, we have that $\hat{U}\hat{\A}\hat{U}^\dagger = \hat{\A}$. Therefore,
    \[U_i\left(I_{d_i} \otimes M_{d_i}(\mathbb{C})\right)U_i^\dagger = I_{d_i} \otimes M_{d_i}(\mathbb{C})\]
    for all $i \in [r]$.

    Now suppose that $V(I_d \otimes M_{d}(\mathbb{C})V^\dagger = I_d \otimes M_{d}(\mathbb{C})$ for some unitary $V$. We will show that $V = M \otimes N$ for some unitaries $M$ and $N$, which will complete the proof of the lemma. 
    
    Since $I \otimes e_ie_j^\dagger$ has rank $d$ which is the minimum nonzero rank, we have that there is some rank one matrix $B_{ij}$ such that
    \[V\left(I \otimes e_ie_j^\dagger\right) V^\dagger = I \otimes B_{ij}.\]
    Since $B_{ij}$ has rank one, it can be written as $u_{ij}v_{ij}^\dagger$ for some \emph{nonzero} vectors $u_{ij}, v_{ij}$. By definition of the $B_{ij}$, we must have that $B_{ij}B^\dagger_{\ell j} = B_{i \ell} \ne 0$ and thus
    \[(v^\dagger_{ij}v_{\ell j}) u_{ij}u_{\ell j}^\dagger = u_{i\ell}v_{i \ell}^\dagger \ne 0.\]
    It follows that $u_{ij}$ is proportional to $u_{i\ell}$ for all $i, \ell, j$. Thus let $w_i$ be a unit vector that is proportional to every $u_{ij}$. Then the above equation also implies that $v_{i\ell}$ is proportional to $w_\ell$ for all $i,\ell$. Therefore, there are scalars $\alpha_{ij}$ such that $B_{ij} = \alpha_{ij}w_iw_j^\dagger$ for all $i,j$. Lastly, since $B_{i 1}B_{j 1}^\dagger = B_{ij}$, we have that
    \[\alpha_{i1} \overline{\alpha_{j1}} w_iw_j^\dagger = B_{ij}\]
    for all $i,j$. Thus letting $x_i = \alpha_{i1} w_i$ satisfies $B_{ij} = x_ix_j^\dagger$.

    Now write $V$ in block form as 
    \[V = \begin{pmatrix} 
    V_{11} & \dots & V_{1d} \\
    \vdots & \ddots & \vdots \\
    V_{d1} & \dots & V_{dd} 
    \end{pmatrix}\]
    where each block is $d \times d$. Then considering the $s,t$-block of $V(I \otimes e_ie_j^\dagger)V^\dagger$, we see that
    \begin{equation}\label{eq:xixj}
        \sum_{\ell} \left(V_{s,\ell} e_i\right)\left(V_{t,\ell} e_j\right)^\dagger = \delta_{st}x_ix_j^\dagger.
    \end{equation}
    In the case where $s=t$ and $i=j$, all terms in the above sum are positive semidefinite and thus it follows that $V_{s,\ell}e_i = \beta_{s,\ell,i} x_i$ for some scalar $\beta_{s,\ell,i}$. We aim to show that $\beta_{s,\ell,i}$ does not depend on $i$. Using Equation~\eqref{eq:xixj}, we have that
    \[\sum_{\ell} \beta_{s,\ell,i}\overline{\beta_{t,\ell,j}} = \delta_{s,t}.\]
    If we let $\beta^{s,i}$ be the vector with entries $\beta_{s,\ell,i}$, then the above says that each $\beta^{s,i}$ is a unit vector and that $\langle \beta^{s,i},\beta^{t,j}\rangle = \delta_{s,t}$, which in particular implies that $\beta^{s,i} = \beta^{s,j}$ for all $i,j$. Therefore we have shown that $\beta_{s,\ell,i}$ depends only on $s$ and $\ell$ and thus
    \[V_{s,\ell}e_i = \beta_{s,\ell}x_i\]
    for some scalars $\beta_{s,\ell}$. It follows that $\beta_{t,k}V_{s,\ell} e_i = \beta_{s,\ell}V_{t,k}e_i$ for all $i$ and therefore all the $V_{s,\ell}$ are proportional to each other. Therefore $V = M \otimes N$ for $M,N \in M_d(\mathbb{C})$. Since $V$ is unitary, we have that $MM^\dagger \otimes NN^\dagger = I$ and therefore both $M$ and $N$ are proportional to unitaries and thus by rescaling them appropriately we will have $V = M \otimes N$ where both $M$ and $N$ are unitaries, as desired. This (finally) completes the proof.
\end{proof}

\end{document}